\documentclass[reqno]{amsart}
\usepackage{amsmath,amssymb}
\usepackage[mathscr]{euscript}
\usepackage{tikz}

\usepackage{xcolor}

\definecolor{bblue}{rgb}{.2,0.2,.8}

\theoremstyle{plain}
\newtheorem{theorem}{Theorem}[section]
\newtheorem{proposition}[theorem]{Proposition}
\newtheorem{lemma}[theorem]{Lemma}
\newtheorem{corollary}[theorem]{Corollary}

\newtheorem{remark}[theorem]{Remark}

\numberwithin{equation}{section}
\numberwithin{theorem}{section}

\newcommand{\mc}[1]{{\mathcal #1}}
\newcommand{\mb}[1]{{\mathbf #1}}
\newcommand{\mf}[1]{{\mathfrak #1}}

\newcommand{\bb}[1]{{\mathbb #1}}
\newcommand{\ms}[1]{{\mathscr #1}}

\renewcommand{\epsilon}{\varepsilon}

\renewcommand{\Cap}{{\rm cap}}

\newcommand{\<}{\langle}
\renewcommand{\>}{\rangle}

\title[Metastable $\Gamma$-expansion of large deviations rate
functions]{Metastable $\Gamma$-expansion of finite state Markov chains level
two large deviations rate functions.}

\author{L. Bertini}
\address{Lorenzo Bertini \hfill\break \indent
   Dipartimento di Matematica, Universit\`a di Roma `La Sapienza'
   \hfill\break \indent   P.le Aldo Moro 2, 00185 Roma, Italy}
 \email{bertini@mat.uniroma1.it}

\author{D. Gabrielli}
\address{\noindent Davide Gabrielli \hfill\break\indent 
 DISIM, Universit\`a dell'Aquila
\hfill\break\indent 
67100 Coppito, L'Aquila, Italy
}
\email{davide.gabrielli@univaq.it}

\author{C. Landim}
\address{Claudio Landim
  \hfill\break\indent IMPA \hfill\break\indent Estrada Dona Castorina
  110, \hfill\break\indent
J. Botanico, 22460 Rio de Janeiro, Brazil\hfill\break\indent
  {\normalfont and} \hfill\break\indent CNRS UMR 6085, Universit\'e de
  Rouen, \hfill\break\indent Avenue de l'Universit\'e, BP.12,
  Technop\^ole du Madril\-let, \hfill\break\indent
F76801 Saint-\'Etienne-du-Rouvray, France.} 
\email{landim@impa.br}

\begin{document}

\begin{abstract}
We examine two analytical characterisation of the metastable behavior
of a Markov chain. The first one expressed in terms of its transition
probabilities, and the second one in terms of its large deviations
rate functional.

Consider a sequence of continuous-time Markov chains
$(X^{(n)}_t:t\ge 0)$ evolving on a fixed finite state space $V$. Under
a hypothesis on the jump rates, we prove the existence of times-scales
$\theta^{(p)}_n$ and probability measures with disjoint supports
$\pi^{(p)}_j$, $j\in S_p$, $1\le p \le \mf q$, such that (a)
$\theta^{(1)}_n \to \infty$,
$\theta^{(k+1)}_n/\theta^{(k)}_n \to \infty$, (b) for all $p$,
$x\in V$, $t>0$, starting from $x$, the distribution of
$X^{(n)}_{t \theta^{(p)}_n}$ converges, as $n\to\infty$, to a convex
combination of the probability measures $\pi^{(p)}_j$. The weights of
the convex combination naturally depend on $x$ and $t$.

Let $\ms I_n$ be the level two large deviations rate functional for
$X^{(n)}_t$, as $t\to\infty$.  Under the same hypothesis on the jump
rates and assuming, furthermore, that the process is reversible, we
prove that $\ms I_n$ can be written as
$\ms I_n = \ms I^{(0)} \,+\, \sum_{1\le p\le \mf q} (1/\theta^{(p)}_n) \,
\ms I^{(p)}$ for some rate functionals $\ms I^{(p)}$ which take finite
values only at convex combinations of the measures $\pi^{(p)}_j$:
$\ms I^{(p)}(\mu) < \infty$ if, and only if, $\mu = \sum_{j\in S_p}
\omega_j\, \pi^{(p)}_j$ for some probability measure $\omega$ in $S_p$.
\end{abstract}

\noindent



\maketitle
\thispagestyle{empty}

\section{Introduction}
\label{sec0}

The metastable behavior of continuous-time Markov chains has attracted
some interest in recent years. We refer to the monographs \cite{ov05,
bh15, lrev, lms2} for the latest developments. In this article, we
propose to investigate the Markov chains metastable behaviour from an
analytical perpective, by showing that the Markov chains semigroup and
large deviations rate function encode the metastable properties of the
process. The main results explain how to extract from these
functionals the metastable time-scales, states and wells.

To tackle this problem we consider a sequence of continuous-time
Markov chains $(X^{(n)}_t :t\ge 0)$ evolving on a finite state space
$V$. Under a natural hypothesis on the jump rates of these chains,
stated in equation \eqref{mh} below, we prove the existence of
\begin{itemize}
\item[(a)] time-scales $\theta^{(1)}_n, \dots, \theta^{(\mf q)}_n$
such that, as $n\to\infty$,  $\theta^{(1)}_n \to \infty$,
$\theta^{(p+1)}_n/\theta^{(p)}_n \to \infty$ for $1\le p < \mf q$;
\item[(b)] and metastable states $\pi^{(p)}_1, \dots , \pi^{(p)}_{\mf n_p}$,
$1\le p\le \mf q$.
\end{itemize}
The parameter $p$ is called the level and indicates the depth of the
wells or, equivalently, the time-scale at which a metastable behaviour
is observed.  The metastable states $\pi^{(p)}_j$ are probability
measures on $V$. It will be shown that, for each fixed level $p$, the
support of the measures $\pi^{(p)}_1, \dots , \pi^{(p)}_{\mf n_p}$ are
disjoint. They represent the wells among which the process $X^{(n)}_t$
evolves in the time-scale $\theta^{(p)}_n$. The number of metastable
set decreases as the time-scales increase: $\mf n_{p+1} < \mf n_p$. A
metastable state at level $p+1$ is a convex combination of metastable
states at level $p$: for each $1\le p< \mf q$ and
$1\le m\le \mf n_{p+1}$,
$\pi^{(p+1)}_{m} = \sum_j \omega^{(m)}_j \pi^{(p)}_{j}$ for some
probability measure $\omega^{(m)}$ on $\{1, \dots, \mf n_p\}$.

The first main result of this article, Theorem \ref{mt0}.(b), states
that for all $t>0$, $x\in V$, the distribution of
$X^{(n)}_{t \theta^{(p)}_n}$ starting from $x$ converges to a convex
combination of the measures $\pi^{(p)}_j$, $1\le j\le \mf n_p$.  More
precisely, denote by $p^{(n)}_t(x,y)$ the transition probabilities of
the Markov chain $X^{(n)}_t$. Then, for each $1\le p\le \mf q$, $t>0$,
$x \in V$, there exists a probability measure
$\omega^{(p)}_{t,x}(\,\cdot\,)$ on $\{1, \dots, \mf n_p\}$ such that
\begin{equation}
\label{i0}
\lim_{n\to\infty} p^{(n)}_{t \theta^{(p)}_n} (x,\,\cdot\,) \;=\;
\sum_{j=1}^{\mf n_p} \omega^{(p)}_{t,x}(j)\; \pi^{(p)}_j (\,\cdot\,)\;.
\end{equation}
The weights $\omega^{(p)}_{t,x}(j)$ of this convex combination
naturally depend on $x$ and $t$, and are obtained by a recursion
procedure.

Theorem \ref{mt0} also characerises the asymptotic behavior of the
transition probabilities at all intermediate time-scales
$\beta_n$. Fix $0\le p \le \mf q$, set $\theta^{(0)}_n =1$,
$\theta^{(\mf q+1)}_n =+\infty$, and consider a sequence $\beta_n$
such that $\beta_n/\theta^{(p)}_n \to \infty$,
$\beta_n/\theta^{(p+1)}_n \to 0$.  Theorem \ref{mt0} provides a formula
for the limit of $p^{(n)}_{\beta_n} (x,\,\cdot\,)$ as $n\to\infty$.
It corresponds to the limit obtained in \eqref{i0} by letting
$t\to\infty$ after $n\to\infty$.

Freidlin and Koralov \cite{fk17}, after \cite{bl4} and \cite{lx16},
examined sequences of Markov chains on finite state spaces under the
same hypothesis \eqref{mh} assumed below and taken from \cite{bl4,
lx16}.  Their main results describes the asymptotic behavior of the
transition probabilities at the intermediate time-scales $\beta_n$
introduced above. These results demonstrate the interest of the theory
developed in \cite{bl2, BL15, llm18, or19, lms2}, which permits to
investigate the asymptotic behavior of the Markov chain exactly at
the metastable time-scale, and not just before or after it.

We turn to the large deviations. Denote by $\ms I_n$ the level two
large deviations rate functional of the Markov chain $X^{(n)}_t$, as
$t\to\infty$ \cite{var}. Under the hypothesis of reversibility, the
second main result of this article provides a $\Gamma$-expansion of
the functional $\ms I_n$ as
\begin{equation}
\label{i2}
\ms I_n \;=\;  \ms I^{(0)} \;+\; \sum_{p=1}^{\mf q} \frac{1}{\theta^{(p)}_n}
\, \ms I^{(p)}\;.
\end{equation}
This expansion has to be understood in the sense that $\ms I_n$,
$\theta^{(p)}_n\, \ms I_n$, $1\le p\le \mf q$, $\Gamma$-converge to
$\ms I^{(0)}$, $\ms I^{(p)}$, respectively.  The rate functionals
$\ms I^{(p)}$ take finite values only at convex combinations of the
metastable states $\pi^{(p)}_j$: $\ms I^{(p)}(\mu) < \infty$ if, and
only if, $\mu = \sum_{j\in S_p} \omega_j\, \pi^{(p)}_j$ for some
probability measure $\omega$ in $S_p$.

Therefore, both the semigroup and the level two large deviations rate
functionals encode all characteristics of the metastable behaviour of
a Markov chain. They provide the time-scales, the metastable states
and the wells. In particular, it becomes a natural problem to prove
such an expansion in other contexts.

We believe that the inductive approach presented here provides a
general method to derive these results, as well as the metastable
behavior in the classical sense \cite{bl2}, of Markov chains with
wells of different depths, even if the state space is not fixed, as
assumed here. To be applied, one needs (a) to show that the process
quickly reaches one of the wells (the initial step of the induction
procedure) and (b) to compute the capacities \eqref{26b} and the
asymptotic jump rates \eqref{34}.

More precisely, inspecting the proof of Theorem \ref{mt3} reveals that
it essentially relies on the convergences of the generator of the
trace process on the wells (more exactly on the convergence of the
average rates $r^{(p)}_n(i,j)$ introduced in \eqref{41} below). Since
this convergence has been obtained in many different contexts, by
following the strategy proposed here it should be possible to derive
the metastable $\Gamma$-expansion of the large deviations level two
rate function for dynamics in which the state space is not fixed.

This includes random walks in potential fields \cite{lmt2015, ls2018},
condensing zero-range models \cite{bl3, l2014, s2018}, inclusion
processes \cite{GRV-13, CCG-14, bdg17, kim21, ks21}, or statistical
mechanical models in which the volume grows as the temperature
decreases. For example, the Curie-Weiss model in random environment
\cite{begk01, bbi09}, the Blume-Capel model \cite{llm19}, the Potts
model \cite{ls2016, ks21c, lee22}, or the Kawasaki dynamics for the
Ising model \cite{GL15}.

In particular, it should be possible to apply this approach to
non-reversible diffusions in potential fields,\cite{s95, begk02,
lms17, ls19, rs18, ls22, ls22b}, extending Di Ges\`u and Mariani
\cite{GM17}, who prove the $\Gamma$-expansion in the reversible case
in which there is only one well at each different depth.


\section{The model}
\label{sec1}

Let $\color{blue} G=(V,E)$ be a finite directed graph, where $V$
represents the finite set of vertices, and $E$ the set of directed
edges. Denote by $\color{blue} (X^{(n)}_t : t\ge 0)$, $n\ge 1$, a
sequence of $V$-valued, irreducible continuous-time Markov chains,
whose jump rates are represented by $\color{blue} R_n(x,y)$. We assume
that $R_n(x,y)>0$ for all $(x,y)\in E$ and $n\ge 1$. The generator
reads as
\begin{equation*}
{ \color{blue} (\ms L_n f)(x)}
\;=\; \sum_{y\colon (x,y)\in E} R_n(x,y)\, \{\,
f(y)\,-\, f(x)\,\}\;.
\end{equation*}
Denote by $\color{blue} \lambda_n(x)$, $x\in V$, the holding rates of
the Markov chain $X^{(n)}_t$ and by $\color{blue} p_n(x,y)$, $x$,
$y\in V$, the jump probabilities, so that
$R_n(x,y) = \lambda_n(x) \, p_n(x,y)$.

Let $\color{blue} \pi_n$ stand for the unique stationary state. The
so-called Matrix tree Theorem \cite[Lemma 6.3.1]{FW98} provides a
representation of the measure $\pi_n$ in terms of arborescences of the
graph $(V,E)$.

Denote by $\color{blue} D(\bb R_+, W)$, $W$ a finite set, the space
of right-continuous functions $\mf x: \bb R_+ \to W$ with left-limits
endowed with the Skorohod topology and the associated Borel
$\sigma$-algebra. Let $\color{blue} \mb P_{\! x}=\mb P^n_{\! x}$,
$x\in V$, be the probability measure on the path space $D(\bb R_+, V)$
induced by the Markov chain $X^{(n)}_t$ starting from $x$. Expectation
with respect to $\mb P_{\! x}$ is represented by $\color{blue} \mb E_x$.

Denote by $p^{(n)}_t(x,y)$ the transition probability of
the Markov chain $X^{(n)}_t$:
\begin{equation*}
{ \color{blue}  p^{(n)}_t(x,y)}
\;:=\; \mb P^n_{\! x} \big[\, X_t \,=\, y\,\big]\;,
\quad x\,,\, y \in V\;,\;\; t\,>\,0\;.
\end{equation*}
Since the chain is irreducible and $\pi_n$ its stationary state, by
the ergodic theorem for finite state-spaces Markov chains,
\begin{equation*}
\lim_{t\to\infty} p^{(n)}_t(x,y) \;=\; \pi_n(y)\;\;
\text{for all $x$, $y\in V$}\;.
\end{equation*}

\subsection*{Longer time-scales}

Assume that $\lim_n R_n(x,y)$ exists for all $(x,y)\in E$, and denote
by $\bb R_{0} (x,y) \in [0,\infty)$ its limit:
\begin{equation}
\label{01}
{\color{blue} \bb R_{0} (x,y)}
\;:=\; \lim_n R_n(x,y)\;, \quad (x,y) \,\in\, E \;. 
\end{equation}
Let $\bb E_0$ be the set of edges whose asymptotic rate is positive:
$\color{blue} \bb E_0 \,:=\, \{\, (x,y)\in E : \bb R_{0}
(x,y)>0\,\}$, and assume that $\bb E_0 \not = \varnothing$.  The jump
rates $\bb R_0 (x,y)$ induce a continuous-time Markov chain on $V$,
denoted by $\color{blue} (\bb X_t:t\ge 0)$, which, of course, may be
reducible. Denote by $\color{blue} \bb L^{(0)}$ its generator.

Denote by $\color{blue} \ms V_{1}, \dots, \ms V_{\mf n}$,
$\mf n\ge 1$, the closed irreducible classes of $\bb X_t$, and let
\begin{equation}
\label{05}
{\color{blue} S} \;: =\; \{1, \dots, \mf n\}\;, \quad 
{\color{blue} \ms V}  \;: =\; \bigcup_{j\in S} \ms V_{j} \;, \quad
{\color{blue} \Delta } \;: =\;  V  \, \setminus \, \ms V \;. 
\end{equation}
The set $\Delta$ may be empty and some of the sets $\ms V_j$ may be
singletons.

Let $\color{blue} \bb Q_x$ be the probability measure on
$D(\bb R_+, V)$ induced by the Markov chain $\bb X_t$ starting from
$x$. 

For two sequences of positive real numbers $(\alpha_n : n\ge 1)$,
$(\beta_n : n\ge 1)$, $\color{blue} \alpha_n \prec \beta_n$ or
$\color{blue} \beta_n \succ \alpha_n$ means that
$\lim_{n\to\infty} \alpha_n/\beta_n = 0$. Similarly,
$\color{blue} \alpha_n \preceq \beta_n$ or
$\color{blue} \beta_n \succeq \alpha_n$ indicates that either
$\alpha_n \prec \beta_n$ or $\alpha_n/\beta_n$ converges to a positive
real number $a\in (0,\infty)$.

Let
\begin{equation*}
\gamma_n \;:=\;  \max_{(x,y)\in E \setminus \bb E_0} R_n(x,y) 
\end{equation*}
so that $\gamma_n \prec 1$. Choose a sequence $\beta_n$ such that
$1\prec \beta_n \prec \gamma_n^{-1}$. Couple
$X^{(n)}_t$ and $\bb X_t$ making them jump as much as possible
together. Denote by $\widehat{\mb P}_{\! x } $ the coupling
measure. Since $\beta_n \prec \gamma_n^{-1}$, for all $x\in V$
\begin{equation*}
\lim_{n\to \infty}
\widehat{\mb P}_{\! x } \, \big[ \, \mb X^{(n)}_t \,=\, \bb X_t \;, 0\le
t\le \beta_n\,\big] \;=\; 1\;.
\end{equation*}
In particular, for all $x$, $y\in V$
\begin{equation*}
\lim_{n\to\infty} p^{(n)}_{\beta_n} (x,y) \;=\;
\lim_{n\to\infty} \widehat{\mb P}_{\! x } \,
\big[ \, \bb X_{\beta_n} \,=\, y \,\big] \;=\;
\sum_{j\in S} \mf a^{(0)} (x,j)\, \pi^\sharp_j (y)\;,
\end{equation*}
where $\color{blue} \pi^\sharp_j$, $1\le j\le \mf n$, represents the
stationary states of the Markov chain $\bb X$ restricted to $\ms V_j$
and $\mf a^{(0)} (x,j)$ the probability that the chain $\bb X_t$
starting from $x$ is absorbed by the closed recurrent class $\ms V_j$:
\begin{equation}
\label{92}
\mf a^{(0)} (x,j) \;:=\; \lim_{t\to\infty} \bb Q_{\! x } \,
\big[ \, \bb X_{t} \,\in\, \ms V_j \,\big] \;.
\end{equation}

In the first part of this article, we investigate the asymptotic
behaviour of $p^{(n)}_{\beta_n} (x,y)$ in different time-scales
$\beta_n$. The definition of the time-scales and the description of
the asymptotic behaviour is based on a construction of a tree
\cite{bl4, lx16} presented after the statement of the main hypothesis
of the article.

\subsection*{The main assumption}

Two sequences of positive real numbers $(\alpha_n : n\ge 1)$,
$(\beta_n : n\ge 1)$ are said to be \emph{comparable} if
$\alpha_n \prec \beta_n$, $\beta_n \prec \alpha_n$ or
$\alpha_n / \beta_n \to a \in (0,\infty)$. This condition excludes the
possibility that the sequence $\alpha_n /\beta_n$ oscillates between
two finite values and does not converge.

A set of sequences
$(\alpha^{\mf u}_n: n\ge 1)$, $\mf u \in \mf R$, of positive real
numbers, indexed by some finite set $\mf R$, is said to be comparable
if for all $\mf u$, $\mf v \in\mf R$ the sequence
$(\alpha^{\mf u}_n : n \ge 1)$, $(\alpha^{\mf v}_n : n \ge 1)$ are
comparable.

Let $\bb Z_+ = \{0, 1, 2, \dots \}$, and let $\color{blue} \Sigma_m$,
$m\ge 1$, be the set of functions $k: E \to \bb Z_+$ such that
$\sum_{(x,y)\in E} k(x,y) = m$.  We assume throughout this article
that for every $m\ge 1$ the set of sequences
\begin{equation}
\label{mh}
\big(\, \prod_{(x,y)\in E} R_n(x,y)^{k(x,y)} : n \ge 1 \,\big)
\;,\quad k\in\Sigma_m \;,
\end{equation}
is comparable.

\begin{remark}
\label{rm3}
This hypothesis on the jump rates is taken from \cite{bl4} and
\cite{lx16}. It also appears in \cite{fk17}, what supports the
assertions that this condition is natural in the context of
metastability.

As observed in \cite{bl4}, assumption \eqref{mh} is fulfilled by all
statistical mechanics models which evolve on a fixed state space and
whose metastable behaviour has been derived.  This includes the Ising
model \cite{ns91, ns92, bc96, bm02}, the Potts model with or without a
small external field \cite{nz19, ks21b}, the Blume-Capel model
\cite{co96, ll16}, and conservative Kawasaki dynamics \cite{bhn06,
GHNO09, HNT12, bl15b}.
\end{remark}

\subsection*{A rooted tree}

In this subsection, we present the construction, proposed in
\cite{bl4, lx16}, of a rooted tree which describes all different
metastable behaviours of the Markov chain $X^{(n)}_t$. This
construction plays a fundamental role in the statement of the main
theorems of this article. The reader will find at the end of this
section a simple example which may help to understand the
construction.

The tree satisfies the following conditions:

\begin{itemize}
\item[(a)] Each vertex of the tree represents a subset of $V$;
\item[(b)] Each generation forms a partition of $V$;
\item[(c)] The children of each vertex form a partition of the parent.
\item[(d)] The generation $p+1$ is strictly coarser than the
generation $p$.
\end{itemize}

The tree is constructed by induction starting from the leaves to the
root.  It corresponds to a deterministic coalescence process. Denote
by $\color{blue} \mf q$ the number of steps in the recursive
construction of the tree. At each level $1\le p\le \mf q$, the
procedure generates a partition
$\{ \ms V^{(p)}_1, \dots, \ms V^{(p)}_{\mf n_p}, \Delta_p\}$, a
time-scale $\theta^{(p)}_n$ and a $\{1, \dots, \mf n_p\}$-valued
continuous-time Markov chains $\bb X^{(p)}_t$ which describes the
evolution of the chain $X^{(n)}_{t \theta^{(p)}_n}$ among the subsets
$\ms V^{(p)}_1, \dots, \ms V^{(p)}_{\mf n_p}$, called hereafter
\emph{wells}.

The leaves are the sets $\ms V_1, \dots, \ms V_{\mf n}, \Delta$
introduced in \eqref{05}. We proceed by induction. Let
$\color{blue} S_1 = S$, $\color{blue} \mf n_1 = \mf n$,
$\color{blue} \ms V^{(1)}_j = \ms V_j$, $j\in S_1$,
$\color{blue} \Delta_1 = \Delta$, and assume that the recursion has
produced the sets
$\ms V^{(p)}_1, \dots, \ms V^{(p)}_{\mf n_p}, \Delta_p$ for some
$p\ge 1$, which forms a partition of $V$.

Denote by $H_{\ms A}$, $H^+_{\ms A}$, ${\ms A}\subset V$, the hitting
and return time of ${\ms A}$:
\begin{equation} 
\label{201}
{\color{blue} H_{\ms A} } \;: =\;
\inf \big \{t>0 : X^{(n)}_t \in {\ms A} \big\}\;,
\quad
{\color{blue} H^+_{\ms A}} \;: =\;
\inf \big \{t>\tau_1 : X^{(n)}_t \in {\ms A} \big\}\; ,  
\end{equation}
where $\tau_1$ represents the time of the first jump of the chain
$X^{(n)}_t$:
$\color{blue} \tau_1 = \inf\{t>0 : X^{(n)}_t \not = X^{(n)}_0\}$.

For two non-empty, disjoint subsets $\ms A$, $\ms B$ of $V$, denote by
$\Cap_n(\ms A, \ms B)$ the capacity between $\ms A$ and $\ms B$:
\begin{equation}
\label{202}
{\color{blue} \Cap_n(\ms A, \ms B)}
\;:=\; \sum_{x\in \ms A} \pi_n(x)\, \lambda_n(x) 
\, \mb P^n_{\! x} \big[\, H_{\ms B} < H^+_{\ms A}\, \big]\;.
\end{equation}
Set $\color{blue} S_p = \{1, \dots, \mf n_p\}$, and let
$\theta^{(p)}_n$ be defined by
\begin{equation}
\label{26b}
{\color{blue} \frac 1{\theta^{(p)}_n}} \;:=\;
\sum_{i\in S_p}  \frac{\Cap_n (\ms V^{(p)}_i, \breve{\ms V}^{(p)}_i)}
{\pi_n(\ms V^{(p)}_i)}  \;, \quad
\text{where}\;\;
{\color{blue} \breve{\ms V}^{(p)}_i}  \,:=\, \bigcup_{j\in S_p \setminus\{ i\}} \ms
V^{(p)}_j\;. 
\end{equation}
The ratio
$\pi_n(\ms V^{(p)}_i) /\Cap_n (\ms V^{(p)}_i, \breve{\ms V}^{(p)}_i)$
represents the time it takes for the chain $X^{(n)}_t$, starting from
a point in $\ms V^{(p)}_i$ to reach the set $\breve{\ms
V}^{(p)}_i$. Therefore, $\theta^{(p)}_n$ corresponds to the smallest
time needed to observe such a jump.

Recall from \eqref{100} the definition of the trace of a Markov
chain. Denote by $\color{blue} \{Y^{n,p}_t: t\ge 0\}$ the trace of
$\{X^{(n)}_t: t\ge 0\} $ on $\ms V^{(p)}$, and by
$\color{blue} R^{(p)}_n : \ms V^{(p)} \times \ms V^{(p)} \to \bb R_+$
its jump rates. By equation (2.5) in \cite{lrev},
\begin{equation}
\label{40}
R^{(p)}_n (x,y) \;=\; \lambda_n(x) \; \mb P^n_{\! x} \big[ H_y
= H^+_{\ms V^{(p)}} \big]\;, \quad x\,,\; y \in \ms V^{(p)} \,, 
\; x\not = y  \;.
\end{equation}

Denote by $r^{(p)}_n(i,j)$ the mean rate at which the trace process
jumps from $\ms V^{(p)}_i$ to $\ms V^{(p)}_j$:
\begin{equation}
\label{41}
{\color{blue}  r^{(p)}_n(i,j)}
\; := \; \frac{1}{\pi_n(\ms V^{(p)}_i)}
\sum_{x \in\ms V^{(p)}_i} \pi_n(x) 
\sum_{y\in\ms V^{(p)}_j} R^{(p)}_n(x,y) \;.
\end{equation}
Under the assumption \eqref{mh}, \cite{lx16} proved that the sequences
$\theta^{(p)}_n \, r^{(p)}_n(i,j)$ converge for all
$i\not = j\in S_p$.  Denote the limits by $r^{(p)} (i,j)$:
\begin{equation}
\label{34}
{\color{blue}  r^{(p)} (i,j)} \;:=\;  \lim_{n\to\infty} \theta^{(p)}_n
\, r^{(p)}_n(i, j) \;\in\; \bb R_+\;.
\end{equation}

Denote by $\color{blue} (\bb X^{(p)}_t : t\ge 0)$ the $S_p$-valued
continuous-time Markov chain induced by the jump rates $r^{(p)}(j,k)$,
and by $\color{blue} \bb L^{(p)}$ its generator.  Let
$\Phi_p : \ms V^{(p)}\to S_p$ be the projection which sends the points
in $\ms V^{(p)}_j$ to $j$:
\begin{equation*}
{\color{blue} \Phi_p} \;:=\; \sum_{k\in S_p} k\; \chi_{_{\ms V^{(p)}_k}}\;.
\end{equation*}
In this formula and below, $\color{blue} \chi_{_{\ms A}}$ stands for
the indicator function of the set $\ms A$.

Next theorem is the main result in \cite{lx16}.

\begin{theorem}
\label{t1}
Assume that condition \eqref{mh} is in force. Then, for each
$1\le p\le \mf q$, $j\in S_p$, $x\in \ms V^{(p)}_j$, under the measure
$\mb P^n_{\! x}$, the sequence of $S_p$-valued, hidden Markov
processes $\Phi_p (X^{(n)}_{t\theta^{(p)}_n})$ converges weakly in the
Skorohod topology to $\bb X^{(p)}_t$. Moreover, the time spent in
$\Delta_p$ is negligible in the sense that for all $t>0$,
\begin{equation*}
\lim_{n\to\infty} \max_{y\in \ms V^{(p)}}\,
\mb E^n_{\! x} \Big[\, \int_0^t \chi_{_{\Delta_p}}
(X^{(n)}_{s\theta^{(p)}_n})\; ds \,\Big] \;=\; 0\;.
\end{equation*}
\end{theorem}

The process $\bb X^{(p)}_t$ describes therefore how the chain
$X^{(n)}_{t}$ evolves among the wells $\ms V^{(p)}_j$ in the
time-scale $\theta^{(p)}_n$.  Let $p^{(p)}_t(i,j)$ be the transition
probabilities:
\begin{equation}
\label{61}
{\color{blue}  p^{(p)}_t(i,j)} \;=\; \bb Q^{(p)}_i [\, X_t
= j\,]\;,
\quad t\,\ge \, 0\;, i\,,\, j \in S_p\;,
\end{equation}
where $\color{blue} \bb Q^{(p)}_i$ stands for the probability measure
on the path space $D(\bb R_+, S_p)$ induced by the Markov chain $\bb X^{(p)}_t$
starting from $i$.

By \cite[Theorem 2.7]{lx16}, there exists $j$, $k\in S$ such that
$r^{(p)} (j,k) >0$. Actually, by the proof of this result
\begin{equation}
\label{36}
\text{$\displaystyle \sum_{k\not = j} r^{(p)} (j,k) >0$
for all $j\in S_p$ such that}\quad 
\lim_{n\to \infty} \theta^{(p)}_n \,
\frac{\Cap_n (\ms V^{(p)}_j, \breve{\ms V}^{(p)}_j)}
{\pi_n(\ms V^{(p)}_j)} >0 \;.
\end{equation}

Denote by
$\color{blue} \mf R^{(p)}_1, \dots, \mf R^{(p)}_{\mf n_{p+1}}$ the
recurrent classes of the $S_p$-valued chain $\bb X^{(p)}_t$, and by
$\color{blue} \mf T_p$ the transient states. Let
${\color{blue} \mf R^{(p)}} = \cup_j \mf R^{(p)}_j$, and observe that
$\{\mf R^{(p)}_1, \dots, \mf R^{(p)}_{\mf n_{p+1}}, \mf T_p\}$ forms a
partition of the set $S_p$. This partition of $S_p$ induces a new
partition of the set $V$. Let
\begin{equation*}
{\color{blue}  \ms V^{(p+1)}_m}
\;:= \; \bigcup_{j\in \mf R^{(p)}_m} \ms V^{(p)}_j\;, \quad
{\color{blue}  \ms T^{(p+1)}}
\;:= \; \bigcup_{j\in \mf T_p} \ms V^{(p)}_j\;, 
\quad m\in  {\color{blue} S_{p+1} \;:=\; \{1, \dots, \mf n_{p+1}\}} \;,
\end{equation*}
so that $V  \,=\,  \Delta_{p+1}   \, \cup \, \ms V^{(p+1)}$, where
\begin{equation}
\label{05b}
{\color{blue} \ms V^{(p+1)}}
\;=\; \bigcup_{m\in S_{p+1}} \ms V^{(p+1)}_{m}\;,
\quad
{\color{blue}  \Delta_{p+1}}
\;:= \; \Delta_p \,\cup\, \ms T^{(p+1)} \;.
\end{equation}

The subsets
$\ms V^{(p+1)}_1, \dots, \ms V^{(p+1)}_{\mf n_{p+1}}, \Delta_{p+1}$ of
$V$ are the result of the recursive procedure. We claim that
conditions (a)--(d) hold at step $p+1$ if they are fulfilled up to
step $p$ in the induction argument.

The sets
$\ms V^{(p+1)}_{1}, \dots, \ms V^{(p+1)}_{\mf n_{p+1}}$,
$\Delta_{p+1}$ constitute a partition of $V$ because the sets
$\mf R^{(p)}_1, \dots, \mf R^{(p)}_{\mf n_{p+1}}$, $\mf T_p$ form a
partition of $S_p$, and the sets
$\ms V^{(p)}_{1}, \dots, \ms V^{(p)}_{\mf n_{p}}$, $\Delta_{p}$ one of
$V$. Conditions (a)--(c) are therefore satisfied.

To show that the partition obtained at step $p+1$ is strictly coarser
than $\{\ms V^{(p)}_{1}, \dots,$
$\ms V^{(p)}_{\mf n_{p}}, \Delta_{p}\}$, observe that, by \eqref{36},
$r^{(p)}(j,k)>0$ for some $k\neq j \in S_p$. Hence, either $j$ is a
transient state for the process $\bb X^{(p)}_t$ or the closed
recurrent class which contains $j$ also contains $k$. In the first
case $\Delta_p \subsetneq \Delta_{p+1}$, and in the second one there
exists $m\in S_{p+1}$ such that
$ \ms V^{(p)}_{j} \cup \ms V^{(p)}_{k} \subset \ms V^{(p+1)}_{m}$.
Therefore, the new partition
$\{\ms V^{(p+1)}_{1}, \dots, \ms V^{(p+1)}_{\mf n_{p+1}},
\Delta_{p+1}\}$ of $V$ satisfies the conditions (d).

The construction terminates when the $S_p$-valued Markov chain
$\bb X^{(p)}_t$ has only one recurrent class so that $\mf
n_{p+1}=1$. In this situation, the partition at step $p+1$ is
$\ms V^{(p+1)}_{1}$, $\Delta_{p+1}$.

This completes the construction of the rooted tree.  Recall that we
denote by $\color{blue} \mf q$ the number of steps of the scheme.  As
claimed at the beginning of the procedure, for each
$1\le p\le \mf q$, we generated a time-scale $\theta^{(p)}_n$, a
partition
$\ms P_p = \{\ms V^{(p)}_{1}, \dots, \ms V^{(p)}_{\mf n_{p}},
\Delta_{p}\}$, where
$\ms P_1 = \{\ms V_{1}, \dots, \ms V_{\mf n}, \Delta \}$,
$\ms P_{\mf q+1} = \{\ms V^{(\mf q+1)}_{1}, \Delta_{\mf q+1} \}$, and
a $S_p$-valued continuous-time Markov chain $\bb X^{(p)}_t$.

Furthermore, by construction,
\begin{equation}
\label{59}
\Delta_p \;\subset\; \Delta_{p+1} \;, \quad 1\,\le\, p\,\le\, \mf q\;,
\end{equation}
by \cite[Assertion 8.B]{lx16},
\begin{equation}
\label{51}
\theta^{(p)}_n \;\prec\;  \theta^{(p+1)}_n \;, \quad
1\, \le\,  p \, <\, \mf q\;,
\end{equation}
and by \cite[Assertion 8.A]{lx16} or equation (8.2) of this article,
\begin{equation}
\label{58}
\lim_{n\to\infty} \frac{\pi_n(x)}{\pi_n(\ms V^{(p)}_{j})}
\;\;\text{exists and belongs to $(0,1]$}
\end{equation}
for all $1\le p\le \mf q+1$, $j\in S_p$, $x\in \ms V^{(p)}_{j}$.

The partitions $\ms P_1, \dots, \ms P_{\mf q+1}$ form a rooted tree
whose root ($0$-th generation) is $V$, first generation is
$\{\ms V^{(\mf q+1)}_{1}, \Delta_{\mf q+1} \}$ and last
($(\mf q+1)$-th) generation is
$\{\ms V_{1}, \dots, \ms V_{\mf n}, \Delta \}$.  Note that the set
$\ms V^{(p+1)}$ corresponds to the set of recurrent points for the
chain $\bb X^{(p)}_t$. In contrast, the points in $\Delta_{p+1}$ are
either transient for this chain or negligible in the sense that the
chain $X^{(n)}_t$ remains a negligible amount of time on the set
$\Delta_p$ in the time-scale $\theta^{(p)}_n$ (cf. \cite{bl4, lx16}).

\smallskip\noindent{\bf Example.}  We conclude this section with an
example to help the reader understanding the tree's construction. Let
$V= \{0, \dots, 29\}$, and consider the energy
$\bb H: V \to \{0, \dots, 4\}$ given in Figure \ref{fig1}. Note that
$\bb H (k+ 1) - \bb H (k) = \pm 1$ for $0\le k<29$. The energy $\bb H$
has $9$ local minima, represented in Figure 1 by $x_1, \dots x_9$.

Consider the $V$-valued continuous-time Markov chain $X^{(n)}_t$ whose
jump rates are given by $R_n(k,j)=0$ if $j\not = k \pm 1$ and
$R_n(k,k\pm 1) = \exp\{ - n\, [\bb H (k\pm 1) - \bb H (k)]^+ \}$,
where $a^+ = \max\{a,0\}$. Hence if $\bb H (k\pm 1) - \bb H (k) = -1$
the chain jumps from $k$ to $k\pm 1$ at rate $1$, while if
$\bb H (k\pm 1) - \bb H (k) = +1$ it jumps from $k$ to $k\pm 1$ at
rate $e^{-n}$. More simply, observing the energy landscape presented
in Figure \ref{fig1}, the chain jumps ``downwards'' at rate $1$ and
jumps ``upwards'' at rate $e^{-n}$.

It is easy to check that the stationary state, denoted by $\pi_n$, is
given by $\pi_n(k) = (1/Z_n) \exp\{-n\, \bb H(k)\}$, where $Z_n$ is a
normalising constant, and that $\pi_n$ satisfies the detailed balance
conditions. In particular, and since the downward jump rates are equal
to $1$, $c_n(j,k) := \pi_n(j) \, R_n(j,k) = \pi_n(j) \wedge
\pi_n(k)$. It follows from this identity and Lemma \ref{l23} below
that the capacities introduced in \eqref{202} are easy to estimate in
this example.

\begin{figure}[h]
\centering
\begin{tikzpicture}[scale = .4]
\draw[very thick] (0,0) -- (1,-1) -- (2,0) -- (3,-1) --
(4,0) -- (5,-1) -- (6,-2) -- (7,-1) -- (8,0) -- (9,-1) -- (10,-2) --
(11,-3) -- (12,-2) -- (13,-1) -- (14,-2) -- (15,-1) -- (16, -2) --
(17, -1) -- (18,0) -- (19,1) -- (20, 0) -- (21, -1) -- (22, 0) -- (23,
-1) -- (24, 0) -- (25, 1) -- (26, 0) -- (27,-1) -- (28, -2) -- (29,-3);
\draw[thick] (4.75,-1) -- (5.25,-1);
\draw[thick] (6.75,-1) -- (7.25,-1);
\draw[thick] (8.75,-1) -- (9.25,-1);
\draw[thick] (9.75,-2) -- (10.25,-2);
\draw[thick] (11.75,-2) -- (12.25,-2);
\draw[thick] (16.75,-1) -- (17.25,-1);
\draw[thick] (17.75,0) -- (18.25,-0);
\draw[thick] (19.75,0) -- (20.25,-0);
\draw[thick] (23.75,0) -- (24.25,-0);
\draw[thick] (25.75,0) -- (26.25,-0);
\draw[thick] (26.75,-1) -- (27.25,-1);
\draw[thick] (27.75,-2) -- (28.25,-2);
\foreach \x in {0, ..., 29}
\fill (\x,.-4) circle [radius = .2cm];
\draw (1,-5) node {$x_1$};
\draw (3,-5) node {$x_2$};
\draw (6,-5) node {$x_3$};
\draw (11,-5) node {$x_4$};
\draw (14,-5) node {$x_5$};
\draw (16,-5) node {$x_6$};
\draw (21,-5) node {$x_7$};
\draw (23,-5) node {$x_8$};
\draw (29,-5) node {$x_9$};
\end{tikzpicture}
\caption{The energy landscape of the Markov chain $X^{(n)}_t$.}
\label{fig1}
\end{figure}
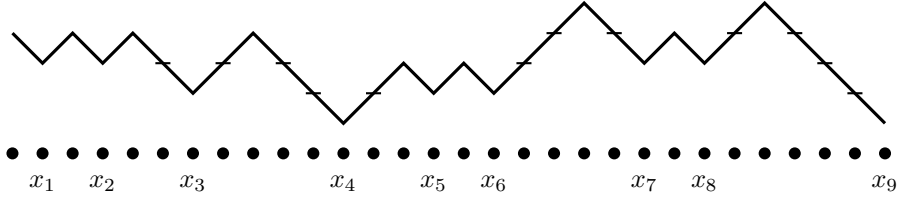

Consider the tree construction presented at the beginning of this
section.

\smallskip\noindent{\bf Step 1: the leaves.} In the first step we
determine the leaves of the tree, which correspond to the closed
irreducible classes of the chain $\bb X_t$. In this example, the
closed irreducible classes are the local minima of the energy $\bb H$
so that $\mf n=9$, $\ms V_j = \{x_j\}$, $1\le j\le 9$,
$\Delta = V \setminus \{x_1, \dots, x_9\}$, and the leaves are the
sets $\Delta$ and $\ms V_j$, $1\le j\le 9$.

Denote by $\mf q+1$, $\mf q\ge 0$, to total number of generations of
the tree. The exact value of $\mf q+1$ will only be known at the end
of the construction.

\smallskip\noindent{\bf Step 2: the generation $\mf q$.} The second
step consists in determining the smallest transition time between a
well $\ms V_j$ to a well $\ms V_k$. This is the smallest time-scale it
takes for the process $X^{(n)}_t$ starting from $\ms V_j$ to hit
$\ms V_k$. In the above example this time-scale is
$\theta^{(1)}_n = e^n$. In this time scale, the trace of $X^{(n)}_t$
on $\ms V = \cup_j \ms V_j$ evolves as a Markov chain and converges,
as $n\to \infty$, to a $\ms V$-valued Markov chain, represented by
$\bb X^{(1)}_t$. The states $x_1$ and $x_2$ are transient states for
$\bb X^{(1)}_t$ and absorbed at the recurrent state $x_3$. Similarly,
the states $x_5$ and $x_6$ are transient states for $\bb X^{(1)}_t$
and are absorbed by $x_4$. The states $x_7$, $x_8$ form a closed
irreducible class of $\bb X^{(1)}_t$, as well as the point $x_9$.

Therefore, $\mf T_1 =\{1, 2, 5, 6\}$, $\mf R^{(1)}_1= \{3\}$,
$\mf R^{(1)}_2= \{4\}$, $\mf R^{(1)}_3= \{7,8\}$,
$\mf R^{(1)}_4= \{9\}$, so that $\ms V^{(2)}_1 = \{x_3\}$,
$\ms V^{(2)}_2 = \{x_4\}$, $\ms V^{(2)}_3= \{x_7,x_8\}$,
$\ms V^{(2)}_4= \{x_9\}$, $\ms T_2 = \{x_1, x_2, x_5, x_6\}$.
Moreover, the generation $\mf q$ of the tree has $5$ elements:
$\Delta_2 = \Delta \cup \ms T_2$, and $\ms V^{(2)}_j$, $1\le j\le 4$.

\smallskip\noindent{\bf Step 3: the generation $\mf q-1$.}  At this
point, we need to determine the smallest transition time between the
wells $\ms V^{(2)}_1$, $\ms V^{(2)}_2$, $\ms V^{(2)}_3$ and
$\ms V^{(2)}_4$. In this example the smallest transition time is
$\theta^{(2)}_n = e^{2n}$.

Let $\ms V^{(2)} = \cup_{1\le j\le 4} \ms V^{(2)}_j$, and denote by
$Y^{n,2}_t$ the trace of the process $X^{(n)}_t$ on $\ms V^{(2)}$.
Consider the projection $\Phi_2: \ms V^{(2)} \to S_2 = \{1, 2, 3, 4\}$
which sends the points in $\ms V^{(2)}_j$ to $j$. Note that $\Phi_2$
is not a bijection. In consequence the process $\Phi_2 (Y^{n,2}_t)$ is
not a Markov chain. It is however possible to prove (cf. \cite{bl2})
that the process $\Phi_2 (Y^{n,2}_{t\theta^{(2)}_n})$ converges to a
$S_2$-valued Markov chain, denoted by $\bb X^{(2)}_t$.

The states $1$ and $3$, which corresponds to the sets $\ms V^{(2)}_1$
and $\ms V^{(2)}_3$, respectively, are transient for the chain
$\bb X^{(2)}_t$, while the states $2$ and $4$, which corresponds to
the sets $\ms V^{(2)}_2$ and $\ms V^{(2)}_4$, respectively, form
closed irreducible classes. The state $1$ is absorbed at $2$, while
the state $3$ may be absorbed at $2$ or $4$.

Thus, in this example, $\mf T_2 =\{1, 3\}$, $\mf R^{(2)}_1= \{2\}$,
$\mf R^{(2)}_2= \{4\}$, so that $\ms V^{(3)}_1 = \{x_4\}$,
$\ms V^{(3)}_2 = \{x_9\}$, $\ms T_3 = \{x_3, x_7, x_8\}$. The
generation $\mf q-1$ of the tree has $3$ elements:
$\Delta_3 = \Delta_2 \cup \ms T_3$, and $\ms V^{(3)}_j$, $j=1,2$.

\smallskip\noindent{\bf Step 4: the generation $\mf q-2$.} We need now
to determine the smallest transition time between the wells
$\ms V^{(3)}_1$ and $\ms V^{(3)}_2$. In this example it is
$\theta^{(3)}_n = e^{3n}$.

Let $\ms V^{(3)} = \ms V^{(3)}_1 \cup \ms V^{(3)}_2$, and denote by
$Y^{n,3}_t$ the trace of the process $X^{(n)}_t$ on $\ms V^{(3)}$. It
is however possible to prove (cf. \cite{bl2}) that the process
$Y^{n,3}_{t\theta^{(3)}_n}$ converges to a $\{1,2\}$-valued Markov
chain, denoted by $\bb X^{(3)}_t$.

The states $\{1,2\}$ form a irreducible class for $\bb X^{(3)}_t$.
Hence $\mf T_3$ is empty and $\mf R^{(3)}= \mf R^{(3)}_1= \{1,2\}$, so
that $\ms V^{(4)}_1 = \{x_4, x_9\}$, $\ms T_4 = \varnothing$. The
generation $\mf q-2$ of the tree has $2$ elements:
$\Delta_4 = \Delta_3$, and $\ms V^{(4)}_1$.

As there is only one closed irreducible class, the construction is
completed and the value of $\mf q$ is revealed. The partition
$\{\Delta_4 \,,\, \ms V^{(4)}_1\}$ of $V$ corresponds to the first
generation. Since, by construction, it is also the $(\mf q-2)$-th
generation, we deduce that $\mf q=3$ and that the tree has $\mf q+1=4$
generations. To get a rooted tree, we declare that the root, which
corresponds to the zeroth generation, is the set
$V = \Delta_4 \cup \ms V^{(4)}_1$. The tree associated to the example
presented in Figure \ref{fig1} is depicted in Figure \ref{fig2}.

\begin{figure}[h]
\centering
\begin{tikzpicture}[scale = .4]
\draw (0,0) node {$\{x_4\}$};
\draw (3,0) node {$\{x_9\}$};
\draw (6,0) node {$\{x_3\}$};
\draw (9,0) node {$\{x_7\}$};
\draw (12,0) node {$\{x_8\}$};
\draw (15,0) node {$\{x_1\}$};
\draw (18,0) node {$\{x_2\}$};
\draw (21,0) node {$\{x_5\}$};
\draw (24,0) node {$\{x_6\}$};
\draw (27,0) node {$\Delta_1$};
\draw (0,-4) node {$\{x_4\}$};
\draw (3,-4) node {$\{x_9\}$};
\draw (6,-4) node {$\{x_3\}$};
\draw (10.5,-4) node {$\{x_7, x_8\}$};
\draw (21,-4) node {$\Delta_2$};
\draw[very thick] (0,-.6) -- (0,-3.4);
\draw[very thick] (3,-.6) -- (3,-3.4);
\draw[very thick] (6,-.6) -- (6,-3.4);
\draw[very thick] (9,-.6) -- (10.5,-3.4);
\draw[very thick] (12,-.6) -- (10.5,-3.4);
\draw[very thick] (15,-.6) -- (21,-3.4);
\draw[very thick] (18,-.6) -- (21,-3.4);
\draw[very thick] (21,-.6) -- (21,-3.4);
\draw[very thick] (24,-.6) -- (21,-3.4);
\draw[very thick] (27,-.6) -- (21,-3.4);
\draw (0,-8) node {$\{x_4\}$};
\draw (3,-8) node {$\{x_9\}$};
\draw (13.5,-8) node {$\Delta_3$};
\draw[very thick] (0,-7.4) -- (0,-4.6);
\draw[very thick] (3,- 7.4) -- (3,-4.6);
\draw[very thick] (13.5,-7.4) -- (6,-4.6);
\draw[very thick] (13.5,-7.4) -- (10.5,-4.6);
\draw[very thick] (13.5,-7.4) -- (21,-4.6);
\draw (1.5,-12) node {$\{x_4, x_9\}$};
\draw (13.5,-12) node {$\Delta_4$};
\draw[very thick] (0,-8.6) -- (1.5,-11.4);
\draw[very thick] (3,- 8.6) -- (1.5,-11.4);
\draw[very thick] (13.5,-8.6) -- (13.5,-11.4);
\draw (7.5,-16) node {$V$};
\draw[very thick] (1.5,-12.6) -- (7.5,-15.4);
\draw[very thick] (13.5,-12.6) -- (7.5,-15.4);
\end{tikzpicture}
\caption{The tree or coalescence process generated by the Markov chain
$X^{(n)}_t$.}
\label{fig2}
\end{figure}
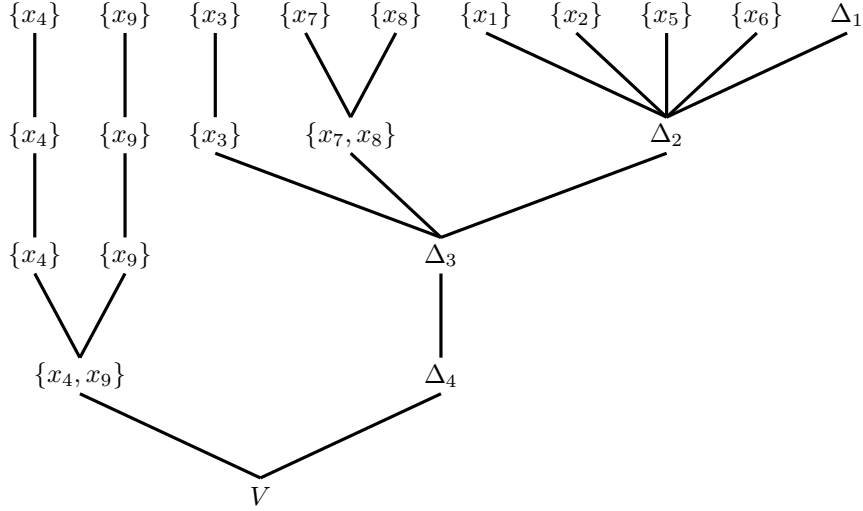

\section{The main results}
\label{sec2}

In this section, we enunciate the main results of the article.  The
statements require a further layer in the tree construction presented
in the previous section. At each step $1\le p\le \mf q +1$, we
introduce a set of probability measures $\pi^{(p)}_j$, $j\in S_p$, on
$V$. The construction of these measures is carried out below by
induction. In Proposition \ref{p3}, however, we characterise the
measure $\pi^{(p)}_j$ as the limit of the stationary state $\pi_n$
conditioned to $\ms V^{(p)}_j$. In particular,
\begin{equation}
\label{52}
\text{ the support of $\pi^{(p)}_j$ is the set $\ms V^{(p)}_j$}\; .
\end{equation}
Moreover, in Theorem \ref{mt0}.(b) we show that for all $t>0$, $x\in
V$, the distribution of $X^{(n)}_{t \theta^{(p)}_n}$ starting from $x$
converges to a convex combination of the measures $\pi^{(p)}_j$, $j\in
S_p$. The weights of this convex combination depend on $x$ and
$t$. This result asserts, therefore, that the measures $\pi^{(p)}_j$
are the metastable states of the process $X^{(n)}_t$ observed on the
time-scale $\theta^{(p)}_n$.

We proceed by induction. Let $\pi^{(1)}_j$, $j\in S_{1}$, be the
probability measure on $\ms V^{(1)}_j$ given by
$\color{blue} \pi^{(1)}_j = \pi^\sharp_j$, where, recall,
$\pi^\sharp_j$ represents the stationary states of the Markov chain
$\bb X_t$ restricted to the closed irreducible set
$\ms V^{(1)}_j = \ms V_j$. Clearly, condition \eqref{52} is fulfilled.

Fix $1\le p\le \mf q$, and assume that the probability measures
$\pi^{(p)}_j$, $j\in S_p$, has been defined and satisfy condition
\eqref{52}. Denote by $\color{blue} M^{(p)}_m(\cdot)$,
$m\in S_{p+1}$, the stationary state of the Markov chain
$\bb X^{(p)}_t$ restricted to $\mf R^{(p)}_m$. The measure
$ M^{(p)}_m$ is understood as a measure on $S_p=\{1, \dots, \mf n_p\}$
which vanishes on the complement of $\mf R^{(p)}_m$.  Let
$\pi^{(p+1)}_m$ be the probability measure on $\ms V^{(p)}_m$ given by
\begin{equation}
\label{80}
{\color{blue} \pi^{(p+1)}_m (x)}
\;:=\; \sum_{j\in \mf R^{(p)}_m} M^{(p)}_m(j)\, \pi^{(p)}_j (x)\;, \quad
x\in V\;.
\end{equation}

Clearly, condition \eqref{52} is in force. Moreover, $\pi^{(p+1)}_m$
is a convex combination of the measures $\pi^{(p)}_j$,
$j\in \mf R^{(p)}_m$. A fortiori, for each $1\le p\le \mf q +1$,
$m\in S_p$, $\pi^{(p)}_m$ is a convex combination of the measures
$\pi^\sharp_j$, $j\in S$.

We further add absorption probabilities at each step. Let
$\mf a^{(0)} (x, j)$, $x\in V$, $j\in S_1$, be the probability that the
Markov chain $\bb X_t$ starting from $x$ is absorbed at the closed
irreducible set $\mc V^{(1)}_j$:
\begin{equation}
\label{33}
{\color{blue} \mf a^{(0)} (x,j)}
\;:=\;  \lim_{t\to\infty}
\bb Q_x\big[\, \bb X_t \, \in\, \ms V_j \,\big]\;.
\end{equation}
Note that $a^{(0)} (x,\,\cdot\,)$ is a probability measure on $S_1$
for each $x\in V$.

Fix $1\le p\le \mf q$ and assume that $\mf a^{(p-1)} (x, j)$ has been
defined. Let $\mf A^{(p)} (j, m)$, $j\in S_p$, $m\in S_{p+1}$, be the
probability that the chain $\bb X^{(p)}_t$ starting from $j$ has been
absorbed at the closed irreducible set $\mf R^{(p)}_m$:
\begin{equation}
\label{63}
{\color{blue}  \mf A^{(p)} (j,m)} \;:=\; \lim_{t\to\infty}
\sum_{k\in \mf R^{(p)}_m} p^{(p)}_t(j,k)\;, \quad j\in S_p\;,\;\; m\in S_{p+1} \;.
\end{equation}
For $x\in V$, $m\in S_{p+1}$, let
\begin{equation}
\label{33b}
{\color{blue} \mf a^{(p)} (x,m)}
\;:=\; \sum_{j\in S_p} \mf a^{(p-1)} (x, j) \, \mf A^{(p)} (j, m) \;.
\end{equation}
Since $\mf A^{(p)} (j, \,\cdot\,)$ is a probability measure on
$S_{p+1}$, it is easy to show by induction that
$a^{(p)} (x,\,\cdot\,)$ is a probability measure on $S_{p+1}$ for each
$x\in V$, $1\le p\le \mf q$. 

Let $\color{blue} \theta^{(0)}_n =1$,
$\color{blue} \theta^{(\mf q +1)}_n =+\infty$ for all $n\ge 1$. The
first main result of the article reads as follows. It provides a
complete description of the ergodic behavior of the Markov chain
$X^{(n)}_t$. 

\begin{theorem}
\label{mt0}
Assume that condition \eqref{mh} is in force. Then,

\begin{itemize}
\item[(a)] For each $1\le p\le \mf q +1$, sequence $(\beta_n:n\ge 1)$
such that  $\theta^{(p-1)}_n \,\prec\, \beta_n\,\prec\,
\theta^{(p)}_n$, and $x \in V$,  
\begin{equation}
\label{27b}
\lim_{n\to\infty} p^{(n)}_{\beta_n} (x,\,\cdot\, ) \;=\;
{\color{blue} \Pi_{p-1}(x,\,\cdot\,) }\;:=\;
\sum_{j\in S_p}
\mf a^{(p-1)} (x,j)\, \pi^{(p)}_j (\,\cdot\, )\;.
\end{equation}

\item[(b)] For each $1\le p\le \mf q$, $t>0$, $x \in V$,
\begin{equation}
\label{60}
\lim_{n\to\infty} p^{(n)}_{t \theta^{(p)}_n} (x,\,\cdot\,) \;=\;
\sum_{j\in S_p} \omega^{(p)}_t(x,j)\, \pi^{(p)}_j (\,\cdot\,)\;,
\end{equation}
where
\begin{equation*}
\omega^{(p)}_t(x,j) \;=\;
\sum_{k\in S_{p}} \mf a^{(p-1)} (x,k) \; p^{(p)}_t(k,j) \;.
\end{equation*}

\item[(c)] For all $1\le p\le \mf q$, $j\in S_p$, $x\in V$, 
\begin{equation*}
\lim_{t\to0} \lim_{n\to\infty} p^{(n)}_{t \theta^{(p)}_n} (x,\,\cdot\,) 
\;=\; \sum_{j\in S_p}
\mf a^{(p-1)} (x,j)\, \pi^{(p)}_j (\,\cdot\, )
\end{equation*}

\item[(d)] For all $1\le p\le \mf q$, $1\le j\le \mf n_p$, $x\in V$, 
\begin{equation*}
\lim_{t\to\infty} \lim_{n\to\infty} p^{(n)}_{t \theta^{(p)}_n} (x,\,\cdot\,)
\;=\; \sum_{m\in S_{p+1}}
\mf a^{(p)}(x,m)\, \pi^{(p+1)}_m (\,\cdot\,)\;.
\end{equation*}
\end{itemize}
Moreover,
\begin{equation}
\label{70}
\lim_{n\to \infty} \pi_n(\Delta_{\mf q+1}) \;=\; 0\;, \quad
\lim_{n\to\infty} \pi_n(x)
\;\;\text{exists and belongs to $(0,1]$}
\end{equation}
for all $x\,\in\, \ms V^{(\mf q+1)}$.
\end{theorem}

Note that the right-hand side of (c) and (d) coincide with the one
obtained in (a). These assertion state that at the time-scale
$\theta^{(p)}_n$ a smooth transition between two different regimes is
observed.

Part (b) of this theorem states that, starting from $x$, the
distribution of the process at time $t\theta^{(p)}_n$ is close to a
convex combination of the measures $\pi^{(p)}_k$, $k\in S_p$. The
weight of the measure $\pi^{(p)}_k$ is given by the probability that
the process is initially attracted to a well $\ms V^{(p)}_j$ times the
probability that the dynamics among the wells drives the process from
the well $\ms V_j$ to the well $\ms V_k$ in the ``macroscopic'' time
intervall $[0,t]$.

The next result provides a formula for the measures $\pi^{(p)}_j$ and
for the absorbing probabilities $\mf a^{(p-1)} (x,j)$. Recall that for
each $x\in V$, $\mf a^{(p-1)} (x,\,\cdot\,)$ is a probability measure
on $S_p$.

\begin{proposition}
\label{p3}
Fix $1\le p\le \mf q+1$, $j\in S_p$. For all $z\in \ms V^{(p)}_j$, 
\begin{equation*}
\lim_{n\to\infty} \frac{\pi_n(z)}{\pi_n(\ms V^{(p)}_j)}
\;=\; \pi^{(p)}_j(z) \;.
\end{equation*}
If $x\in \ms V^{(p)}_j$, then
$\mf a^{(p-1)} (x,k)=\delta_{j,k}$, $k\in S_p$. On the other hand, if
$x\not\in \ms V^{(p)}$, then
\begin{equation*}
\mf a^{(p-1)} (x,j) \;=\; \lim_{n\to\infty} \mb P^n_x
\big[\, H_{\ms V^{(p)}_j} \,<\,  H_{\breve{\ms V}^{(p)}_j}\,\big]\;. 
\end{equation*}
\end{proposition}

\subsection*{Large deviations rate function expansion.}

We assume from now on that the dynamics is reversible: $\pi_n(x)
\, R_n(x,y) =  \pi_n(y)\, R_n(y,x)$ for all $(x,y)\in E$.
For a  probability measure $\nu$ on a finite space $W$ and two
functions $f$, $g:W \to \bb R$, let  
\begin{equation*}
{\color{blue} \< \, f \,,\, g \,\>_{\nu} } \;=\;
\sum_{x\in W} f(x)\, g(x)\, \nu(x)\;.
\end{equation*}

By \cite{var}, for each fixed $n\ge 1$, the occupation time
distribution of the chain $X^n_t$, defined
by 
\begin{equation*}
\frac{1}{t} \int_0^t \delta_{X^n_s}\; ds \;,
\end{equation*}
satisfies a large deviations principle as $t\to\infty$, the so-called
level 2 LDP.  In this formula, $\delta_x$, $x\in V$, represents the
Dirac measure concentrated at $x$, so that
$t^{-1} \int_0^t \delta_{X^n_s}\; ds$ is a random element of
$\color{blue}\ms P(V)$, the space of probability measures on $V$.
Denote by $\ms I_n: \ms P(V) \to [0,\infty]$ the level two large 
deviations rate function:
\begin{equation}
\label{f4}
{\color{blue} \ms I_n (\mu)} \;=\; -\, \inf_u \sum_{x\in V} \frac{(\ms L_n
u)(x)}{u(x)} \; \mu(x)\;,
\end{equation}
where the infimum is performed over all functions
$u: V \to (0,\infty)$.  Since we assumed reversibility and
$\pi_n(x)>0$ for all $x\in V$, for all measures $\mu \in \ms P(V)$, by
\cite[Theorem 5]{dv75},
\begin{equation}
\label{f6}
\ms I_n(\mu)
\;=\; \< \,
\sqrt{f_n} \,,\, (-\, \ms L_n) \sqrt{f_n} \,\>_{\pi_n}\;, 
\end{equation}
where $f_n(x) = \mu (x)/\pi_n(x)$.

The second main result of this article provides an expansion of the
rate function $\ms I_n$. Recall that we denote by $\bb L^{(0)}$ the
generator of the Markov chain $\bb X_t$ introduced right after
\eqref{01}. Let $\ms I^{(0)} : \ms P(V) \to \bb R_+$ be given by
\begin{equation}
\label{f11}
{\color{blue} \ms I^{(0)} (\mu)} \;=\; -\,
\inf_{u>0}   \, \sum_{x\in V} \mu(x) \frac{(\bb L^{(0)} u)(x)}{u(x)}\;,
\end{equation}
where the supremum is carried over all functions
$u: V \to (0,\infty)$. Theorem \ref{mt3} below states that the
sequence of rate functions $\ms I_n$ $\Gamma$-converges to
$\ms I^{(0)}$.  In \eqref{79}, we show that
$\ms I^{(0)} (\mu) \,=\, 0$ if and only if there exists a probability
measure $\omega$ on $S_1$ such that
\begin{equation}
\label{79b}
\mu \;=\; \sum_{j\in S_1} \omega_j\, \pi^{(1)}_j\;.
\end{equation}
For such measures $\mu$, it is natural to consider the limit $\beta_n
\ms I_{n} (\mu)$ for some sequence $\beta_n\to\infty$. 

Fix $1\le p\le \mf q$. Denote by $\color{blue} \ms P(S_p)$ the set of
probability measures on $S_p$. Let
$\ms I^{(p)} \colon \ms P(V) \to [0,+\infty]$ be the functional given
by
\begin{equation}
\label{83b}
{\color{blue} \ms I^{(p)} (\mu) } \, :=\,
\left\{
\begin{aligned}
& -\,  \inf_{\mb h} \,  \sum_{j\in S_p} \omega_j \,
\frac{\bb L^{(p)} \mb h}{\mb h} \quad \text{if}\;\;
\mu = \sum_{j\in S_p} \omega_j \, \pi^{(p)}_j \;\; \text{and}\;\;
\omega \in \ms P(S_p)\;,  \\
& +\infty \quad\text{otherwise}\;.
\end{aligned}
\right.
\end{equation}
In this formula, the infimum is carried over all functions
$\mb h:S_p \to (0,\infty)$.  We prove in \eqref{82} that
\begin{equation*}
\ms I^{(p+1)} (\mu) \;<\; \infty \quad \text{if and only if}\quad
\ms I^{(p)} (\mu) \;=\; 0\;.
\end{equation*}
By \eqref{79b}, this assertion holds also for $p=0$.

Recall the definition of $\Gamma$-convergence. We refer to \cite{Br}
for an overview on this subject. Fix a Polish space $\mc X$ and a
sequence $(U_n : n\in\bb N)$ of functionals on $\mc X$,
$U_n\colon \mc X \to [0,+\infty]$. The sequence $U_n$
\emph{$\Gamma$-converges} to the functional
$U\colon \mc X\to [0,+\infty]$ if and only if the two following
conditions are met:

\begin{itemize}
\item [(i)]\emph{$\Gamma$-liminf.} The functional $U$ is a
$\Gamma$-liminf for the sequence $U_n$: For each $x\in\mc X$ and each
sequence $x_n\to x$, we have that $\liminf_n U_n(x_n) \ge U(x)$.

\item [(ii)]\emph{$\Gamma$-limsup.} The functional $U$ is a
$\Gamma$-limsup for the sequence $U_n$: For each $x\in\mc X$ there
exists a sequence $x_n\to x$ such that $\limsup_n U_n(x_n) \le U(x)$.
\end{itemize}

\begin{theorem}
\label{mt3}
The functional $\ms I_n$ $\Gamma$-converges to $\ms I^{(0)}$. Moreover,
for each $1\le p\le \mf q$, the functional $\theta^{(p)}_n\, \ms I_n$
$\Gamma$-converges to $\ms I^{(p)}$.
\end{theorem}

This theorem provides an expansion of the large deviations rate
function $\ms I_n$ which can be written as
\begin{equation}
\label{f05}
\ms I_n \;=\; \ms I^{(0)} \;+\; \sum_{p=1}^{\mf q}
\frac{1}{\theta^{(p)}_n} \; \ms I^{(p)}\;.
\end{equation}
Therefore, the rate function $\ms I_n$ encodes all the characteristics
of the metastable behavior of the chain $X^{(n)}_t$. The time-scales
$\theta^{(p)}_n$ appear as the weights of the expansion, and the
meta-stable states $\pi^{(p)}_j$, $j\in S_p$, generate the space where
the rate functional $\ms I^{(p)} (\mu)$ is finite. Indeed, by
\eqref{09c}, $\ms I^{(p)} (\mu)$ is finite if and only if $\mu$ is a
convex combination of the measures $\pi^{(p)}_j$, $j\in S_p$.

Theorem \ref{mt3} extends to the context of continuous-time Markov
chains evolving on finite state-spaces a result by Di Ges\`u and
Mariani \cite{GM17} proved for reversible diffusions with a single
valley at each different depth.

\begin{remark}
\label{rm1}
Theorem \ref{mt3} should hold for nonreversible
dynamics. Reversibility is assumed here only to compute the
$\Gamma$-limsup through formula \eqref{f6}. It should also be possible
to obtain a metastable $\Gamma$-expansion for the level 2.5 large
deviations rate function derived in \cite{bfg}.
\end{remark}

\begin{remark}
\label{rm5}
The proof of Theorems \ref{mt0} and \ref{mt3} do not require the full
strength of assumption \eqref{mh}, but only the ability to
compute some capacities, the limit of the ratio of some measures and
of mean jump rates. Stating, however, the minimal conditions would
require much work.
\end{remark}

\section{The first time-scale}
\label{sec3}

In this section, we prove conditions (a) and (b) of Theorem \ref{mt0}
for $p=1$. Throughout the article, we adopt the following notation,
$\color{blue} O(\epsilon)$ represents a term whose absolute value is
bounded by $C_0\, \epsilon$ for some constant $C_0$ independent of $n$
and $\epsilon$. Similarly, $\color{blue} o_n(1)$ represents a term
which vanishes as $n\to\infty$.

Recall that we denote by $(\bb X_t : t\ge 0)$ the
$V$-valued continuous-time Markov chain with jump rates
$\bb R_0(x,y)$, and by $\bb Q_x$ the probability measure on
$D(\bb R_+, V)$ induced by the chain $\bb X_t$ with jump rates
$\bb R_0$ starting from $x$. For $x$, $y\in S$, let
\begin{equation}
\label{30}
{\color{blue} \omega (x,y)} \;: =\; \lim_{t\to\infty}
\bb Q_x\big[\, \bb X_t \, = \, y\,\big]\;.
\end{equation}
Clearly,
\begin{equation}
\label{32}
\omega (x,y) \;=\; 0\,, \;\; y\in\Delta \quad\text{and}\quad
\omega (x,y)  \;=\; \mf a^{(0)} (x,j)\, \pi^\sharp_j (y)\,,  \;\; y\in\ms V_j\;,  
\end{equation}
where $\mf a^{(0)} (x,j)$ has been introduced in \eqref{33}.

Denote by $\ms W_j$, $j\in S$, the set of points in $V$ which may end
in the set $\ms V_j$:
\begin{equation}
\label{31}
{\color{blue} \ms W_j} \;: =\; \big \{ \, x \in V :
\mf a^{(0)} (x,j) \,>\, 0\,\big\}
\;.
\end{equation}
Note that $V = \cup_j \ms W_j$.  Let $\ms B_j$ be the set of points
attracted to $\ms V_j$:
\begin{equation*}
{\color{blue} \ms B_j} \,: =\, \big\{ \, x\in V :
\mf a^{(0)} (x,j) \,=\, 1 \,\big\}\;.
\end{equation*}
Clearly, $\ms V_j \subset \ms B_j \subset \ms W_j$, and
$\ms B_j = \ms W_j \setminus (\cup_{k\not = j} \ms W_k) = V \setminus
(\cup_{k\not = j} \ms W_k)$. In other words, $\ms B_j^c = \cup_{k\not
= j} \ms W_k$. Moreover, as $\mf a^{(0)} (x,j) =0$ for
$x\in \cup_{k\not = j} \ms B_k$ and $\mf a^{(0)} (x,j) = 1$ for
$x\in \ms B_j$,
\begin{equation}
\label{35}
\omega (x,y)  \;=\; 0 \,, \quad
\omega (x,z)  \;=\;  \pi^\sharp_j (z)
\,,  \quad  x\,,\, z \in\ms V_j\,, \;\; y\in\ms V_k
\,, \;\; k\,\neq\, j \;.
\end{equation}

The first result describes the asymptotic behavior of
$p^{(n)}_{t} (x,y)$ in the slowest time-scale, $t=O(1)$. 

\begin{lemma}
\label{l09}
For every $\epsilon>0$, there exists $T_\epsilon$ such that
\begin{equation*}
\limsup_{n\to\infty} \big|\,
\mb P^n_{\! x}[\, X_{T_\epsilon} = y \,]
\,-\, \omega(x,y) \,\big| \;\le\; \epsilon\quad
\text{for all}\;\; x\,,\, y\, \in V \;,
\end{equation*}
where $\omega(x,y)$ has been introduced in \eqref{30}.
\end{lemma}

\begin{proof}
Fix $\epsilon>0$. By the ergodic theorem, there exists
$T_\epsilon<\infty$ such that
\begin{equation}
\label{29}
\big|\, \bb Q_{\! x}[\, \bb X_{T_\epsilon} = y \,]
\,-\, \omega(x,y) \,\big| \;\le\; \epsilon
\end{equation}
for all $x$, $y\in V$.

Couple $X^{(n)}_t$ and $\bb X_t$ making them jump together as much as
possible. Denote by $\bb P^{(n)}_x$ the measure on
$D(\bb R_+, V\times V)$ induced by the basic coupling starting from
$(x,x)$. By \eqref{01}, for all $T>0$,
\begin{equation}
\label{27}
\lim_{n\to\infty} \bb P^{(n)}_z
\big[\, \bb X_t \, = \, X^{(n)}_t \;,\;\; 0\le t\le T \,\big] \;
= \; 1\;.
\end{equation}
The assertion of the lemma follows from \eqref{29} and \eqref{27} with
$T = T_\epsilon$.
\end{proof}

Recall the definition of the sets $\ms V_j$, $1\le j\le \mf n$,
introduced in \eqref{05}. The chain $\bb X_t$ has only one closed
irreducible class if, and only if, $\mf n=1$.

\begin{corollary}
\label{l11}
Assume that $\mf n=1$, Then, $\lim_{n\to\infty} p^{(n)}_{\beta_n}(x,y)
= \pi^\sharp(y)$ for all $x$, $y\in V$, $\beta_n \succ 1$.
\end{corollary}

\begin{proof}
Fix $\epsilon>0$, and let $T_\epsilon$ be the constant given by Lemma
\ref{l09}. By the Markov property,
\begin{equation*}
p^{(n)}_{\beta_n}(x,y) \;=\; \sum_{z\in V}
p^{(n)}_{\beta_n - T_\epsilon}(x,z) \; p^{(n)}_{T_\epsilon}(z,y)\;.
\end{equation*}
By Lemma \ref{l09} and \eqref{32}, since $\mf a^{(0)} (y,1) = 1$ for
all $y\in V$, the right-hand side is equal to
\begin{equation*}
\sum_{z\in V}
p^{(n)}_{\beta_n - T_\epsilon}(x,z) \, \pi^\sharp(y) \;+\; O(\epsilon)
\;+\; o_n(1) \;=\;
\pi^\sharp(y) \;+\; O(\epsilon) \;+\; o_n(1) \;,
\end{equation*}
which completes the proof of the corollary.
\end{proof}

Corollary \ref{l11} shows that the asymptotic behavior of the
transition probability $p^{(n)}_t$ is trivial if $\mf n=1$, that is if
the Markov chain $\bb X_t$ has a unique closed irreducible
class. Assume that $\mf n\ge 2$.

\subsection*{The time-scale $\theta^{(1)}_n$}

Recall the definition of $\mf n_1$, $S_1$, and the sets
$\ms V^{(1)}_j$, $j\in S_1$, $\Delta_1$, introduced just above
\eqref{201}.  Let $\theta_n = \theta^{(1)}_n$ be given by \eqref{26b}
with $p=1$.

Recall from \cite[Section 2.3]{lx16} the definition of the sequence
$\alpha_n$. In the present context, by \eqref{01}, the sequence
$\alpha_n$ converges to a positive real number.  By Assertions 7.B and
equation (7.4) in \cite{lx16}, $\theta_n\succ 1$. The next result is
the first assertion of Theorem \ref{mt0}.

\begin{proposition}
\label{mt1}
Let $( \beta_n : n\ge 1)$ be a sequence such that
$1\prec \beta_n \prec \theta_n$. Then, \eqref{27b} holds for all $x$,
$y\in V$.
\end{proposition}

Recall that we call the sets $\ms V_j$ wells.  A time scale
$\beta_n\prec \theta_n$ is not long enough to allow the process to
jump from a well to another. This is the content of the next two
results. Lemma \ref{l07} states that starting from a well $\ms V_j$
the process does not visit another well (the set $\breve{\ms V}_j$
introduced in \eqref{26b}) in a time-scale $\beta_n$ such that
$\beta_n \prec \theta_n$. Corollary \ref{l08} extends this result
asserting that the points that might end up in another well (the set
$\cup_{k\not = j} \ms W_k$) are also not visited in this time-scale.

\begin{lemma}
\label{l07}
Let $(\beta_n : n\ge 1)$ be a sequence such that
$\beta_n \prec \theta_n$. Then, for all $j\in S_1$, $x\in \ms V_j$, 
\begin{equation*}
\lim_{n\to\infty}
\mb P^{n}_{\! x} \big[\,
H_{\breve{\ms V}_j} \,<\, \beta_n \,\big]\,=\; 0\;.
\end{equation*}
\end{lemma}

\begin{proof}
Fix $j\in S$, $x\in \ms V_j$.  By Lemma \ref{l06} and \eqref{58}, the
probability appearing in the statement of the lemma is bounded by
$ C_0 \, \beta_n\, \Cap_n (\{x\} \, ,\, \breve{\ms V}_j) / \pi_n (\ms
V_j)$ for some finite constant $C_0$, independent of $n$ and whose
value may change from line to line. By equation (B2) in \cite{lrev},
this expression is bounded by
$ C_0 \, \beta_n\, \Cap_n (\ms V_j \, ,\, \breve{\ms V}_j) / \pi_n
(\ms V_j)$. By the definition \eqref{26b} of $\theta_n$, this
expression is less than or equal to $ C_0 \, \beta_n\, /\, \theta_n$.
This concludes the proof of the lemma.
\end{proof}

\begin{corollary}
\label{l08}
Let $(\beta_n : n\ge 1)$ be an increasing sequence such that
$\beta_n \prec \theta_n$. Then, for all $j\in S$, $x\in \ms V_j$,
\begin{equation*}
\lim_{n\to\infty}
\mb P^{n}_{\! x} \big[\,
H_{\ms B_j^c} \,<\, \beta_n \,\big]\,=\; 0\;.
\end{equation*}
\end{corollary}

\begin{proof}
Assume first that $\beta_n \succ 1$.  Fix $j\in S$ and
$x\in \ms V_j$ and keep in mind that $\ms B_j^c = \cup_{k\not =j} \ms
W_k$.

We proceed by contradiction. Suppose the assertion does not hold.  In
this case, there exists $\delta>0$, $k\not = j$, $z\in \ms W_k$ and a
subsequence $n'$, still denoted by $n$, such that
$\mb P^{n}_{\! x} [\, H_{z} \,<\, \beta_n \,]\,>\, \delta$ for all
$n$. By the strong Markov property and this bound,
\begin{equation*}
\mb P^{n}_{\! x} \big[\,
H_{\breve{\ms V}_j} \,<\, 2\, \beta_n \,\big]\,\ge\,
\mb P^{n}_{\! x} [\, H_{z} \,<\, \beta_n \,]\,
\mb P^{n}_{\! z} \big[\,
H_{\breve{\ms V}_j} \,<\, \beta_n \,\big]
\,\ge\, \delta \,
\mb P^{n}_{\! z} \big[\,
H_{\breve{\ms V}_j} \,<\, \beta_n \,\big]\;.
\end{equation*}
Since $z\in \ms W_k$, there exists $\delta'>0$ and $T_0<\infty$, such
that $\bb Q_z \big[\, H_{\ms V_k} < T_0 \,\big] \, >\, \delta'$. By
\eqref{27} this estimate extends to $X^{(n)}_t$:
$\mb P^{n}_{\! z} \big[\, H_{\ms V_k} < T_0 \,\big] \, >\, \delta'/2$
for all $n$ sufficiently large.

Combining the previous estimates yields that
$\mb P^{n}_{\! x} [\, H_{\breve{\ms V}_j} \,<\, 2\, \beta_n \,]\,\ge\,
\delta\, \delta'/2$ for all $n$ sufficiently large because
$\beta_n \to\infty$. This result contradicts the assertion of Lemma
\ref{l07} and completes the proof of the corollary in the case
$\beta_n \succ 1$.

If the sequence $\beta_n$ is bounded, the result follows from the
coupling \eqref{27} because $\bb Q_x [\, H_{\ms W_k} < \infty \,]
\,\le\, \bb Q_x [\, H_{\ms V^c_j} < \infty \,]
\,=\, 0$ for all $x\in \ms V_j$, $k\not = j$.
\end{proof}

\begin{proof}[Proof of Proposition \ref{mt1}]
Fix $x$, $y\in V$, $\epsilon>0$, and recall the definition of
$\omega(x,y)$ introduced in \eqref{30}. Since $\ms V$ represents the
set of recurrent points of the chain $\bb X_t$, there exists
$T_\epsilon>0$ such that
\begin{equation}
\label{28}
\bb Q_w\big[\, \bb X_T \, \in\, \ms V\,\big] \;\ge\; 1\,-\, \epsilon
\;, \quad
\big|\, \bb Q_w\big[\, \bb X_T \, = \, z\,\big] \,-\, 
\omega(w,z) \,\big| \;\le\; \epsilon
\end{equation}
for all $w$, $z\in V$, $T\ge T_\epsilon$.  

Assume first that $y\in \Delta$. By the Markov property,
\begin{equation*}
\mb P^{n}_{\! x} \big[\, X_{\beta_n} \,=\, y \,\big]
\;=\; \sum_{z\in V}
\mb P^{n}_{\! x} \big[\, X_{\beta_n- T_\epsilon} \,=\, z \,\big]\,
\mb P^{n}_{\! z} \big[\, X_{T_\epsilon} \,=\, y \,\big]\;.
\end{equation*}
By \eqref{27}, \eqref{28} and \eqref{32}, the right-hand side is
bounded by $o_n(1) + \epsilon$, which proves \eqref{27b} for $y\in
\Delta$. 

Assume that $y\in \ms V_k$ for some $k\in S_1$.  By the Markov property,
\begin{equation*}
\mb P^{n}_{\! x} \big[\, X_{\beta_n} \,=\, y \,\big]
\;=\; \sum_{z\in V}
\mb P^{n}_{\! x} \big[\, X_{T_\epsilon} \,=\, z \,\big]\,
\mb P^{n}_{\! z} \big[\, X_{\beta_n-T_\epsilon} \,=\, y \,\big]\;.
\end{equation*}
By \eqref{27}, \eqref{28} and \eqref{32}, the right-hand side is equal
to
\begin{equation*}
\sum_{j\in S} \sum_{z\in \ms V_j } \mf a^{(0)}(x,j) \, \pi^\sharp_j(z)
\, \mb P^{n}_{\! z} \big[\, X_{\beta_n-T_\epsilon} \,=\, y \,\big]
\,+\, o_n(1) \,+\, O(\epsilon) \;.
\end{equation*}
Since $\beta_n\prec \theta_n$, by Corollary \ref{l08}, we may add
inside the probability the event $\{H_{\ms B^c_j} \,\ge\,
\beta_n\}$. The previous sum is thus equal to
\begin{equation*}
\sum_{j\in S} \sum_{z\in \ms V_j }  \mf a^{(0)}(x,j)\, \pi^\sharp_j(z)
\, \mb P^{n}_{\! z} \big[\, X_{\beta_n-T_\epsilon} \,=\, y
\,,\, H_{\ms B^c_j} \,\ge\, \beta_n \,\big]
\,+\, o_n(1) \,+\, O(\epsilon)\;.
\end{equation*}
As $y$ belongs to $\ms V_k$ and $\ms V_k \cap \ms B_j =\varnothing$ if
$j\not = k$, this sum is equal to
\begin{equation*}
\sum_{z\in \ms V_k}  \mf a^{(0)}(x,k)\, \pi^\sharp_k(z)
\, \mb P^{n}_{\! z} \big[\, X_{\beta_n-T_\epsilon} \,=\, y
\,,\, H_{\ms B^c_k} \,\ge\, \beta_n \,\big]
\,+\, o_n(1) \,+\, O(\epsilon)\;.
\end{equation*}
In view of the presence of the event $\{H_{\ms B^c_k} \,\ge\,
\beta_n\}$, the previous probability is equal to 
\begin{equation*}
\sum_{w\in \ms B_k}
\mb P^{n}_{\! z} \big[\, X_{\beta_n-T_\epsilon} \,=\, y
\,,\, X_{\beta_n-2T_\epsilon} \,=\, w
\,,\, H_{\ms B^c_k} \,\ge\, \beta_n \,\big]
\end{equation*}
By Corollary \ref{l08}, we may remove the event
$\{H_{\ms B^c_k} \,\ge\, \beta_n\}$ at a cost $o_n(1)$ and apply the
Markov property to conclude that the previous sum is equal to
\begin{equation*}
\sum_{w\in \ms B_k}
\mb P^{n}_{\! z} \big[\, X_{\beta_n-2T_\epsilon} \,=\, w
\,\big] \; 
\mb P^{n}_{\! w} \big[\, X_{T_\epsilon} \,=\, y \,\big]
\,+\, o_n(1)\;.
\end{equation*}
By \eqref{27}, \eqref{28} and \eqref{32}, this expression is equal
to
\begin{equation*}
\sum_{w\in \ms B_k} \mb P^{n}_{\! z} \big[\, X_{\beta_n-2T_\epsilon} \,=\, w
\,\big] \;  \mf a^{(0)}(w,k) \, \pi^\sharp_k(y)
\,+\, o_n(1) \,+\, O(\epsilon) \;.
\end{equation*}
Since $w$ belongs to $\ms B_k$, $\mf a^{(0)}(w,k) =1$ and the previous
expression is equal to
\begin{equation*}
\pi^\sharp_k(y) \, 
\mb P^{n}_{\! z} \big[\, X_{\beta_n-2T_\epsilon} \,\in \, \ms B_k
\,\big] \;  
\,+\, o_n(1) \,+\, O(\epsilon) \;.
\end{equation*}
Since $z$ belongs to $\ms V_k$ and
$\{X_{\beta_n-2T_\epsilon} \,\not \in \, \ms B_k\} \subset \{H_{\ms
B^c_k} \,\le\, \beta_n \}$, by Corollary \ref{l08}, the expression in
the previous displayed equation is equal to
$\pi^\sharp_k(y) \,+\, o_n(1) \,+\, O(\epsilon)$.

Combining the previous estimates yields that
\begin{equation*}
\begin{aligned}
\mb P^{n}_{\! x} \big[\, X_{\beta_n} \,=\, y \,\big]
\; & =\; \sum_{z\in \ms V_k}  \mf a^{(0)}(x,k)\, \pi^\sharp_k(z)
\pi^\sharp_k(y) \,+\, o_n(1) \,+\, O(\epsilon) \\
& =\; \mf a^{(0)}(x,k)\,
\pi^\sharp_k(y) \,+\, o_n(1) \,+\, O(\epsilon)  \;,
\end{aligned}
\end{equation*}
as claimed.
\end{proof}

\subsection*{The time-scale $t\, \theta_n$}

We turn to the proof of Theorem \ref{mt0}.(b) for $p=1$.

\begin{proposition}
\label{mt2}
Assertion \eqref{60} holds for $p=1$ and all $t>0$, $x\in V$.
\end{proposition}

The proof of this result relies on the following lemma.

\begin{lemma}
\label{l10}
Recall the definition of the set $\Delta$ introduced in
\eqref{05}. Then, 
\begin{equation*}
\lim_{\delta\to 0} \limsup_{n\to\infty} \max_{j\in S}
\max_{x\in \ms V_j} \sup_{2\delta \le s \le 3\delta}
\mb P^n_{\! x} \big[\, X_{s\theta_n} \in \Delta\,\big] \;=\; 0\;.
\end{equation*}
\end{lemma}

\begin{proof}
Fix $\epsilon>0$ and let $T_\epsilon$ be the constant given by Lemma
\ref{l09}.  By the Markov property, the probability appearing in the
statement of the lemma is bounded by
\begin{equation*}
\max_{y\in V} \mb P^n_{\! y} \big[\, X_{T_\epsilon} \in \Delta\,\big]
\end{equation*}
By \eqref{32} and Lemma \ref{l09}, this expression is bounded by
$\epsilon + o_n(1)$, which proves the lemma.
\end{proof}

By \cite[Proposition 2.1]{llm18}, \cite[Theorem 2.7]{lx16} and Lemma
\ref{l10}, for every $t>0$, $j$, $k\in S_1$, $x\in \ms V_j$, 
\begin{equation}
\label{37}
\lim_{n\to\infty} \mb P^n_{\! x} \big[\, X_{t\theta_n}
\,\in \, \ms V_k \,\big] \;=\; p^{(1)}_t(j,k)\;,
\end{equation}
where the transition probability $p^{(1)}_t$ has been introduced in
\eqref{61}. 

\begin{proof}[Proof of Proposition \ref{mt2}]

Suppose that $y\in \Delta$ and fix $t>0$, $\epsilon>0$. In this case,
by the Markov property
\begin{equation*}
\mb P^n_{\! x} \big[\, X_{t\theta_n} \,=\, y \,\big] \;=\;
\sum_{z\in V} \mb P^n_{\! x} \big[\, X_{t\theta_n - T_\epsilon}
\,=\, z \,\big] \, 
\mb P^n_{\! z} \big[\, X_{T_\epsilon} \,=\, y \,\big]\;, 
\end{equation*}
where $T_\epsilon$ is given by Lemma \ref{l09}. By this lemma, the
second probability on the right hand side is bounded by
$\omega(z,y) + \epsilon + o_n(1)$. By \eqref{32}, as $y\in \Delta$,
$\omega(z,y)=0$ so that
\begin{equation*}
\lim_{n\to\infty} \mb P^n_{\! x} \big[\, X_{t\theta_n} \,=\, y \,\big] \;=\;
0\;,
\end{equation*}
as claimed.

Suppose that $y\in \ms V_m$ for some $m\in S_1$ and fix $t>0$,
$\epsilon>0$. By the Markov property
\begin{equation*}
\mb P^n_{\! x} \big[\, X_{t\theta_n} \,=\, y \,\big] \;=\;
\sum_{z, z' \in V} \mb P^n_{\! x} \big[\, X_{T_\epsilon} \,=\, z \,\big]\
\mb P^n_{\! z} \big[\, X_{t\theta_n - 2 T_\epsilon}
\,=\, z' \,\big] \, 
\mb P^n_{\! z'} \big[\, X_{T_\epsilon} \,=\, y \,\big]\;, 
\end{equation*}
where $T_\epsilon$ is given by Lemma \ref{l09}. By this lemma and
\eqref{32}, which asserts that $\omega(x',y') =0$ if $y'\in\Delta$,
this expression is equal to
\begin{equation*}
\sum_{z' \in V} \sum_{j\in S_1} \sum_{z\in \ms V_j} \omega(x,z) \,
\mb P^n_{\! z} \big[\, X_{t\theta_n - 2 T_\epsilon}
\,=\, z' \,\big] \, \omega(z',y) \;+\; \epsilon \;+\; o_n(1) \;.
\end{equation*}
The first part of the proof permits to restrict the first sum to
$z'\in \ms V$. Since $y\in \ms V_m$, by \eqref{35}, we may further
restrict the sum to $z'\in \ms V_m$, and then replace $\omega(z',y)$
by $\pi^{(1)}_m(y)$. Hence the previous sum is equal to
\begin{equation*}
\pi^{(1)}_m(y)\, \sum_{j\in S_1} \sum_{z\in \ms V_j} \omega(x,z) \,
\mb P^n_{\! z} \big[\, X_{t\theta_n - 2 T_\epsilon}
\,\in \, \ms V_m  \,\big] \,  \;+\; \epsilon \;+\; o_n(1) \;,
\end{equation*}
where we summed over $z' \in \ms V_m$. By \eqref{37}, as $n\to\infty$,
this expression converges to
\begin{equation*}
\pi^{(1)}_m(y)\, \sum_{j\in S_1} \sum_{z\in \ms V_j} \omega(x,z) \,
p^{(1)}_t(j,m)  \;+\; \epsilon \;=\;
\sum_{j\in S_1} \mf a^{(0)} (x,j) \, p^{(1)}_t(j,m) \, \pi^{(1)}_m(y) \;+\;
\epsilon  \;,
\end{equation*}
as claimed.
\end{proof}

\section{Longer time-scales}
\label{sec4}

In this section, we complete the proof of Theorem \ref{mt0}.  We first
derive some properties of the weights $\mf a^{(p)}$ needed in the
argument. Recall that $\mf q$ represents the number of time-scales or
steps in the construction of the rooted tree in Section
\ref{sec1}. Moreover, the chain $\bb X^{(\mf q)}_t$ has only one
closed irreducible class.

Next result states that a point in the closed irreducible class
$\ms V^{(p+1)}_\ell$ is not absorbed at $\ms V^{(p+1)}_m$ for
$m\neq \ell$. 

\begin{lemma}
\label{l18}
For all $0\le p< \mf q$, $\ell \in S_{p+1}$, $x\in \ms
V^{(p+1)}_\ell$, 
\begin{equation}
\label{67}
\mf a^{(p)}(x,m) \;=\; 0
\quad \text{for all} \quad m \,\in\, S_{p+1} \setminus \{\ell\} \; .
\end{equation}
\end{lemma}

\begin{proof}
The proof is by induction in $p$. For $p=0$, by definition \eqref{33}
of $\mf a^{(0)}$, for all $\ell \in S_{1}$, $x\in \ms V_\ell$, $m \in
S_{1} \setminus \{\ell\}$, 
\begin{equation*}
\mf a ^{(0)}(x,m)  \;=\; \lim_{t\to\infty}
\bb Q_x\big[\, \bb X_t \, \in\, \ms V_m \,\big] \;=\; 0
\end{equation*}
because
the sets $\ms V_k$ are the closed irreducible classes of the chain
$\bb X_t$.

Assume that \eqref{67} holds for $0\le p\le r-1$. Fix
$\ell \in S_{r+1}$, $x\in \ms V^{(r+1)}_\ell$,
$m \in S_{r+1} \setminus \{\ell\}$. By definition of
$\mf a^{(r)}(x,m)$,
\begin{equation*}
\mf a^{(r)} (x,m)
\;:=\; \sum_{j\in S_r} \mf a^{(r-1)} (x, j) \, \mf A^{(r)} (j, m) \;.
\end{equation*}
We may restrict the sum to $j \in \mf R^{(r)}_m$. Indeed, since
$S_r \setminus \mf R^{(r)}_m = \cup_{k \in S_{r+1} \setminus \{m\}}
\mf R^{(r)}_k$ and since the sets $\mf R^{(r)}_k$, $k\in S_{r+1}$,
are the closed irreducible classes of the chain $\bb X^{(r)}_t$,
$\mf A^{(r)} (j, m)=0$ for $j\in S_r \setminus \mf R^{(r)}_m$. Hence,
\begin{equation*}
\mf a^{(r)} (x,m)
\;:=\; \sum_{j\in \mf R^{(r)}_m}
\mf a^{(r-1)} (x, j) \, \mf A^{(r)} (j, m) \;.
\end{equation*}

On the other hand, as
$x\in \ms V^{(r+1)}_\ell = \cup_{i \in \mf R^{(r)}_\ell} \ms
V^{(r)}_i$ and $\mf R^{(r)}_\ell \cap \mf R^{(r)}_m = \varnothing$
because $\ell \not = m$, $x$ belongs to some $\ms V^{(r)}_i$ with
$i\not \in \mf R^{(r)}_m$. Thus, by the induction assumption
$\mf a^{(r-1)} (x, j) =0$ for all $j\in \mf R^{(r)}_m$, which yields
that $\mf a^{(r)} (x,m)=0$, as claimed.
\end{proof}

The previous result is stated for $p<\mf q$ because
$\bb X^{(\mf q)}_t$ has only one irreducible class which makes
$S_{\mf q+1}$ a singleton.

It has been noted, just before the statement of Theorem \ref{mt0},
that $\mf a^{(p)} (x, \,\cdot\,)$ is a probability measure on
$S_{p+1}$ for all $x\in V$. Therefore, by the previous lemma, for all
$1\le p < \mf q$, $\ell \in S_{p+1}$, $x\in \ms V^{(p+1)}_\ell$,
\begin{equation}
\label{69}
\mf a^{(p)} (x, \ell) \;=\; 1 \quad\text{so that}\quad 
\Pi_p(x,\,\cdot\,) \;=\; \pi^{(p+1)}_\ell
(\,\cdot\,) \;,
\end{equation}
where $\Pi_p(x,\,\cdot\,)$ has been introduced in \eqref{27b}.  In
particular, under these conditions on $\ell$ and $x$,
\begin{equation}
\label{68}
\Pi_p(x,y) \;=\; 0
\end{equation}
for all $y\in \ms V^{(p+1)}_m$, $m\in S_{p+1} \setminus \{\ell\}$.

This identity can be extended. Since the support of the measure
$\pi^{(p+1)}_m(\,\cdot\,)$ is the set $\ms V^{(p+1)}_m$,
$m\in S_{p+1}$, and $\cup_m \ms V^{(p+1)}_m = \ms V^{(p+1)}$,
\begin{equation}
\label{66}
\Pi_p(x, y) \;=\; 0 \quad\text{for all}\quad x\,\in\, V\,, \;
y\,\in\, \big(\, \ms V^{(p+1)} \,\big)^c \,=\, \Delta_{p+1}\;.
\end{equation}

\subsection*{ Induction hypotheses:}

Assume that we proved for some $1\le p<\mf q$ that
for all $t>0$, $x$, $y\in V$,
\begin{equation}
\label{62}
\lim_{n\to\infty} p^{(n)}_{t\, \theta^{(p)}_n} (x,y) \;=\;
\sum_{k\in S_p} \omega^{(p)}_t (x, k) \; \pi^{(p)}_k(y) \;,
\end{equation}
where $ \omega^{(p)}_t$, $\pi^{(p)}_k$ are as in the statement of
Theorem \ref{mt0}. This assertion for $p=1$ is the content of
Proposition \ref{mt2}.

\subsection*{The time scale $t\theta^{(p)}_n$, as $t\to\infty$}

Recall the definition of $\mf A^{(p)}(j,m)$, $m\in S_{p+1}$,
$j\in S_p$, introduced in \eqref{63}.  With this notation, for every
$j$, $k\in S_p$,
\begin{equation}
\label{64}
\lim_{t\to\infty} p^{(p)}_t(j,k) \;=\; \sum_{m\in S_{p+1}}
\mf A^{(p)} (j,m)\; M^{(p)}_m(k)\;,
\end{equation}
where, recall, $M^{(p)}_m(\cdot)$, $m\in S_{p+1}$, the stationary
state of the Markov chain $\bb X^{(p)}_t$ restricted to
$\mf R^{(p)}_m$.  In particular, $\lim_{t\to\infty} p^{(p)}_t(j,k) =0$
for every $k\in \mf T_p$.

By the induction assumption \eqref{62}, the definition of
$\omega^{(p)}_t$, \eqref{64} and the fact that the support of the
measure $M^{(p)}_m$ is the set $\mf R^{(p)}_m$, for all $x\in V$,
\begin{equation*}
\lim_{t\to\infty} \lim_{n\to\infty} p^{(n)}_{t\, \theta^{(p)}_n} (x,\,\cdot\,)
\; =\; \sum_{j\in S_p} \sum_{m\in S_{p+1}} \sum_{k\in \mf R^{(p)}_m}
\mf a^{(p-1)}  (x,j) \, \mf A^{(p)} (j,m) \, M^{(p)}_m(k)\, 
\pi^{(p)}_k(\,\cdot\,)  \;.
\end{equation*}
By definition of the measures $\pi^{(p+1)}_m$ and by the one of
$a^{(p)} (x,m)$, given in \eqref{33b}, this expression is equal to
\begin{equation*}
\sum_{j\in S_p} \sum_{m\in S_{p+1}} 
\mf a^{(p-1)} (x,j) \, \mf A^{(p)} (j,m) \, \pi^{(p+1)}_m(\,\cdot\,)
\; =\; \sum_{m\in S_{p+1}} 
\mf a^{(p)} (x,m) \, \pi^{(p+1)}_m(\,\cdot\,) \;. 
\end{equation*}
Hence, we proved that for all $x\in V$,
\begin{equation}
\label{65}
\lim_{t\to\infty} \lim_{n\to\infty} p^{(n)}_{t\, \theta^{(p)}_n} (x,\,\cdot\,)
\; =\; \sum_{m\in S_{p+1}} 
\mf a^{(p)} (x,m) \, \pi^{(p+1)}_m(\,\cdot\,)
\;=\; \Pi_p(x, \,\cdot\, )  \;.
\end{equation}

The argument above shows that Theorem \ref{mt0}.(d) follows from
Theorem \ref{mt0}.(b). Assertion (c) of this theorem follows from the
fact that $p^{(p)}_t(j,k)$ converges to $\delta_{j,k}$ as $t\to 0$.

\subsection*{The time scale $\theta^{(p)}_n\prec \beta_n \prec \theta^{(p+1)}_n$}

By \eqref{65} and \eqref{66},
\begin{equation}
\label{46}
\lim_{t\to\infty} \lim_{n\to\infty} \max_{x\in V }\mb P^{n}_{\! x} \big[\,
X_{t \theta^{(p)}_n} \not\in \ms V^{(p+1)} \,\big] \;=\; 0\;.
\end{equation}
In particular,
\begin{equation}
\label{42}
\lim_{t\to\infty} \limsup_{n\to\infty} \max_{x\in V }\mb P^{n}_{\! x} \big[\,
H_{\breve{\ms V}^{(p+1)}} > t \theta^{(p)}_n \,\big] \;=\; 0\;.
\end{equation}

Suppose that $S_{p+1}$ is a singleton. In other words, that the chain
$\bb X^{(p)}_t$ has a unique closed irreducible class. In this case
$p=\mf q$ and $\theta^{(p+1)}_n = +\infty$ for all $n\ge 1$.  If
$S_{p+1}$ is not a singleton, recall from \eqref{26b} the definition
of $\theta^{(p+1)}_n$. As stated in \eqref{51}, by \cite[Assertion
8.B]{lx16}, $\theta^{(p)}_n \prec \theta^{(p+1)}_n$.

\begin{lemma}
\label{l14}
Let $(\beta_n:n\ge 1)$ be a sequence such that
$\theta^{(p)}_n \prec \beta_n \prec \theta^{(p+1)}_n$.
Then, for all $m\in S_{p+1}$,
\begin{equation*}
\lim_{n\to\infty} \max_{x\in \ms V^{(p+1)}_m }\mb P^{n}_{\! x} \big[\,
H_{\breve{\ms V}^{(p+1)}_m} < \beta_n \,\big] \;=\; 0\;,
\end{equation*}
where $\breve{\ms V}^{(p+1)}_m$ has been introduced in \eqref{26b}.
\end{lemma}

\begin{proof}
Fix $m\in S_{p+1}$, $x\in \ms V^{(p+1)}_m$.  By Lemma \ref{l06} and
\eqref{58}, the probability appearing above is bounded by
$ C_0 \, \beta_n\, \Cap_n (\{x\} \, ,\, \breve{\ms V}^{(p+1)}_m) /
\pi_n (\ms V^{(p+1)}_m)$ for some finite constant $C_0$ independent
of $n$. By equation (B2) in \cite{lrev}, this expression is bounded by
$ C_0 \, \beta_n\, \Cap_n (\ms V^{(p+1)}_m \, ,\, \breve{\ms
V}^{(p+1)}_m) / \pi_n (\ms V^{(p+1)}_m)$. By the definition
\eqref{26b} of $\theta^{(p+1)}_n$, this expression is less than or
equal to $ C_0 \, \beta_n\, /\, \theta^{(p+1)}_n$.  This concludes the
proof of the lemma.
\end{proof}

Let $\Delta_{p+1, m}$, $m \in S_{p+1}$, be the set of points in
$\Delta_{p+1}$ which may be absorbed by a set $\ms V^{(p+1)}_\ell$,
$\ell \not = m$, in the time-scale $\theta^{(p)}_n$:
\begin{equation*}
{\color{blue} \Delta_{p+1, m} }\; :=\;
\Big\{\, x\in \Delta_{p+1} : 
\sum_{\ell \in S_{p+1} \setminus \{m\}} \mf a^{(p)}(x,\ell)  >0\,\Big \}\;.
\end{equation*}

\begin{corollary}
\label{l15}
Let $(\beta_n:n\ge 1)$ be a sequence such that
$\theta^{(p)}_n \prec \beta_n \prec \theta^{(p+1)}_n$.
Then, for all $m\in S_{p+1}$,
\begin{equation*}
\lim_{n\to\infty} \max_{x\in \ms V^{(p+1)}_m }\mb P^{n}_{\! x} \big[\,
H_{\Delta_{p+1, m}} < \beta_n \,\big] \;=\; 0\;.
\end{equation*}
\end{corollary}

\begin{proof}
Suppose the assertion is not true. Then, there exists $\delta>0$,
$x\in \ms V^{(p+1)}_m$ and a subsequence $n'$, still denoted by $n$,
such that
\begin{equation*}
\mb P^{n}_{\! x} \big[\,
H_{\Delta_{p+1, m}} < \beta_n \,\big] \;\ge \; \delta
\end{equation*}
for all $n$ sufficiently large.

Fix $t>0$ to be chosen later. Denote by
$\vartheta_s: D(\bb R_+, V) \to D(\bb R_+, V)$, $s\ge 0$, the
semigroup of translations of a trajectory:
$\color{blue} (\vartheta_s \mf x)(r) = \mf x(r+s)$, $r\ge 0$. By the
strong Markov property,
\begin{equation*}
\begin{aligned}
& \mb P^{n}_{\! x} \big[\,
H_{\breve{\ms V}^{(p+1)}_m} < \beta_n \,+\, t\, \theta^{(p)}_n \,\big]
\;  \ge\;
\mb P^{n}_{\! x} \big[\,
H_{\Delta_{p+1, m}} < \beta_n \,,\,
H_{\breve{\ms V}^{(p+1)}_m} \,\circ\, \vartheta_{H_{\Delta_{p+1, m}}}
< t\, \theta^{(p)}_n \,\big] \\
&\qquad\qquad   \ge\;
\mb P^{n}_{\! x} \big[\,
H_{\Delta_{p+1, m}} < \beta_n \,\big] \, \min_{z\in \Delta_{p+1, m}}
\mb P^{n}_{\! z} \big[\,  H_{\breve{\ms V}^{(p+1)}_m} 
< t\, \theta^{(p)}_n \,\big] \\
&\qquad\qquad   \ge\;
\mb P^{n}_{\! x} \big[\,
H_{\Delta_{p+1, m}} < \beta_n \,\big] \, \min_{z\in \Delta_{p+1, m}}
\mb P^{n}_{\! z} \big[\, 
X_{t\, \theta^{(p)}_n} \in \breve{\ms V}^{(p+1)}_m \,\big]\;.
\end{aligned}
\end{equation*}
By the first part of the proof, the first term is bounded below by
$\delta$ for $n$ sufficiently large. By Theorem \ref{mt0}.(d), proved
in the previous subsection for $p$, for each $z\in \Delta_{p+1, m}$,
the second probability converges, as $n\to\infty$ and then
$t\to\infty$, to
\begin{equation*}
\sum_{\ell \in S_{p+1} \setminus \{m\}} \mf a^{(p)}(z,\ell) \;.
\end{equation*}
By definition of $\Delta_{p+1, m}$, this term is strictly positive for
each $z\in \Delta_{p+1, m}$. Therefore, there exist $\delta'>0$ and
$t_0<\infty$ such that
\begin{equation*}
\liminf_{n\to \infty} \min_{z\in \Delta_{p+1, m}}
\mb P^{n}_{\! z} \big[\, 
X_{t_0\, \theta^{(p)}_n} \in \breve{\ms V}^{(p+1)}_m \,\big] \;\ge\;
\delta' \;.
\end{equation*}

Putting together the previous estimates yields that
\begin{equation*}
\liminf_{n\to \infty} \mb P^{n}_{\! x} \big[\,
H_{\breve{\ms V}^{(p+1)}_m} < \beta_n \,+\, t_0\, \theta^{(p)}_n \,\big]
\;  >\; 0 \;,
\end{equation*}
in contradiction with the statement of Lemma \ref{l14}. This completes
the proof of the corollary.
\end{proof}

For $m\in S_{p+1}$, let
\begin{equation*}
{\color{blue} \ms U^{(p+1)}_m}\; :=\;
\Big\{\, x\in V : \mf a^{(p)}(x,m)  =1\,\Big \} \;.
\end{equation*}
By \eqref{69} and the definition of the set $\Delta_{p+1,m}$,
introduced just before the statement of Corollary \ref{l15}, the set
$\ms U^{(p+1)}_m$ is equal to
$\ms V^{(p+1)}_m \, \cup\, [\, \Delta_{p+1} \setminus
\Delta_{p+1,m}\,]$. Thus,
$(\ms U^{(p+1)}_m)^c = \breve{\ms V}^{(p+1)}_m \cup \Delta_{p+1,m}$.

\begin{proposition}
\label{l12}
Let $\theta^{(p)}_n \prec \beta_n \prec \theta^{(p+1)}_n$. Then, for
all $x \in V$,
\begin{equation*}
\lim_{n\to\infty} p^{(n)}_{\beta_n} (x,\,\cdot\,)
\; =\; \Pi_p(x, \,\cdot\, ) \;,  
\end{equation*}
where $\Pi_p(x, \,\cdot\, )$ has been introduced in \eqref{27b}.
\end{proposition}

\begin{proof}
Fix $\epsilon>0$. By \eqref{65}, there exists $t_\epsilon$ such that
\begin{equation}
\label{39}
\Big|\, \lim_{n\to\infty} p^{(n)}_{t\, \theta^{(p)}_n} (x,y)
\, - \, \Pi_p(x, y) \,\Big| \;<\; \epsilon
\end{equation}
for all $x$, $y\in V$, $t>t_\epsilon$.

Fix $t>t_\epsilon$. By the Markov property,
\begin{equation*}
p^{(n)}_{\beta_n} (x,y) \;=\;
\sum_{z\in V} p^{(n)}_{t\, \theta^{(p)}_n} (x,z)
\; p^{(n)}_{\beta_n - t\, \theta^{(p)}_n} (z,y)\;.
\end{equation*}
By \eqref{39} and \eqref{66}, this expression is equal to
\begin{equation}
\label{44}
\sum_{z\in \ms V^{(p+1)}} \Pi_p  (x,z)
\; \mb P^{n}_{\! z}\big[\, X_{\beta_n - t\, \theta^{(p)}_n} = y\,\big]
\;+\; o_n(1) \;+\; O(\epsilon) \;.
\end{equation}

Fix $s>0$, and rewrite the sum appearing in \eqref{44} as
\begin{equation*}
\sum_{m\in S_{p+1}} \sum_{z\in \ms V^{(p+1)}_m}
\sum_{w\in V} \Pi_p  (x,z)
\; \mb P^{n}_{\! z}\big[\, X_{\beta_n - t\, \theta^{(p)}_n} = y
\,,\, X_{\beta_n - (t+s)\, \theta^{(p)}_n} = w \,\big] \;.
\end{equation*}
We have shown just above the statement of the proposition that
$(\ms U^{(p+1)}_m)^c = \breve{\ms V}^{(p+1)}_m \cup
\Delta_{p+1,m}$. Hence, by Lemma \ref{l14} and Corollary \ref{l15}, we
may restrict the third sum to $w\in \ms U^{(p+1)}_m$ by paying a price
of order $o_n(1)$. Apply the Markov property to rewrite the resulting
expression as
\begin{equation*}
\sum_{m\in S_{p+1}} \sum_{z\in \ms V^{(p+1)}_m}
\sum_{w\in \ms U^{(p+1)}_m} \Pi_p  (x,z)
\; \mb P^{n}_{\! z}\big[\, X_{\beta_n - (t+s)\, \theta^{(p)}_n} = w \,\big]
\, \mb P^{n}_{\! w}\big[\, X_{s\, \theta^{(p)}_n} = y \,\big] \;.
\end{equation*}

By \eqref{65} the last probability converges, as $n\to\infty$, and
then $s\to \infty$, to $\Pi_p(w,y)$. By definition of $\Pi_p$ and the
one of $\ms U^{(p+1)}_m$, since $w\in \ms U^{(p+1)}_m$ and
$\mf a^{(p)}(x,\,\cdot\,)$ is a probability measure on $S_{p+1}$,
$\Pi_p(w,y) = \pi^{(p+1)}_m(y)$. This expression does not depend on
$w$. By Lemma \ref{l14} and Corollary \ref{l15}, the previous sum is
thus equal to
\begin{equation*}
\sum_{m\in S_{p+1}} \sum_{z\in \ms V^{(p+1)}_m}
\Pi_p  (x,z) \; \pi^{(p+1)}_m(y) \;+\; o_n(1)\;.
\end{equation*}
By the definition \eqref{27b} of $\Pi_p$, this expression is equal to
\begin{equation*}
\sum_{\ell\in S_{p+1}} \mf a^{(p)}(x,\ell) \; \pi^{(p+1)}_\ell(y)  \;,
\end{equation*}
as claimed.
\end{proof}

\subsection*{The time scale $\theta^{(p+1)}_n$}

If $S_{p+1}$ is a singleton,  $p=\mf q$,
$\theta^{(p+1)}_n = +\infty$ for all $n$ and the proof of Theorem
\ref{mt0} ends at the previous step where we considered the time-scale
$\theta^{(p)}_n \prec \beta_n \prec \theta^{(p+1)}_n \equiv +\infty$.

Assume that $S_{p+1}$ is not a singleton.  The next result completes
the recursive argument and the proof of Theorem \ref{mt0}.  It states
that the induction hypothesis \eqref{62} holds at level $p+1$ if it
holds at level $p$.

\begin{proposition}
\label{l16}
For all $t>0$, $x$, $y\in V$,
\begin{equation*}
\lim_{n\to\infty} p^{(n)}_{t\, \theta^{(p+1)}_n} (x,y) \;=\;
\sum_{m\in S_{p+1}} \omega^{(p+1)}_t (x, m) \; \pi^{(p+1)}_m(y)\;.
\end{equation*}
\end{proposition}

The proof of this result is based on Lemma \ref{l17} below. 

\begin{lemma}
\label{l17}
Recall the definition of the set $\Delta_{p+1}$ introduced in
\eqref{05b}. Then, 
\begin{equation*}
\lim_{\delta\to 0} \limsup_{n\to\infty} \max_{m\in S_{p+1}}
\max_{x\in \ms V^{(p+1)}_m} \sup_{2\delta \le s \le 3\delta}
\mb P^n_{\! x} \big[\, X_{s\theta^{(p+1)}_n} \in \Delta_{p+1}\,\big] \;=\; 0\;.
\end{equation*}
\end{lemma}

\begin{proof}
Fix $\delta>0$, $\epsilon>0$. By \eqref{46}, there exists
$t_\epsilon<\infty$
\begin{equation}
\label{47}
\lim_{n\to\infty} \max_{x\in V }\mb P^{n}_{\! x} \big[\,
X_{t \theta^{(p)}_n} \not\in \ms V^{(p+1)} \,\big] \;\le \; \epsilon
\end{equation}
for all $t\ge t_\epsilon$.  By the Markov property, since
$\theta^{(p)}_n \prec \theta^{(p+1)}_n$, the probability appearing in
the statement of the lemma is bounded by
\begin{equation*}
\max_{y\in V} \mb P^n_{\! y} \big[\, X_{t_\epsilon \theta^{(p)}_n}
\in \Delta_{p+1}\,\big] 
\end{equation*}
for all $x\in V$, $s\in [2\delta, 3\delta]$.  By \eqref{47}, this
expression is bounded by $\epsilon + o_n(1)$, which proves the lemma.
\end{proof}

By \cite[Proposition 2.1]{llm18}, \cite[Theorem 2.7]{lx16} and Lemma
\ref{l17} for every $t>0$, $\ell$, $m \in S_{p+1}$, $x\in \ms V^{(p+1)}_\ell$, 
\begin{equation}
\label{37b}
\lim_{n\to\infty} \mb P^n_{\! x} \big[\, X_{t\theta^{(p+1)}_n}
\,\in \, \ms V^{(p+1)}_m \,\big] \;=\; p^{(p+1)}_t(\ell, m)\;,
\end{equation}
where, recall, $ p^{(p+1)}_t(\ell, m)$ is the transition probability
of the $S_{p+1}$-valued Markov chain $\bb X^{(p+1)}_t$.

\begin{proof}[Proof of Proposition \ref{l16}]

Suppose that $y\in \Delta_{p+1}$ and fix $t>0$, $\epsilon>0$. Recall
the definition of $t_\epsilon$ introduced in \eqref{47}. By the
Markov property,
\begin{equation*}
\begin{aligned}
\mb P^n_{\! x} \big[\, X_{t\theta^{(p+1)}_n} \,=\, y \,\big] \; & =\;
\sum_{z\in V} \mb P^n_{\! x} \big[\, X_{t\theta^{(p+1)}_n -
t_\epsilon \theta^{(p)}_n}
\,=\, z \,\big] \, 
\mb P^n_{\! z} \big[\, X_{t_\epsilon \theta^{(p)}_n} \,=\, y \,\big] \\
\; & \le\; \max_{z\in V} \mb P^n_{\! z}
\big[\, X_{t_\epsilon \theta^{(p)}_n} \,=\, y \,\big]
\;.
\end{aligned}
\end{equation*}
By \eqref{47}, this maximum is bounded by $\epsilon + o_n(1)$, so that
\begin{equation*}
\lim_{n\to\infty} \mb P^n_{\! x} \big[\, X_{t\theta^{(p+1)}_n} \,=\, y
\,\big] \;=\;
0\;,
\end{equation*}
as claimed.

Suppose that $y\in \ms V^{(p+1)}_m$ for some $m\in S_{p+1}$ and fix
$t>0$, $\epsilon>0$. Recall the definition of $\Pi_p$, introduced in
\eqref{27b}. Choose $t_\epsilon$ large enough for
\begin{equation}
\label{48}
\max_{z,z'\in V} \Big|\, \lim_{n\to\infty} p^{(n)}_{t\, \theta^{(p)}_n}
(z, z') \,-\,  \Pi_p(z,z') \,\Big| 
\;\le \;  \epsilon
\end{equation}
for all $t\ge t_\epsilon$.

By the Markov property, as $\theta^{(p)}_n \prec \theta^{(p+1)}_n$, 
\begin{equation*}
\begin{aligned}
& \mb P^n_{\! x} \big[\, X_{t\theta^{(p+1)}_n} \,=\, y \,\big] \\
&\quad \;=\;
\sum_{z, z' \in V} \mb P^n_{\! x} \big[\, X_{t_\epsilon\, \theta^{(p)}_n} \,=\, z \,\big]\
\mb P^n_{\! z} \big[\, X_{t\theta^{(p+1)}_n - 2 t_\epsilon\, \theta^{(p)}_n}
\,=\, z' \,\big] \, 
\mb P^n_{\! z'} \big[\, X_{t_\epsilon\, \theta^{(p)}_n} \,=\, y \,\big]\;.
\end{aligned}
\end{equation*}
By \eqref{48} and \eqref{66}, this expression is equal to
\begin{equation*}
\sum_{z' \in V} \sum_{\ell \in S_{p+1}} \sum_{z \in \ms V^{(p+1)}_\ell} \Pi_p(x,z) \,
\mb P^n_{\! z} \big[\, X_{t\theta^{(p+1)}_n - 2 t_\epsilon\,
\theta^{(p)}_n} \, \,=\, z' \,\big] \, \Pi_p(z',y)   \;+\; O(\epsilon) \;+\; o_n(1) \;.
\end{equation*}
The first part of the proof permits to restrict the first sum to
$z'\in \ms V^{(p+1)}$. Since $y\in \ms V^{(p+1)}_m$, by \eqref{68} we
may further restrict the sum to $z'\in \ms V^{(p+1)}_m$. At this
point, by \eqref{69}, we may replace $\Pi_p (z',y)$ by
$\pi^{(p+1)}_m(y)$. Hence, the previous sum is equal to
\begin{equation*}
\pi^{(p+1)}_m(y) \, \sum_{\ell \in S_{p+1}} \sum_{z \in \ms V^{(p+1)}_\ell} \Pi_p(x,z) \,
\mb P^n_{\! z} \big[\, X_{t\theta^{(p+1)}_n - 2 t_\epsilon\,
\theta^{(p)}_n} \, \,\in \, \ms V^{(p+1)}_m \,\big] 
\;+\; O(\epsilon) \;+\; o_n(1) \;,
\end{equation*}
where we summed over $z' \in \ms V^{(p+1)}_m$. By \eqref{37b}, as
$n\to\infty$, this expression converges to
\begin{equation*}
\sum_{\ell \in S_{p+1}} \sum_{z \in \ms V^{(p+1)}_\ell} \Pi_p(x,z) \,
p^{(p+1)}_t(\ell, m) \, \pi^{(p+1)}_m(y) \;+\; O(\epsilon)\;.
\end{equation*}
By the definition \eqref{27b} of $\Pi_p$ and since the measure
$\pi^{(p+1)}_k (\,\cdot\,)$, $k\in S_{p+1}$, is supported on
$\ms V^{(p+1)}_k$, the previous expression is equal to
\begin{equation*}
\sum_{\ell \in S_{p+1}}  \mf a^{(p)} (x, \ell) \,
p^{(p+1)}_t(\ell, m) \, \pi^{(p+1)}_m(y) \;+\; O(\epsilon) \;,
\end{equation*}
as claimed.
\end{proof}

\subsection*{Proof of \eqref{70}}

Recall that $\theta^{(\mf q +1)}_n \equiv +\infty$, and fix a
sequence $\beta_n$ such that
$\theta^{(\mf q)}_n \prec \beta_n \prec \theta^{(\mf q+1)}_n$. Since
$\pi_n$ is the stationary state,
\begin{equation*}
\pi_n(\Delta_{\mf q+1}) \;=\; \sum_{x\in V} \pi_n(x)\,
\mb P^n_{\! x} \big[\,
X_{\beta_n} \in \Delta_{\mf q+1}\, \big] \;\le\;
\max_{x\in V} \mb P^n_{\! x} \big[\,
X_{\beta_n} \in \Delta_{\mf q+1}\, \big]\;.
\end{equation*}
By the tree construction, $S_{\mf q+1}$ is a singleton and there is
only one measure at step $\mf q+1$, the measure $\pi^{(\mf q+1)}_1$
which is concentrated on $\ms V^{(\mf q+1)}_1 = \ms V^{(\mf
q+1)}$. Since $\pi^{(\mf q+1)}_1 (\Delta_{\mf q+1}) =0 $, by
\eqref{27b}, and the previous displayed equation,
\begin{equation*}
\limsup_{n\to \infty} \pi_n(\Delta_{\mf q+1}) \;\le\;
\pi^{(\mf q+1)}_1 (\Delta_{\mf q+1}) =0\;.
\end{equation*}

It follows from the previous estimate that
$\lim_{n\to \infty} \pi_n(\ms V^{(\mf q+1)}) =1$. Hence, by
\eqref{58}, for all $x\in \ms V^{(\mf q+1)}$, 
\begin{equation*}
\lim_{n\to\infty} \pi_n(x)
\;\;\text{exists and belongs to $(0,1]$}\; .
\end{equation*}

\section{Proof of Proposition \ref{p3}}
\label{sec5}

The proof is divided in several lemmata.  We start with the asymptotic
behavior of the stationary states $\pi_n$.

\begin{lemma}
\label{l32}
For all $j\in S_1$, $x\in \ms V_j$, 
\begin{equation*}
\lim_{n\to\infty} \frac{\pi_n(x)}{\pi_n(\ms V_j)}
\;=\; \pi^\sharp (x) \;>\; 0\;.
\end{equation*}
\end{lemma}

\begin{proof}
Fix $j\in S_1$. By \eqref{58}, the limit $\pi_n(x)/\pi_n(\ms V_j)$
exists for all $x\in \ms V_j$ and is strictly positive. It remains to
show that it is equal to $\pi^\sharp (x)$. Denote the limit by
$m(x)$. Since $\pi_n$ is a stationary state, for all $x\in \ms V_j$,
\begin{equation*}
\sum_{y\in V} \pi_n(x) \, R_n(x,y) \;=\;
\sum_{y\in V} \pi_n(y) \, R_n(y,x) \; \ge \;
\sum_{y\in \ms V_j} \pi_n(y) \, R_n(y,x) \;.
\end{equation*}
As $\ms V_j$ is a closed irreducible class for
the chain $\bb X_t$, dividing by $\pi_n(\ms V_j)$ and passing to
the limit yields that
\begin{equation*}
\sum_{y\in \ms V_j} m(x) \, \bb R_0 (x,y) 
\; \ge \; \sum_{y\in \ms V_j} m(y) \, \bb R_0(y,x) \;.
\end{equation*}
Summing over $x\in \ms V_j$ shows that this inequality must be an
identity for all $x\in \ms V_j$. Therefore, $m$ is a stationary state
for the chain $\bb X_t$ on $\ms V_j$ what implies that $m=\pi^\sharp$, as
claimed. 
\end{proof}

\begin{lemma}
\label{l33}
Fix $1\le p\le \mf q$. For all $m\in S_{p+1}$, $j\in \mf R^{(p)}_m$, 
\begin{equation*}
\lim_{n\to\infty} \frac{\pi_n(\ms V^{(p)}_j)}{\pi_n(\ms V^{(p+1)}_m)}
\;=\; M^{(p)}_m(j) \;.
\end{equation*}
\end{lemma}

\begin{proof}
Fix $1\le p\le \mf q$ and $m\in S_{p+1}$. Consider the sequence of
measures on $\mf R^{(p)}_m$ defined by
$ m_n(j) = \pi_n(\ms V^{(p)}_j)/\pi_n(\ms V^{(p+1)}_m)$. By
\eqref{58}, it converges to a limiting measure, denoted by $m(j)$.

By \cite[Proposition 6.3]{bl2},
$\pi_n(\,\cdot\,)/\pi_n(\ms V^{(p)})$ is the stationary state of the
chain $Y^{n,p}_t$, the trace of $X^{(n)}_t$ on $\ms V^{(p)}$. Hence,
for all $j\in \mf R^{(p)}_m$, $x\in \ms V^{(p)}_j$,
\begin{equation*}
\sum_{y\in \ms V^{(p)}} \pi_n(x) \, R^{(p)}_n(x,y) \;=\;
\sum_{y\in \ms V^{(p)}} \pi_n(y) \, R^{(p)}_n(y,x) \; \ge \;
\sum_{k\in \mf R^{(p)}_m} \sum_{y\in \ms V^{(p)}_k} \pi_n(y) \,
R^{(p)}_n(y,x) \;. 
\end{equation*}
Sum over all $x\in \ms V^{(p)}_j$ to get that
\begin{equation*}
\sum_{k\in S_p} \sum_{x\in \ms V^{(p)}_j}
\pi_n(x) \, R^{(p)}_n(x,\ms V^{(p)}_k) \; \ge \;
\sum_{k\in \mf R^{(p)}_m} \sum_{y\in \ms V^{(p)}_k} \pi_n(y) \,
R^{(p)}_n(y,\ms V^{(p)}_j) \;,
\end{equation*}
where
$R^{(p)}_n(z,\ms V^{(p)}_\ell) = \sum_{w\in \ms V^{(p)}_\ell}
R^{(p)}_n(z,w) $. Remove on both sides of this inequality the case
$k=j$. By \eqref{41}, this new expression divided by
$\pi_n(\ms V^{(p+1)}_m)$ is equal to
\begin{equation*}
\frac{\pi_n(\ms V^{(p)}_j)}{\pi_n(\ms V^{(p+1)}_m)}
\sum_{k\in S_p\setminus\{j\}} r^{(p)}_n (j,k)  \; \ge \;
\sum_{k\in \mf R^{(p)}_m \setminus\{j\}}
\frac{\pi_n(\ms V^{(p)}_k)}{\pi_n(\ms V^{(p+1)}_m)}
\, r^{(p)}_n (k,j) \;. 
\end{equation*}
By the assumption on the measure $m_n$ and by \eqref{34}, as
$n\to\infty$, this expression multiplied by $\theta^{(p)}_n$ on both
sides converges to
\begin{equation*}
m(j) \sum_{k\in S_p\setminus\{j\}} r^{(p)} (j,k)  \; \ge \;
\sum_{k\in \mf R^{(p)}_m \setminus\{j\}} m(k) \, r^{(p)} (k,j) \;. 
\end{equation*}
Since $\mf R^{(p)}_m$ is a closed irreducible class for the chain
$\bb X^{(p)}_t$, $r^{(p)} (j,k) =0$ for all $k\not\in \mf R^{(p)}_m$,
and the first sum can be restricted to this later set.  Summing over
$j$ yields that this inequality must be an identity for all
$j$. Therefore, $m$ is a stationary state for the Markov chain
$\bb X^{(p)}_t$ restricted to $\mf R^{(p)}_m$. By ergodicity,
$m= M^{(p)}_m$, as claimed.
\end{proof}

\begin{corollary}
\label{l34}
Fix $1\le p\le \mf q+1$. For all $j\in S_{p}$, $x\in \ms V^{(p)}_j$, 
\begin{equation*}
\lim_{n\to\infty} \frac{\pi_n(x)}{\pi_n(\ms V^{(p)}_j)}
\;=\; \pi^{(p)}_j(x) \;.
\end{equation*}
\end{corollary}

\begin{proof}
The proof is performed by induction. Lemma \ref{l32} covers the case
$p=1$. Assume that this corollary has been proven for all
$1\le p <p_0$, where $p_0\le \mf q+1$. Fix $j\in S_{p_0}$ and
$x\in \ms V^{(p_0)}_j$.  By construction of $\ms V^{(p_0)}_j$, there
exists $k\in S_{p_0-1}$ such that
$x\in \ms V^{(p_0-1)}_k \subset \ms V^{(p_0)}_j$. We can write
\begin{equation*}
\frac{\pi_n(x)}{\pi_n(\ms V^{(p_0)}_j)} \;=\;
\frac{\pi_n(x)}{\pi_n(\ms V^{(p_0-1)}_k)} \;
\frac{\pi_n(\ms V^{(p_0-1)}_k) }{\pi_n(\ms V^{(p_0)}_j)}\; \cdot
\end{equation*}
By Lemma \ref{l33} and the induction assumption, as $n\to\infty$, this
expression converges to 
\begin{equation*}
\pi^{(p_0-1)}_k (x)\,  M^{(p_0-1)}_j(k)\;.
\end{equation*}
By \eqref{80}, this expression is equal to $\pi^{(p_0)}_j (x)$ as claimed.
\end{proof}

We turn to the absorbing probabilities. We first consider the case
where the state belongs to the valley.

\begin{lemma}
\label{l21}
For all $1\le p\le \mf q+1$, $j\in S_p$ and $x\in \ms V^{(p)}_j$, $\mf
a^{(p-1)} (x,j)=1$.
\end{lemma}

\begin{proof}
The proof is by induction on $p$. Fix $j\in S_1$ and $x\in \ms V_j$.
By \eqref{33}, $\mf a^{(0)} (x,j)=1$ because $\ms V_j$ is a closed
irreducible class for $\bb X_t$ and $x$ belongs to $\ms V_j$.

Suppose that the results has been proved for $p-1$. This means that if
$j\in S_{p}$ and $x\in \ms V^{(p)}_j$, then
$\mf a^{(p-1)} (x,j) =1$.  As $\mf a^{(p-1)} (x\,,\,\cdot\,)$ is a
probability measure on $S_{p}$, $\mf a^{(p-1)} (x,k) = 0$ for all
$k\in S_{p} \setminus \{j\}$.

Fix $m\in S_{p+1}$ and $x\in \ms V^{(p+1)}_m$. As
$ \ms V^{(p+1)}_m = \cup_{j\in \mf R^{(p)}_m} \ms V^{(p)}_j$,
$x\in \ms V^{(p)}_j$ for some $j\in \mf R^{(p)}_m$. By \eqref{33b},
and since, by the induction hypothesis, $\mf a^{(p-1)} (x, k) =
\delta_{j,k}$, 
\begin{equation*}
\mf a^{(p)} (x,m)
\;=\; \sum_{k\in S_p} \mf a^{(p-1)} (x, k) \, \mf A^{(p)} (k, m)
\;=\; \mf A^{(p)} (j, m)\;.
\end{equation*}
As $j\in \mf R^{(p)}_m$ and $\mf R^{(p)}_m$ is a closed irreducible
class for $\bb X^{(p)}_t$, by the definition \eqref{63} of $\mf
A^{(p)}$, $\mf A^{(p)} (j, m) =1$, which completes the proof of the
lemma. 
\end{proof}

It follows from this lemma and from \eqref{27b} that for all
$1\le p\le \mf q+1$, $j\in S_p$, $x\in \ms V^{(p)}_j$ and sequences
$\beta_n$ such that
$\theta^{(p-1)}_n \,\prec\, \beta_n\,\prec\, \theta^{(p)}_n$
\begin{equation}
\label{86}
\lim_{n\to\infty} p^{(n)}_{\beta_n} (x,\,\cdot\, ) \;=\;
\pi^{(p)}_j (\,\cdot\, )\;.
\end{equation}

Lemma \ref{l21} provides a formula for $a^{(p-1)} (x\,,\, \cdot\,)$
when $x\in \ms V^{(p)}$. Lemma \ref{l22} completes the
characterisation of $a^{(p-1)} (x\,,\, \cdot\,)$. The proof of this
result relies on the following bound.

We claim that for all $a>0$, $1\le p\le \mf q+1$,
$x\not\in \ms V^{(p)}$ and sequence $\beta_n$ such that
$\theta^{(p-1)}_n \,\prec\, \beta_n\,\prec\, \theta^{(p)}_n$,
\begin{equation}
\label{87}
\lim_{n\to\infty} \max_{x\in V} \mb P^n_{\! x} \big[\,
H_{\ms V^{(p)}} \,>\, a\, \beta_n
\,\big] \; = \; 0 \;.
\end{equation}

If $x\in \ms V^{(p)}$, there is nothing to prove. Fix
$x\not\in \ms V^{(p)}$ and observe that
$\{ H_{\ms V^{(p)}} \,>\, a\,\beta_n \} \subset \, \int_{[0,a\beta_n]}
\chi_{\Delta_p}(X^n_s)\, ds \ge a\, \beta_n$.  Hence, by Chebyshev
inequality,
\begin{equation*}
\mb P^n_{\! x} \big[\, H_{\ms V^{(p)}} \,>\, a\, \beta_n
\,\big] \;\le\; \mb P^n_{\! x} \big[\,
\int_0^{a\beta_n}
\chi_{\Delta_p}(X^n_s)\, ds \ge a\, \beta_n 
\,\big] \;\le\; \frac{1}{a}\, \int_0^{a}
\mb E^n_x \big[\, \chi_{\Delta_p}(X^n_{s\beta_n})\, ds 
\,\big] \;.
\end{equation*}
The last term can be written as
\begin{equation*}
\sum_{z\in \Delta_p} \frac{1}{a}\, \int_0^{a}
p^{(n)}_{s\beta_n}(x,z)\, ds \;.
\end{equation*}
For each fixed $0<s<a$ the sequence $s\beta_n$ satisfies the
hypotheses of Theorem \ref{mt0}.(a). Hence, since
$\pi^{(p)}_j(\Delta_p) =0$ for all $j\in S_p$,
$p^{(n)}_{s\beta_n}(x,z) \to 0$. Therefore, by the dominated
convergence theorem, the previous expression vanishes, which proves
claim \eqref{87}.

\begin{lemma}
\label{l22}
For all $1\le p\le \mf q+1$, $j\in S_p$, $x\in V$,
\begin{equation*}
\mf a^{(p-1)} (x,j) \;=\; \lim_{n\to\infty} \mb P^n_{\! x}
\big[\, H_{\ms V^{(p)}_j} \,<\,  H_{\breve{\ms V}^{(p)}_j}\,\big]\;. 
\end{equation*}
\end{lemma}

\begin{proof}
Fix $1\le p\le \mf q+1$ and $j\in S_p$. If $x\in \ms V^{(p)}$, this
result follows from Lemma \ref{l21}. Assume that
$x\not\in \ms V^{(p)}$ and fix a sequence $\beta_n$ such that
$\theta^{(p-1)}_n \,\prec\, \beta_n\,\prec\, \theta^{(p)}_n$.  On the
one hand, by \eqref{27b},
\begin{equation*}
\lim_{n\to\infty} \sum_{y\in \ms V^{(p)}_j}
p^{(n)}_{\beta_n} (x,y) \;=\;
\mf a^{(p-1)} (x,j)\;.
\end{equation*}
On the other hand,
\begin{equation*}
\sum_{y\in \ms V^{(p)}_j} p^{(n)}_{\beta_n} (x,y) \;=\;
\mb P^n_{\! x} \big[\, X_{\beta_n} \, \in\,  \ms V^{(p)}_j \,\big]
\;=\; \sum_{k\in S_p} \mb P^n_{\! x} \big[\,
H_{\ms V^{(p)}_k} \,=\,  H_{\ms V^{(p)}}
\,,\, X_{\beta_n} \, \in\,  \ms V^{(p)}_j \,\big] \;.
\end{equation*}
Fix $k\in S_p$ and $0<\epsilon<1$. By \eqref{87}, the previous
probability for the fixed $k$ is equal to
\begin{equation*}
\mb P^n_{\! x} \big[\,  H_{\ms V^{(p)}} \,<\, \epsilon\, \beta_n\,,\,
H_{\ms V^{(p)}_k} \,=\,  H_{\ms V^{(p)}} \,,\,
X_{\beta_n} \, \in\,  \ms V^{(p)}_j  \,\big] \; + \; o_n(1)\;.
\end{equation*}
By the strong Markov property at $H_{\ms V^{(p)}}$, the previous
probability is equal to 
\begin{equation*}
\mb E^n_x \Big[\,  H_{\ms V^{(p)}} \,<\, \epsilon\, \beta_n\,,\,
H_{\ms V^{(p)}_k} \,=\,  H_{\ms V^{(p)}} \,,\,
\mb P^n_{X(H_{\ms V^{(p)}})}  \big[\, X_{\beta_n - H_{\ms V^{(p)}}}
\, \in\,  \ms V^{(p)}_j \,\big] \, \Big] \;.
\end{equation*}
In this formula, one computes the probability
$\mb P^n_{X(H_{\ms V^{(p)}})} [\, X_{\beta_n - t} \, \in\, \ms
V^{(p)}_j \,]$ and then replace $t$ by $H_{\ms V^{(p)}}$.  
After the proof of this lemma, we show that for all $z \in V$
\begin{equation}
\label{f13}
\sup_{t\le \epsilon \beta_n} \, \mb P^n_{\! z}  \big[\, X_{\beta_n - t}
\, \in\,  \ms V^{(p)}_j \,\big]   \;\le\;
\max_{y\in \ms V^{(p)}_j} \mb P^n_{\! y}  \big[\,
H_{\breve{\ms V}^{(p)}_j} < \epsilon\, \beta_n \,\big]
\;+\; 
\mb P^n_{\! z}  \big[\, X_{\beta_n}
\, \in\,  \ms V^{(p)}_j \cup \Delta_p \,\big]\;.
\end{equation}
By \eqref{88},
\begin{equation*}
\lim_{n\to\infty} \max_{y\in \ms V^{(p)}_j} \mb P^n_{\! y}  \big[\,
H_{\breve{\ms V}^{(p)}_j} < \epsilon\, \beta_n \,\big] \;=\; 0\;.
\end{equation*}

Therefore, up to this point, we proved that
\begin{equation*}
\begin{aligned}
& \mf a^{(p-1)} (x,j) \;\le\; \\
& \quad \sum_{k\in S_p} \liminf_{n\to\infty}
\mb E^n_x \Big[\,  H_{\ms V^{(p)}} \,<\, \epsilon\, \beta_n\,,\,
H_{\ms V^{(p)}_k} \,=\,  H_{\ms V^{(p)}} \,,\,
\mb P^n_{\! X(H_{\ms V^{(p)}})}  \big[\, X_{\beta_n}
\, \in\,  \ms V^{(p)}_j \cup \Delta_p \,\big] \, \Big] \;.
\end{aligned}
\end{equation*}
By \eqref{27b} and Lemma \ref{l21}, if $k\not =j$ the previous
expectation vanishes as $n\to\infty$. If $k=j$ by the same reasons,
the probability inside the expectation converges to $1$ as
$n\to\infty$. Hence,
\begin{equation*}
\mf a^{(p-1)} (x,j) \;\le\;
\liminf_{n\to \infty} \mb P^n_{\! x} \big[\,
H_{\ms V^{(p)}} \,<\, \epsilon\, \beta_n\,,\,
H_{\ms V^{(p)}_j} \,=\,  H_{\ms V^{(p)}} \,\big] \;.
\end{equation*}
Therefore, by \eqref{87}, for all $j\in S_p$,
\begin{equation*}
\mf a^{(p-1)} (x,j) \;\le\;
\liminf_{n\to \infty} \mb P^n_{\! x} \big[\,
H_{\ms V^{(p)}_j} \,=\,  H_{\ms V^{(p)}} \,\big] \;.
\end{equation*}

The previous inequality implies that equality holds for all
$j\in S_p$. Indeed, assume that strict inequality holds for some
$j\in S_p$. Then, as $\mf a^{(p-1)} (x\,,\, \cdot\,)$ is a probability
measure on $S_p$,
\begin{equation*}
\begin{aligned}
1\;=\; \sum_{j\in S_p} \mf a^{(p-1)} (x,j) \; & <\;
\sum_{j\in S_p} \liminf_{n\to \infty} \mb P^n_{\! x} \big[\,
H_{\ms V^{(p)}_j} \,=\,  H_{\ms V^{(p)}} \,\big] \\
& \le \; \liminf_{n\to \infty}
\sum_{j\in S_p}  \mb P^n_{\! x} \big[\,
H_{\ms V^{(p)}_j} \,=\,  H_{\ms V^{(p)}} \,\big] \;=\; 1\;,
\end{aligned}
\end{equation*}
which is a contradiction.
\end{proof}

We turn to the proof of \eqref{f13}. Inserting the event
$\{X_{\beta_n} \, \in\, \ms V^{(p)}_j \cup \Delta_p\}$ and its
complement inside the probability appearing on the left-hand side of
\eqref{f13} yields that this probability is bounded by
\begin{equation*}
\begin{aligned}
& \mb P^n_{\! z}  \big[\, X_{\beta_n - t}
\, \in\,  \ms V^{(p)}_j \,,\, X_{\beta_n}
\, \not  \in\,  \ms V^{(p)}_j \cup \Delta_p \,\big]
\;+\; 
\mb P^n_{\! z}  \big[\, X_{\beta_n}
\, \in\,  \ms V^{(p)}_j \cup \Delta_p \,\big] \\
&\quad \le\;
\max_{y\in \ms V^{(p)}_j} \mb P^n_{\! y}  \big[\, X_{t}
\, \not  \in\,  \ms V^{(p)}_j \cup \Delta_p \,\big]
\;+\; 
\mb P^n_{\! z}  \big[\, X_{\beta_n}
\, \in\,  \ms V^{(p)}_j \cup \Delta_p \,\big] \;,
\end{aligned}
\end{equation*}
where we used the Markov property to estimate the first by the second
line. As $t\le \epsilon\, \beta_n$, this expression is clearly bounded by
\begin{equation*}
\max_{y\in \ms V^{(p)}_j} \mb P^n_{\! y}  \big[\,
H_{\breve{\ms V}^{(p)}_j} < \epsilon\, \beta_n \,\big]
\;+\; 
\mb P^n_{\! z}  \big[\, X_{\beta_n}
\, \in\,  \ms V^{(p)}_j \cup \Delta_p \,\big]\;,
\end{equation*}
as claimed in \eqref{f13}.

To complete the proof of Lemma \ref{l22}, it remains to show that for
all $1\le p\le \mf q$, $j\in S_p$,
\begin{equation}
\label{88}
\lim_{a\to 0} \limsup_{n\to \infty}
\max_{x\in \ms V^{(p)}_j} \mb P^n_{\! x} \big[\,
H_{\breve{\ms V}^{(p)}_j} \,<\, a\, \theta^{(p)}_n \,\big] \;=\; 0\;.
\end{equation}

Fix $1\le p\le \mf q$, $j\in S_p$, $x\in \ms V^{(p)}_j$.  Recall that
$Y^{n,p}$ represents the trace of the process $X^n_t$ on
$\ms V^{(p)}$, and that $\Phi_p : \ms V^{(p)} \to S_p$ stands for the
projection which sends $x\in \ms V^{(p)}_j$ to $j$. By \cite[Theorems
2.1 and 2.12]{lx16}, under $\mb P^n_{\! x}$, the process
$\Phi_p (Y^{n,p}_{t \theta^{(p)}_n})$ converges weakly in the Skorohod
topology to $\bb X^{(p)}_t$. In particular,
\begin{equation*}
\lim_{a\to 0} \limsup_{n\to \infty}
\mb P^n_{\! x} \big[\, H_{\breve{\ms V}^{(p)}_j}(Y^{n,p})
\,<\, a\, \theta^{(p)}_n \,\big] \;=\; 0\;.
\end{equation*}
In this formula, $\color{blue} H_{\breve{\ms V}^{(p)}_j}(Y^{n,p})$
stands for the hitting time of $\breve{\ms V}^{(p)}_j$ for the process
$Y^{n,p}_t$. Since
$H_{\breve{\ms V}^{(p)}_j}(Y^{n,p}) \le H_{\breve{\ms V}^{(p)}_j}$,
assertion \eqref{88} follows from this last result. \smallskip 

We complete this section with a consequence of Lemma \ref{l22}.
Recall from \eqref{61} that $\bb Q^{(p)}_k$ stands for the measure on
$D(\bb R_+, S_p)$ induced by the process $\bb X^{(p)}_t$ starting from
$k\in S_p$.

\begin{lemma}
\label{l37}
For all $2\le p\le \mf q$, $i\in S_{p-1}$ and $x\in \ms V^{(p-1)}_i$,
\begin{equation*}
\mf a^{(p-1)} (x,j) \;=\; \bb Q^{(p-1)}_i
\big[\, H_{\mf R^{(p)}_j} \,<\,  H_{\breve{\mf R}^{(p)}_j}\,\big]\;,
\quad j\in S_p\;,
\end{equation*}
where
${\color{blue} \breve{\mf R}^{(p)}_j} = \cup_{k\in S_p \setminus
\{j\}} \mf R^{(p)}_k$.
\end{lemma}

\begin{proof}
Recall that $Y^{n,p-1}_t$ represents the trace of $X^{(n)}_t$
on $\ms V^{(p-1)}$. By \cite[Theorem 2.1]{bl4}, under the measure
$\mb P^n_{\! x}$ the process
${\color{blue} \bb X^{n,p-1}_t} :=
\Phi_{p-1}(Y^{n,p-1}_{t\theta^{(p-1)}_n})$ converges weakly in the
Skorohod topology to the $S_{p-1}$-valued process $\bb X^{(p-1)}_t$
introduced below \eqref{34}.

Clearly, under the measure $\mb P^n_{\! x}$,
\begin{equation*}
\big \{H_{\ms V^{(p)}_j} (X^n) \,<\,  H_{\breve{\ms V}^{(p)}_j}
(X^n)\, \big\} \;=\;
\big\{H_{\ms V^{(p)}_j} (Y^{n,p-1}) \,<\,  H_{\breve{\ms V}^{(p)}_j}
(Y^{n,p-1})\, \big\}\;.
\end{equation*}
This identity asserts that the process $X^{(n)}$ hits the set
$\ms V^{(p)}_j$ before the set $\breve{\ms V}^{(p)}_j$ if and only if
this happens to the trace process $Y^{n,p-1}$. By projecting the
process $Y^{n,p-1}$ with $\Phi_{p-1}$, the last event becomes
\begin{equation*}
\big\{H_{\mf R^{(p)}_j} (\bb X^{n,p-1}) \,<\,  H_{\breve{\mf R}^{(p)}_j}
(\bb X^{n,p-1})\, \big\}\;,
\end{equation*}
Therefore, by Lemma \ref{l22}, for $2\le p \le \mf q+1$, $j\in S_p$
\begin{equation*}
\begin{aligned}
\mf a^{(p-1)} (x,j) \; & =\; \lim_{n\to\infty} \mb P^n_{\! x}
\big[\, H_{\ms V^{(p)}_j} \,<\,  H_{\breve{\ms V}^{(p)}_j}\,\big]
\\
& =\; \lim_{n\to\infty} \mb P^n_{\! x}
\big[\, H_{\mf R^{(p)}_j} (\bb X^{n,p-1}) \,<\,  H_{\breve{\mf R}^{(p)}_j}
(\bb X^{n,p-1})\,\big]\;.
\end{aligned}
\end{equation*}
As $\bb X^{n,p-1}$ converges weakly in the Skorohod
topology to $\bb X^{(p-1)}$,
\begin{equation*}
\lim_{n\to\infty} \mb P^n_{\! x} \big[\,
H_{\mf R^{(p)}_j} (\bb X^{n,p-1}) \,<\,  H_{\breve{\mf R}^{(p)}_j}
(\bb X^{n,p-1})\, \big] \;=\;
\bb Q^{(p-1)}_{i} \big[\,
H_{\mf R^{(p)}_j}  \,<\,  H_{\breve{\mf R}^{(p)}_j}\, \big] \;,
\end{equation*}
as claimed.
\end{proof}

\section{Preliminary estimates}
\label{sec6}

In this section, we present some estimates needed in the proof of
Theorem \ref{mt3}. We assume throughout it that the process is
reversible. We start with some estimates on the stationary state, now
assumed to be reversible.

Fix $x\in \Delta$. As $x$ is a transient state for the chain
$\bb X_t$, it is eventually absorbed by a closed irreducible class
$\ms V_k$, $k\in S_1$. Fix $j \in S_1$ such that $\mf a^{(0)}(x,j)>0$,
where $\mf a^{(0)}(x,j)$ has been introduced in \eqref{92}.  We claim
that
\begin{equation}
\label{94}
\pi_n(x) \;\prec \; \pi_n(\ms V_j) \;. 
\end{equation}

Indeed, as $\mf a^{(0)}(x,j)>0$, there exists a sequence
$x=x_0, \dots, x_\ell$ of elements of $V$ such that
$\bb R_0(x_i,x_{i+1})>0$, $x_i\in \Delta$, $0\le i <\ell$,
$x_\ell\in \ms V_j$. By reversibility,
\begin{equation*}
\frac{\pi_n(x_i)}{\pi_n(x_{i+1})} \;=\;
\frac{R_n(x_{i+1},x_i)}{R_n(x_i,x_{i+1})}\; \cdot
\end{equation*}
Since $R_n(x_i,x_{i+1}) \to \bb R_0(x_i,x_{i+1})>0$, by \eqref{01},
$\pi_n(x_i) \preceq \pi_n(x_{i+1})$. As $x_{\ell-1}\in \Delta$,
$x_\ell\in \ms V_j$,
$R_n(x_{\ell}, x_{\ell-1}) \to \bb R_0(x_{\ell}, x_{\ell-1})=0$, so
that $\pi_n(x_{\ell-1}) \prec \pi_n(x_{\ell})$, which proves claim
\eqref{94}.

Next result extends this estimate

\begin{lemma}
\label{l25}
Fix $2\le p \le \mf q$, $j\in S_p$, $x\in \ms V^{(p-1)} \setminus \ms
V^{(p)}$. If $\mf a^{(p-1)}(x,j)>0$, then, $\pi_n(x) \;\prec \; \pi_n(\ms
V^{(p)}_j)$. 
\end{lemma}

\begin{proof}
The proof is similar to the one presented to derive \eqref{94}.
Suppose that $x\in \ms V^{(p-1)}_i\setminus \ms V^{(p)}$ for
$i\in S_{p-1}$. As $x$ does not belong to $\ms V^{(p)}$,
$i\in \mf T_{p-1}$.

As $\mf a^{(p-1)}(x,j)>0$, by Lemma \ref{l37}, there exists a sequence
$i=i_0, \dots, i_\ell$ of elements of $S_{p-1}$ such that
$r^{(p-1)}(i_a,i_{a+1})>0$, $i_a\in \mf T_{p-1}$, $0\le a <\ell$,
$i_\ell\in \mf R^{(p-1)}_j$. By reversibility, \eqref{40} and
\eqref{41},
\begin{equation}
\label{97}
\frac{\pi_n(\ms V^{(p-1)}_{i_a})}{\pi_n(\ms V^{(p-1)}_{i_{a+1}})} \;=\;
\frac{r^{(p-1)}_n(i_{a+1},i_a)}{r^{(p-1)}_n(i_a,i_{a+1})}\; \cdot
\end{equation}
Since
$\theta^{(p-1)}_n\, r^{(p-1)}_n(i_a,i_{a+1}) \to
r^{(p-1)}(i_a,i_{a+1})>0$, by \eqref{34},
$\pi_n(\ms V^{(p-1)}_{i_a}) \preceq \pi_n(\ms
V^{(p-1)}_{i_{a+1}})$. As $i_{\ell-1}\in \mf T_{p-1}$,
$i_\ell\in \mf R^{(p-1)}_j$,
$\theta^{(p-1)}_n\, r^{(p)}_n(i_{\ell}, i_{\ell-1}) \to
r^{(p)}(i_{\ell}, i_{\ell-1})=0$, so that
$\pi_n(\ms V^{(p-1)}_{i_{\ell-1}}) \prec \pi_n(\ms
V^{(p-1)}_{i_\ell})$. Since $i_\ell\in \mf R^{(p-1)}_j$,
$\ms V^{(p-1)}_{i_\ell} \subset \ms V^{(p)}_{j}$, and the lemma is
proved.
\end{proof}

\begin{corollary}
\label{l26}
Fix $2\le p \le \mf q$, $j\in S_p$, $x\in V \setminus \ms
V^{(p)}$. If $\mf a^{(p-1)}(x,j)>0$, then, $\pi_n(x) \;\prec \; \pi_n(\ms
V^{(p)}_j)$. 
\end{corollary}

\begin{proof}
Fix $x\in V \setminus \ms V^{(p)}$ and let $r(x)$ be the element $r$
of $\{1, \dots, p\}$ such that
$x\in \ms V^{(r-1)} \setminus \ms V^{(r)}$, where $\ms V^{(0)} = V$.
The proof is by induction on $r(x)$.

If $r(x) =p$, the assertion corresponds to the one of Lemma \ref{l25}.
Suppose that the corollary has been proved for
$y\in \ms V^{(r-1)} \setminus \ms V^{(r)}$ and all
$r \in \{s+1, \dots, p\}$, and fix
$x\in \ms V^{(s-1)} \setminus \ms V^{(s)}$. By the strong Markov
property at time $H_{\ms V^{(s)}}$,
\begin{equation*}
\begin{aligned}
& \mf a^{(p-1)} (x,j) \; =\; \lim_{n\to\infty} \mb P^n_{\! x}
\big[\, H_{\ms V^{(p)}_j} \,<\,  H_{\breve{\ms V}^{(p)}_j}\,\big] \\
& \quad  =\; \lim_{n\to\infty} \sum_{k\in S_s} \sum_{z\in \ms V^{(s)}_k}
\mb P^n_{\! x} \big[\, H_{\ms V^{(s)}_k} \,=\,  H_{\ms V^{(s)}} \,,\,
X (H_{\ms V^{(s)}}) \,=\, z \,\big]
\; \mb P^n_{\! z}
\big[\, H_{\ms V^{(p)}_j} \,<\,  H_{\breve{\ms V}^{(p)}_j} \,\big]
\;. 
\end{aligned}
\end{equation*}
The sum can be restricted to elements $k\in S_s$ and
$z\in \ms V^{(s)}_k$ such that $\mf a^{(s-1)}(x,k)>0$,
$\mf a^{(p-1)}(z,j)>0$. By Lemma \ref{l25}, 
$\pi_n(x) \, \prec \, \pi_n(\ms V^{(s)}_k)$ and by the induction
assumption, $\pi_n(z) \, \preceq \, \pi_n(\ms V^{(p)}_j)$. The previous
estimate may not be strict as it might happen that $z$ belongs to $\ms
V^{(p)}_j$. By \eqref{58}, $\pi_n(\ms V^{(s)}_k) \sim \pi_n(z)$ so
that $\pi_n(x) \, \prec \, \pi_n(\ms V^{(p)}_j)$, as claimed.
\end{proof}

\subsection*{Potential theory}

We turn to estimates involving the capacity.  Recall the definition of
comparable sequences introduced just before the main hypothesis
\eqref{mh}.  Let $c_n\colon E\to \bb R_+$ be given by
$\color{blue} c_n(x,y) := \pi_n(x) R_n(x,y)$ and note that $c_n$ is
symmetric. It follows from \eqref{mh} (cf. equation (2.5) in
\cite{bl4}) that the sequences $c_n(x,y)$ are comparable.  A
self-avoiding path $\gamma$ from $\ms A$ to $\ms B$, $\ms A$,
$\ms B \subset V$, $\ms A\cap \ms B = \varnothing$, is a sequence of
sites $(x_0, x_1, \dots, x_m)$ such that $x_0\in \ms A$,
$x_m\in \ms B$, $x_i \not = x_j$, $i\not =j$, $R_n(x_i,x_{i+1})>0$,
$0\le i <m$. Denote by $\color{bblue} \Gamma_{\ms A, \ms B}$ the set
of self-avoiding paths from $\ms A$ to $\ms B$ and let
\begin{equation*}
{\color{blue} c_n(\ms A, \ms B)} \;:=\; \max_{\gamma\in \Gamma_{\ms A, \ms B}}
c_n (\gamma)\;, \quad
{\color{blue} c_n (\gamma)}  \;:=\; \min_{0\le i<m } c_n(x_i,x_{i+1})\;.
\end{equation*}
Note that there might be more than one optimal path and that
$c_n(\{x\},\{y\}) \ge c_n(x,y)$, with possibly a strict inequality.
Next result is \cite[Lemma 4.1]{bl4}.

\begin{lemma}
\label{l23}
There exists a positive and finite constant $C_1$ such that
\begin{equation*}
C_1^{-1} \;\le\; \frac{\Cap_n(\ms A, \ms B)}{c_n(\ms A, \ms B)} \;\le\; C_1
\end{equation*}
for all $n\ge $ and non-empty, disjoint subsets $\ms A$, $\ms B$ of
$V$.
\end{lemma}

Fix two disjoint, non-empty subsets $\ms A$, $\ms B$ of $V$, and let
$h_{\ms A, \ms B}$ be the equilibrium potential between $\ms A$ and
$\ms B$:
\begin{equation*}
{\color{blue}  h_{\ms A, \ms B}(x) }
\;:=\;  \mb P^n_{\! x} [H_{\ms A} < H_{\ms B}]\;,
\quad x\in V\;.
\end{equation*}
Denote by $D_n(f)$ the Dirichlet form of a function $f:V\to \bb R$:
\begin{equation*}
{\color{blue} D_n(f)} \;:=\; \< \, f \,,\, (-\, \ms L_n) f \,\>_{\pi_n
}\;.
\end{equation*}
It is well known \cite[equation (B.7)]{lrev}, that
\begin{equation*}
\Cap_n(\ms A \,,\, \ms B) \;=\; D_n(h_{\ms A, \ms B})\;.
\end{equation*}

\begin{lemma}
\label{l27}
There exists a finite constant $C_0$, independent of $n$, such that
\begin{equation*}
h_{\ms A, \ms B}(x)^2 \;\le\; C_0\, \frac{\Cap_n (\ms A \,,\, \ms B)}
{\Cap_n (\{x\} \,,\, \ms B)} 
\end{equation*}
for all $x\not \in \ms A \cup \ms B$.
\end{lemma}

\begin{proof}
Let $h = _{\ms A, \ms B}$, and let $\gamma = (x=x_0, \dots, x_m)$ be a
self-avoiding path between $x$ and $\ms B$. Hence
$R_n(x_i,x_{i+1}) >0$, $x_i\not \in \ms B$, $0\le i<m$ and
$x_m\in \ms B$. As $x_m\in \ms B$, $h(x_m)=0$ so that
\begin{equation*}
h(x)^2 \;=\; (\,h(x_0) - h(x_m)\,)^2 \;\le\;
\sum_{i=0}^{m-1} c_n(x_i,x_{i+1}) \,[\, h(x_{i+1} - h(x_i)\,]^2
\sum_{i=0}^{m-1} \frac{1}{c_n(x_i,x_{i+1}) }\;\cdot
\end{equation*}
As the path is self-avoiding, this quantity is bounded by
\begin{equation*}
|E|\, D_n(h)\, \max_{0\le i<m} \frac{1}{c_n(x_i,x_{i+1})} \;=\;
|E|\, \Cap_n(\ms A \,,\, \ms B)\, \max_{0\le i<m}
\frac{1}{c_n(x_i,x_{i+1})} \; \cdot
\end{equation*}
Minimising over all possible paths $\gamma$ from $x$ to $\ms B$ yields
that
\begin{equation*}
h_{\ms A, \ms B}(x)^2 \;\le\;
|E|\, \Cap_n(\ms A \,,\, \ms B)\, 
\frac{1}{\max_\gamma \min_{0\le i<m} c_n(x_i,x_{i+1})}\;\cdot
\end{equation*}
The assertion of the lemma follows from Lemma \ref{l23}.
\end{proof}

\begin{lemma}
\label{l30}
Fix $1\le p \le \mf q$, and suppose that $r^{(p)}(j,k)>0$ for some
$j$, $k\in S_p$. Then,
\begin{equation*}
\liminf_{n\to\infty} \frac{\theta^{(p)}_n}{\pi_n(\ms V^{(p)}_j)}\, \Cap_n(\ms
V^{(p)}_j\,,\, \ms V^{(p)}_k) \;>\; 0 \;.
\end{equation*}
We do not exclude the possibility that this $\limsup$ is $+\infty$.
\end{lemma}

\begin{proof}
We argue by contradiction, proving that if the $\limsup$ vanishes than
$r^{(p)}(j,k)=0$, but we first derive a consequence of the positivity of
$r^{(p)}(j,k)$. 

Fix $x\in \ms V^{(p)}_j$.  The main result in \cite{bl4} states that
under the measure $\mb P^{n}_{\! x}$, the process
$\bb X^{n,p}_t = \Phi_{p}(Y^{n,p}_{t\theta^{(p)}_n})$ converges weakly
in the Skorohod topology to the $S_{p}$-valued process
$\bb X^{(p)}_t$. Hence, if $r^{(p)}(j,k)>0$, for every $a>0$,
\begin{equation}
\label{96}
\liminf_{n\to\infty} \mb P^{n}_{\! x} \big[ \, H_{\ms V^{(p)}_k} (Y^{n,p})
\,<\, a\, \theta^{(p)}_n \,\big] \;\ge\;
\bb Q^{(p)}_{j} \big[ \, H_{k} \,<\, a\,  \,\big] 
\;>\; 0\;.
\end{equation}

Denote by $Y^{n,j,k}_t$ the trace of $X^n_t$ on
$\ms V^{(p)}_j \cup \ms V^{(p)}_k$.  By \cite[Theorem 2.6]{bl2} (for
the process $Y^{n,j,k}_t$ and with $\ms B = \ms W = \ms V^{(p)}_j$,
$\ms B^c = \ms V^{(p)}_k$) and \cite[Theorem 7.1]{bl4} (Condition T4
ensures that the hypothesis (2.14) of \cite[Theorem 2.6]{bl2} is in
force), under $\mb P^n_{\! x}$, the random variable
$H_{\ms V^{(p)}_k} (Y^{n,j,k})/ \theta^{j,k}_n$ converges in
distribution to a mean-one exponencial random variable. In this
formula,
\begin{equation*}
\theta^{j,k}_n \;=\;
\frac{\pi^{j,k}_n(\ms V^{(p)}_j)}
{\Cap^{j,k}_n(\ms V^{(p)}_j\,,\, \ms V^{(p)}_k)} \;,
\quad \pi^{j,k}_n(\ms V^{(p)}_j) \;=\;
\frac{\pi_n(\ms V^{(p)}_j)}{\pi_n(\ms V^{(p)}_j \cup \ms
V^{(p)}_j)}\;, 
\end{equation*}
and $\Cap^{j,k}_n$ stands for the capacity with respect to the trace
process $Y^{n,j,k}_t$. By \cite[Lemma 6.9]{bl2},
$\Cap_n(\ms V^{(p)}_j\,,\, \ms V^{(p)}_k) = \pi_n(\ms V^{(p)}_j \cup
\ms V^{(p)}_j)$ $\Cap^{j,k}_n(\ms V^{(p)}_j\,,\, \ms V^{(p)}_k)$, so
that
\begin{equation*}
\theta^{j,k}_n \;=\;
\frac{\pi_n(\ms V^{(p)}_j)}
{\Cap_n(\ms V^{(p)}_j\,,\, \ms V^{(p)}_k)} \;\cdot
\end{equation*}

Suppose by contradiction that the limit appearing in the statement of
the lemma vanishes, so that $\theta^{(p)}_n / \theta^{j,k}_n \to 0$
and
$\mb P^n_{\! x} [H_{\ms V^{(p)}_k} (Y^{n,j,k}) < a \theta^{(p)}_n] \to 0$
for all $a>0$. Hence, as
$H_{\ms V^{(p)}_k} (Y^{n,j,k}) \,\le\, H_{\ms V^{(p)}_k} (Y^{n,p}) $,
\begin{equation*}
\lim_{n\to\infty} \mb P^{n}_{\! x} \big[ \, H_{\ms V^{(p)}_k} (Y^{n,p})
\,<\, a\, \theta^{(p)}_n \,\big] \;=\; 0\;.
\end{equation*}
This contradicts \eqref{96}, and therefore one must have that
$r^{(p)}(j,k)=0$, completing the proof of the lemma by contradiction.
\end{proof}

Fix $1\le p\le \mf q$, $j\in \mf T_p$. Let $\color{blue} A_j$ be the
recurrent points of the chain $\bb X^{(p)}_t$ which can be hit before
any other recurrent point when the chain starts from $j$. More
precisely, $\ell \in A_j$ if, and only if, $\ell\in \mf R^{(p)}$ and
there exists a path $j_0 = j, j_1, \dots, j_m=\ell$ such that
$r^{(p)}(j_a,j_{a+1}) >0$, $j_a\in \mf T_p$, $0\le a<m$. Let
$\color{blue} \ms A_j = \cup_{\ell \in A_j} \ms V^{(p)}_\ell$.

\begin{lemma}
\label{l31}
Fix $1\le p\le \mf q$, $x\in \ms V^{(p)}_j$, $j\in \mf T_p$. Then,
\begin{equation*}
\liminf_{n\to\infty} \frac{\theta^{(p)}_n}{\pi_n(x)}\,
\Cap_n(\{x\} \,,\, \ms A_j) \;>\; 0 \;.
\end{equation*}
\end{lemma}

\begin{proof}
As $x\in \ms V^{(p)}_j$, $j\in \mf T_p$, there exists a path
$j_0 = j, j_1, \dots, j_m$ such that $r^{(p)}(j_a,j_{a+1}) >0$,
$j_a\in \mf T_p$, $0\le a<m$, $j_m \in A_j$. Moreover, for
$0\le a<m$, by \eqref{97}, with $p$ instead of $p-1$,
$\pi_n(\ms V^{(p)}_{j_a}) \preceq \pi_n(\ms V^{(p)}_{j_{a+1}})$, and,
by Lemma \ref{l30},
\begin{equation*}
\liminf_{n\to\infty} \frac{\theta^{(p)}_n}{\pi_n(\ms V^{(p)}_{j_a})}\, \Cap_n(\ms
V^{(p)}_{j_a}\,,\, \ms V^{(p)}_{j_{a+1}}) \;>\; 0 \;.
\end{equation*}
This limit is finite because this capacity is bounded by the one
obtained by replacing $\ms V^{(p)}_{j_{a+1}}$ by
$\breve{\ms V}^{(p)}_{j_{a}}$, and the limit for this later one is
finite in view of \eqref{34}.

By the previous displayed equation and Lemma \ref{l23}, the exist a
positive constant $c_0$ and self-avoiding paths $\gamma_a$ from
$\ms V^{(p)}_{j_a}$ to $\ms V^{(p)}_{j_{a+1}}$ such that
$c_n(\gamma_a) \ge c_0\, \pi_n(\ms V^{(p)}_{j_a}) / \theta^{(p)}_n \ge
c_0\, \pi_n(\ms V^{(p)}_{j}) / \theta^{(p)}_n $, $0\le a<m$.

Denote by $y_{a}$, $0\le a<m$, the starting points of the paths
$\gamma_a$, and by $x_{a+1}$ its ending point. Let $x_0=x$. Hence
$x_a$, $y_a$ belongs to the same well $\ms V^{(p)}_{j_{a}}$. By
Property (T4) in \cite[Theorem 7.1]{bl4} and Lemma \ref{l23}, there
exist self-avoiding paths $\gamma'_a$ from $x_a$ to $y_a$ such that
$c_n(\gamma'_a) \ge c_0 \pi_n(\ms V^{(p)}_{j_a}) / \theta^{(p-1)}_n
\ge c_0 \pi_n(\ms V^{(p)}_{j}) / \theta^{(p)}_n$, where the value of
the constant $c_0$ may change from line to line.

By concatenating the paths $\gamma_a$, $\gamma'_a$, we obtain a path
$\gamma$ from $x$ to $\ms A_j$ such that
$c_n(\gamma) \ge c_0 \pi_n(\ms V^{(p)}_{j}) / \theta^{(p)}_n$.
If it is not self-avoiding, we may shorten it improving the lower on
$c_n(\gamma)$ and keeping it as a path from $x$ to $\ms A_j$. At
this point, the assertion of the lemma follows from Lemma \ref{l23}
and \eqref{58}.
\end{proof}

Fix $x\in \Delta$. Let $\color{blue} \ms A_x$ be the
recurrent points of the chain $\bb X_t$ which can be hit before
any other recurrent point when the chain starts from $x$. More
precisely, $y \in \ms A_x$ if, and only if, $y\in\ms V$ and
there exists a path $x_0 = x, x_1, \dots, x_m=y$ such that
$\bb R_0(x_a,x_{a+1}) >0$, $x_a\in \Delta$, $0\le a<m$. 

\begin{lemma}
\label{l38}
Fix $x\in \Delta$. Then,
\begin{equation*}
\liminf_{n\to\infty} \frac{1}{\pi_n(x)}\,
\Cap_n(\{x\} \,,\, \ms A_x) \;>\; 0 \;.
\end{equation*}
\end{lemma}

\begin{proof}
By definition of the path from $x$ to $\ms A_x$,
$\bb R_0(x_a,x_{a+1}) >0$ for all $0\le a<m$. Hence
$\pi_n(x_a) \preceq \pi_n(x_{a+1})$ and
$c_n(x_a,x_{a+1}) = \pi_n(x_a) \, R_n (x_a,x_{a+1}) \succeq \pi_n(x_a)
\succeq \pi_n(x)$. This proves that $c_n(\gamma) \succeq  \pi_n(x)$
and completes the proof of the lemma in view of Lemma \ref{l23}.
\end{proof}

\begin{lemma}
\label{l20}
Fix $1\le p\le \mf q$. Then, for all for $x\not\in \ms V^{(p)}$, $j\in
S_p$, 
\begin{equation*}
\lim_{n\to\infty}
\frac{\pi_n(x)}{\pi_n(\ms V^{(p)}_j)}  \; \mb P^n_{\! x}
\big[\, H_{\ms V^{(p)}_j} =  H_{\ms V^{(p)}} \,\big] ^2  \, 
\;=\; 0 \;.
\end{equation*}
\end{lemma}

\begin{proof}
If $\pi_n(x) / \pi_n(\ms V^{(p)}_j) \to 0$, the conclusion is
straightforward. If $\pi_n(\ms V^{(p)}_j) \, \sim\, \pi_n(x)$, by
Corollary \ref{l26}, $\mf a^{(p-1)}(x,j)=0$, so that, by Lemma
\ref{l22}, 
\begin{equation*}
\lim_{n\to\infty}
\mb P^n_{\! x} \big[\, H_{\ms V^{(p)}_j} =  H_{\ms V^{(p)}} \,\big]
\;=\; 0\;,
\end{equation*}
and the assertion of the lemma follows.

Assume that $\pi_n(\ms V^{(p)}_j) \prec \pi_n(x)$, and suppose that
$x\in\ms V$.  Let $1\le r<p$ such that $x\in \ms V^{(r)}_k$ for some
$k\in \mf T_r$. Such $r$ exists and is smaller than $p$ because
$x\not\in \ms V^{(p)}$. 

Recall the definition of the sets $A_k$, $\ms A_k$ introduced just
before Lemma \ref{l31}. Add the index $r$ to recall that $k\in \mf
T_r$ and write $A_{r,k}$, $\ms A_{r,k}$ instead of $A_k$, $\ms A_k$,
respectively. By definition, $\ms A_{r,k} \subset \ms V^{(r+1)}$.

By the tree construction, since $r<p$, there exists
$B\subset S_{r+1}$, such that
$\ms V^{(p)}_j = \cup_{i\in B} \ms V^{(r+1)}_i$.  By Lemma \ref{l25},
$\pi_n(x) \prec \pi_n(\ms V^{(r)}_\ell)$, $\ell\in A_{r,k}$. Thus, as
$\pi_n(\ms V^{(p)}_j) \prec \pi_n(x)$,
$\pi_n(\ms V^{(r+1)}_i) \prec \pi_n(\ms V^{(r)}_\ell)$ for all
$i\in B$, $\ell\in A_{r,k}$. Hence, since by \eqref{58}, all elements
of the same valley have measures of the same order,
$\ms A_{r,k} \cap \ms V^{(p)}_j = \varnothing$.

The proof is by induction on $r$. We first prove it for $r=p-1$. In
the sequel, we show that if it holds for all
$r\in \{r_0+1, \dots, p-1\}$, then it holds for $r_0$ also. First,
assume that $r=p-1$ (and keep the index $r$ of $\ms A_{r,k}$ as $r$,
though $r=p-1$). In this case, since $\ms A_{r,k} \subset \ms V^{(p)}$
and $\ms A_{r,k} \cap \ms V^{(p)}_j = \varnothing$, we have that
$\ms A_{r,k} \subset \breve{\ms V}^{(p)}_j$. Therefore,
\begin{equation*}
\mb P^n_{\! x}
\big[\, H_{\ms V^{(p)}_j} <  H_{\breve{\ms V}^{(p)}_j} \,\big] \;\le\;
\mb P^n_{\! x}
\big[\, H_{\ms V^{(p)}_j} <  H_{\ms A_{r,k}} \,\big]  \;. 
\end{equation*}

By Lemma \ref{l27},
\begin{equation}
\label{98}
\frac{\pi_n(x)}{\pi_n(\ms V^{(p)}_j)}  \; \mb P^n_{\! x}
\big[\, H_{\ms V^{(p)}_j} <  H_{\ms A_{r,k}} \,\big] ^2 \;\le\;
C_0\, \frac{\pi_n(x)}{\pi_n(\ms V^{(p)}_j)} \,
\frac{\Cap_n (\ms V^{(p)}_j \,,\, \ms A_{r,k})}
{\Cap_n (\{x\} \,,\, \ms A_{r,k})}
\end{equation}
for some finite constant $C_0$.  By equation (B.2) in
\cite{lrev},
\begin{equation*}
\Cap_n (\ms V^{(p)}_j \,,\, \ms A_{r,k}) \;\le\;
\Cap_n (\ms V^{(p)}_j\,,\, \breve{\ms V}^{(p)}_j)\;.
\end{equation*}
By \eqref{26b}, this expression is bounded by
$C_0\, \pi_n(\ms V^{(p)}_j)/\theta^{(p)}_n$ for some finite constant
$C_0$ whose value may change from line to line.

On the other hand, by Lemma \ref{l31}, as $r=p-1$,
\begin{equation}
\label{f1}
\Cap_n (\{x\} \,,\, \ms A_{r,k}) \;\ge\;
c_0 \; \pi_n(x)/\theta^{(p-1)}_n 
\end{equation}
for some positive constant $c_0$. Putting together the two previous
estimates, we obtain that the expression in \eqref{98} vanishes as
$n\to\infty$. This completes the proof of the lemma in the case
$r=p-1$. 

We turn to the induction argument. Fix $r<p$ and assume that the
result holds for $r+1, \dots, p-1$.  Recall the notation introduced at
the beginning of the proof and write
\begin{equation}
\label{99}
\mb P^n_{\! x}
\big[\, H_{\ms V^{(p)}_j} =  H_{\ms V^{(p)}} \,\big] \;\le\;
\mb P^n_{\! x}
\big[\, H_{\ms V^{(p)}_j} <  H_{\ms A_{r,k}} \,\big] \;+\;
\mb P^n_{\! x}
\big[\, H_{\ms A_{r,k}} < H_{\ms V^{(p)}_j} <  H_{\breve{\ms V}^{(p)}_j}
\,\big] \;. 
\end{equation}
We estimate separately the square of each term on the right-hand
side.

The argument for the first term is similar to the one presented for
$r=p-1$. By Lemma \ref{l27}, \eqref{98} holds for some finite constant
$C_0$.  By equations (B.1) and (B.2) in \cite{lrev},
\begin{equation*}
\Cap_n (\ms V^{(p)}_j \,,\, \ms A_{r,k}) \;\le\; 
\sum_{i\in B} \Cap_n (\ms V^{(r+1)}_i \,,\, \ms A_{r,k}) \;\le\;
\sum_{i\in B} \Cap_n (\ms V^{(r+1)}_i\,,\, \breve{\ms V}^{(r+1)}_i)\;.
\end{equation*}
By \eqref{26b}, this expression is bounded by
$\sum_{i\in B} \pi_n(\ms V^{(r+1)}_i)/\theta^{(r+1)}_n = \pi_n(\ms
V^{(p)}_j)/\theta^{(r+1)}_n$.  On the other hand, by Lemma \ref{l31},
\eqref{f1} is in force with $\theta^{(r)}_n $ in place of
$\theta^{(p-1)}_n$ and some positive constant $c_0$. Putting together
the two previous estimates, we obtain that the expression in
\eqref{98} vanishes as $n\to\infty$.

We turn to the second term in \eqref{99}. By the strong Markov
property, it is bounded by
\begin{equation*}
\max_{z\in \ms A_{r,k}} \mb P^n_{\! z}
\big[\, H_{\ms V^{(p)}_j} <  H_{\breve{\ms V}^{(p)}_j}
\,\big] \;.
\end{equation*}
To complete the proof, it remains to show that for all
$z\in \ms A_{r,k}$,
\begin{equation*}
\lim_{n\to\infty}
\frac{\pi_n(x)}{\pi_n(\ms V^{(p)}_j)}  \; \mb P^n_{\! z}
\big[\, H_{\ms V^{(p)}_j} <  H_{\breve{\ms V}^{(p)}_j}\,\big] ^2  \, 
\;=\; 0 \;.
\end{equation*}
Since $\pi_n(x) \prec \pi_n(z)$, it is enough
to show that
\begin{equation}
\label{f2}
\lim_{n\to\infty}
\frac{\pi_n(z)}{\pi_n(\ms V^{(p)}_j)}  \; \mb P^n_{\! z}
\big[\, H_{\ms V^{(p)}_j} <  H_{\breve{\ms V}^{(p)}_j} \,\big] ^2  \, 
\;=\; 0 \;.
\end{equation}
This follows from the induction hypothesis. Indeed, as
$z\in \ms V^{(r+1)}$, either $z$ belongs to $\ms V^{(p)}$ or $z$
belongs to some $\ms V^{(s)}_\ell$, $r<s<p$, for some
$\ell \in \mf T_s$. In the first case, the probability vanishes
because $z\in \breve{\ms V}^{(p)}_j$ (as $z\in \ms V^{(p)} \cap \ms A_{r,k}$
and $\ms A_{r,k} \cap \ms V^{(p)}_j = \varnothing$,
$z\in \breve{\ms V}^{(p)}_j$). In the second case, \eqref{f2} holds by
the induction hypothesis.

It remains to consider the case where
$\pi_n(\ms V^{(p)}_j) \prec \pi_n(x)$ and $x\in \Delta$.  We repeat
the induction argument. Write \eqref{99} with $\ms A_x$ instead of
$\ms A_{r,k}$. We estimate the first term on the right-hand side of
\eqref{99} as before, applying Lemma \ref{l38} instead of Lemma
\ref{l31}. The second term is also bounded as before. At the end of
the argument one needs to estimate \eqref{f2} for $z\in \ms V$, 
$\pi_n(\ms V^{(p)}_j) \prec \pi_n(z)$. This has been done in the first
part of the proof.
\end{proof}

\section{Proof of Theorem \ref{mt3}}
\label{sec7}

We assume in this section that the dynamics is reversible:
$\pi_n(x) \, R_n(x,y) = \pi_n(y)\, R_n(y,x)$ for all $(x,y)\in E$.

\subsection*{Elementary properties of $\pi_n$}

The proof of Theorem \ref{mt3} requires some preparation. We first
introduce the transient equivalent classes of the chain $X^{(n)}_t$.
We say that $y$ is equivalent to $x$, $\color{blue} y\sim x$ if $y=x$
or if there exists a sequence $x=x_0, \dots, x_\ell = y$,
$y_0=y, \dots, y_m=x$ such that $\bb R_0(x_i,x_{i+1})>0$,
$\bb R_0(y_j,y_{j+1})>0$ for all $0\le i<\ell$, $0\le j<m$.

This relation divides the set $V$ into equivalent classes.  Clearly
the sets $\ms V_j$ are equivalent classes, but there might be
others. Denote by $\color{blue} \ms C_1, \dots, \ms C_{\mf m}$ the
equivalent classes which \emph{have more than one element} and are not
one of the sets $\ms V_j$, $1\le j\le \mf n$. Note that the sets
$\ms C_1, \dots, \ms C_{\mf m}$, $\ms V_1, \dots, \ms V_{\mf n}$ may
not exhaust $V$: the set $V$ may contain elements which do not belong
to one of the $\ms C_k$'s nor to one of the $\ms V_j$'s.

The first assertion extends \eqref{58} to the sets $\ms C_k$. We claim
that if $x$, $y$ belong to the same class $\ms C_k$, then
\begin{equation}
\label{f8}
\lim_{n\to\infty} \frac{\pi_n(x)}{\pi_n(y)} \; =\; 
a \in (0,\infty) \;.
\end{equation}
Indeed, by definition, there exists a sequence $x=x_0, \dots, x_\ell = y$,
such that $\bb R_0(x_i,x_{i+1})>0$, for all $0\le i<\ell$. By
reversibility,
\begin{equation*}
\frac{\pi_n(x)}{\pi_n(y)} \;=\;
\frac{R_n(x_{\ell},x_{\ell-1}) \dots R_n(x_1 , x_{0})}
{R_n(x_0 , x_{1}) \dots R_n(x_{\ell-1},x_{\ell})} \;\cdot
\end{equation*}
By hypothesis, the denominator converges to a positive real number. On
the other hand, by \eqref{01}, the numerator converges to a
non-negative real number. This proves that $\pi_n(x)/\pi_n(y)  \to 
a \in [0,\infty)$. Inverting the roles of $x$ and $y$ we conclude that 
$a \in (0,\infty)$, as claimed in \eqref{f8}.

Fix an oriented edge $(x,y)\in E$ whose endpoints belong to the same
equivalent class $\ms V^{(1)}_j$ or $\ms C_k$, $j\in S_1$,
$1\le k\le \mf m$. We claim that
\begin{equation}
\label{77b}
\text{$\bb R_0(x,y) >0$ if and only if $\bb R_0(y,x) >0$}\;. 
\end{equation}
Indeed, assume that $\bb R_0(x,y) >0$. Since
$\pi_n(x) \, R_n(x,y) = \pi_n(y)\, R_n(y,x)$, by \eqref{58} and \eqref{f8},
$\lim_{n\to \infty} R_n(y,x) = \bb R_0(x,y) \lim_{n\to \infty}
\pi_n(x) / \pi_n(y) >0$.

Denote by $\color{blue} \bb L^{(0)}_j$,
$\color{blue} \bb L^{(0)}_{T,k}$, $j\in S_1$, $1\le k\le \mf m$, the
generators associated to the rates $\bb R_0$ restricted to the
equivalent classes $\ms V^{(1)}_j$, $\ms C_k$, respectively. This
means that we set to $0$ all jumps from $\ms C_k$ to its complement.
Denote by $\color{blue} \nu_k$ the stationary state of the Markov
chain associated to the generator $\bb L^{(0)}_{T,k}$.

We claim that for all $1\le k\le \mf m$,
\begin{equation}
\label{f9}
\lim_{n\to\infty} \frac{\pi_n(x)}{\pi_n(\ms C_k)} \;=\; \nu_k(x)
\quad\text{for all $x\in \ms C_k$ and that $\nu_k$ is reversible}\;.
\end{equation}
This result extends Lemma \ref{l32} to the transient sets $\ms C_k$.
To establish \eqref{f9}, let $m\in\ms P(\ms C_k)$ be the limit of the
sequence of measures $\pi_n(\,\cdot\,)/\pi_n(\ms C_k)$. This limit
exists by \eqref{f8}.  By reversibility, for all $x$, $y\in\ms C_k$,
\begin{equation*}
\frac{\pi_n(x)}{\pi_n(\ms C_k)} \, R_n(x,y) \;=\;
\frac{\pi_n(y)}{\pi_n(\ms C_k)} \, R_n(y,x) \;.
\end{equation*}
Passing to the limit yields that $m$ satisfies the detailed balance
conditions with respect to $\bb R_0$. Hence $m$ is stationary
(actually, reversible), and, by uniqueness, $m=\nu_k$. This proves
that the sequence of measures $\pi_n(\,\cdot\,)/\pi_n(\ms C_k)$
converges to $\nu_k$ and that $\nu_k$ is reversible.

The same statement yields that  $\pi^{(1)}_j$ is a reversible measure
for the chain $\bb X_t$ restricted to $\ms V_j$, $j\in S_1$.

\subsection*{The functionals $\ms I^{(p)}$}

The first result of this section provides an alternative formula for
the functional $\ms I^{(0)}$ introduced in \eqref{f11}.  Its proof
relies on the construction of a directed graph without directed
loops. The equivalence classes of the chain $\bb X_t$ form the set of
vertices of this directed graph. Denote them by
$\ms Q_1, \dots, \ms Q_\ell$. The sets $\ms V_j$ and $\ms C_k$ belongs
to this set and are vertices of the graph.  In other words, for each
$j\in S_1$, there exists $1\le a\le \ell$ such that
$\ms V_j = \ms Q_a$. A similar statement holds for the sets $\ms C_k$.

Draw a directed arrow from $\ms Q_a$ to $\ms Q_b$ if there exists
$x\in \ms Q_a$ and $y\in \ms Q_b$ such that $\bb R_0(x,y)>0$. Denote
the set of directed edges by $\bb A$ and the graph by
$\bb G = (\mb Q, \bb A)$, where $\mb Q$ is the set
$\{\ms Q_1, \dots, \ms Q_\ell\}$ of vertices.

A path in the graph $\bb G = (\mb Q, \bb A)$ is a sequence vertices
$(\ms Q_{a_j} : 0\le j\le m)$, such that there is a directed arrow
from $\ms Q_{a_j}$ to $\ms Q_{a_{j+1}}$ for $0\le j<m$.

This directed graph has no directed loops because the existence of a
directed loop would contradict the definition of the sets $\ms Q_a$ as
equivalent classes. (Mind that undirected loops might exist). On the
other hand, since the sets $\ms V_j$ are closed irreducible classes,
these sets are not the tail of a directed edge in the graph. Finally,
fix an equivalent class $\ms Q_a$ which is not a set $\ms V_j$. Hence,
the elements of $\ms Q_a$ are transient for the chain $\bb X_t$. In
particular, there is a path
$(\ms Q_a= \ms Q_{a_0}, \dots, \ms Q_{a_m})$ such that $\ms Q_{a_j}$
is not a closed irreducible class for $0\le j<m$, and $\ms Q_{a_m}$ is
one.

Fix an equivalent class $\ms Q_a$ which is not a set $\ms V_j$.
Denote by $\mb D(\ms Q_a)$ the length of the longest path from
$\ms Q_a$ to a closed irreducible class. The function $\mb D$ is well
defined because (a) the set of vertices is finite, (b) there is at
least a path, (c) there are no directed loops in the graph.

Fix $a$, $b$ such that there is a directed arrow from $\ms Q_a$ to
$\ms Q_b$. Then, 
\begin{equation}
\label{f10}
\mb D(\ms Q_a) \;\ge\; \mb D(\ms Q_b) \;+\; 1\;.
\end{equation}
Indeed, it is enough to consider the longest path from $\ms Q_b$ to
the irreducible classes. $\ms Q_a$ does not belong to the path because
there are no directed loops. By adding $\ms Q_a$ at the beginning of
the path from $\ms Q_b$ to the irreducible classes, we obtain a path
from $\ms Q_a$ to the irreducible classes of length
$\mb D(\ms Q_b) + 1$, proving \eqref{f10}.

We may lift the function $\mb D$ to $V$ by setting
$\mb D(x) = \mb D(\ms Q_a)$ for all $x\in \ms Q_a$.

Let $\ms J^{(0)} \colon \ms P(V) \to [0,+\infty]$ be the functional
defined by
\begin{equation}
\label{09b}
\begin{aligned}
\ms J^{(0)} (\mu) \, & =\,
\sum_{j\in S_1} 
\big\< \sqrt{f_j} \,,\, (-\, \bb L^{(0)}_j)
\sqrt{f_j} \,\big\>_{\pi^{(1)}_j} \;+\;
\sum_{k=1}^{\mf m}
\big\< \sqrt{g_k} \,,\, (-\, \bb L^{(0)}_{T,k})
\sqrt{g_k} \,\big\>_{\nu_k} \\ 
\;& +\; \sum_{k=1}^{\mf m} \sum_{x \in \ms C_k} \sum_{y\not\in \ms
C_k} \mu(x)\, \bb R_0(x,y)
+\; \sum_{x \not \in \ms C \cup \ms V} \sum_{y\in V} \mu(x)\, \bb R_0(x,y) \;.
\end{aligned}
\end{equation}
In this formula, $\color{blue} \ms C = \cup_k \ms C_k$,
$f_j(x) = \mu(x) / \pi^{(1)}_j (x)$, $g_k(z) = \mu(z) / \nu_k (z)$,
$x\in \ms V^{(1)}_j$, $z\in \ms C_k$.

\begin{lemma}
\label{l35}
For every $\mu \in \ms P(V)$, $\ms I^{(0)} (\mu) = \ms J^{(0)} (\mu) $.
\end{lemma}

\begin{proof}
Fix $\mu \in \ms P(V)$.  We first prove that
$\ms J^{(0)}(\mu) \le \ms I^{(0)}(\mu)$. By definition of the
generator $\bb L^{(0)}$, 
\begin{equation}
\label{f12}
\ms I^{(0)}(\mu) \;=\; \sup_{u>0}
\, -\, \sum_{(x,y) \in \bb E_0} \frac{\mu(x)}{u (x)}\,
\bb R_0(x,y) \, [\, u (y) - u (x)\,] \;,
\end{equation}
where the sum is performed over all directed edges of $\bb E_0$.

Fix $\ell\ge 1$, and define $u_\ell : V \to (0,\infty)$ by
\begin{equation*}
u_\ell(x) \;=\; \ell^{\mb D(x)}\; \sqrt{ \frac{\mu(x) + \epsilon}{\pi^{(1)}_j(x)}}\;,
\quad
u_\ell(y) \;=\; \ell^{\mb D(y)}\; \sqrt{ \frac{\mu(y) + \epsilon}{\nu_k(y)}}\;,
\quad
u_\ell(z) \,=\, \ell^{\mb D(z)} \;,
\end{equation*}
for $x\in \ms V_j$, $j\in S_1$, $y\in \ms C_k$, $1\le k\le \mf m$, and
$z\not\in \ms V \cup \ms C$. Here, $\epsilon = 1/\ell$ and guarantees
that $u$ is positive. By definition of $\ms I^{(0)}$,
\begin{equation}
\label{84}
\ms I^{(0)} (\mu) \; \ge \;
\limsup_{\ell\to\infty}  \, -\, \sum_{(x,y) \in \bb E_0} \frac{\mu(x)}{u_\ell(x)}\,
\bb R_0(x,y) \, [\, u_\ell(y) - u_\ell(x)\,] \;. 
\end{equation}

We examine the asymptotic behavior of the right-hand of
\eqref{84}. Fix $(x,y)\in \bb E_0$, and suppose, first, that $x$,
$y\in \ms V_j$ for some $j\in S_1$. In this case, the factors
$\ell^{\mb D}$ cancel, and, as $\ell\to\infty$, the corresponding term
in \eqref{84} converges to
\begin{equation*}
\pi^{(1)}_j(x) \, \sqrt{f_j(x)} \,
\bb R_0(x,y) \, \big[\, \sqrt{f_j(y)} \,-\, \sqrt{f_j(x)} \,\big] 
\end{equation*}
where $f_j(x) = \mu(x)/\pi^{(1)}_j(x)$. Therefore, the contributions to
the right-hand side of \eqref{84}, of the sum over the edges
$(x,y)\in \bb E_0$ such that $x$, $y\in\ms V_j$ is
\begin{equation*}
\big\< \sqrt{f_j} \,,\, (-\, \bb L^{(0)}_j)
\sqrt{f_j} \,\big\>_{\pi^{(1)}_j} \;.
\end{equation*}

The same argument yields that the contributions to
the right-hand side of \eqref{84}, of the sum over the edges
$(x,y)\in \bb E_0$ such that $x$, $y\in\ms C_k$ for some $1\le k\le
\mf m$, is
\begin{equation*}
\big\< \sqrt{g_k} \,,\, (-\, \bb L^{(0)}_{T,k})
\sqrt{g_k} \,\big\>_{\nu_k} \;,
\end{equation*}
where $g_k(x) = \mu(x)/\nu_k(x)$.

Up to this point we considered all edges $(x,y)\in \bb E_0$ whose head
and tail belong to the same equivalent class $\ms V_j$ or $\ms
C_k$. Assume now that this is not the case, and consider the term
\begin{equation*}
-\, \frac{\mu(x)}{u_\ell(x)}\, \bb R_0(x,y) \, [\, u_\ell(y) - u_\ell(x)\,]
\;=\; \mu(x) \, \bb R_0(x,y)\;-\;
\frac{\mu(x)}{u_\ell(x)}\, \bb R_0(x,y) \, u_\ell(y)\;. 
\end{equation*}
By definition, and since the measures $\pi^{(1)}_j$, $\nu_k$ are
strictly positive, $u_\ell(y) \le C_0 \, \ell^{\mb D(y)}$,
$\mu(x)/u_\ell(x) \le C_0 \, \ell^{-\mb D(x)}$ for some finite
constant $C_0$ independent of $x$, $y$ and $\ell$. The absolute value
of the second term is thus bounded above by
$C_0 \,\ell^{\mb D(y)-\mb D(x)}$.  Since there is an edge from $x$ to
$y$, by \eqref{f10}, $\mb D(x) \ge \mb D(y) +1$, which proves that the
second term of the previous displayed equation vanishes as
$\ell\to\infty$. 

Fix an edge $(x,y)\in \bb E_0$ whose head and tail do not belong to
the same equivalent class $\ms V_j$ or $\ms C_k$. Since $\ms V_j$ is a
closed irreducible class, $x\not\in\ms V$. Suppose that
$x\in \ms C_k$. Hence, $y\not\in \ms C_k$ because they do not belong
to the same class. These are the terms which respond for the third sum
in \eqref{09b}. the terms in which $x\not \in \ms C$ respond for the
fourth sum in \eqref{09b}, completing the proof that $\ms I^{(0)}(\mu)
\ge \ms J^{(0)}(\mu)$.

We turn to the reverse inequality,
$\ms J^{(0)}(\mu) \ge \ms I^{(0)}(\mu)$.  By \eqref{f12},
\begin{equation*}
\ms I^{(0)}(\mu) \;\le\; \sum_{j\in S_1} \ms I^{(0)}_{\ms V_j}(\mu)
\;+\; \sum_{k=1}^{\mf m} \ms I^{(0)}_{\ms C_k}(\mu)
\;+\; \ms I^{(0)}_{\ms R}(\mu) \;,
\end{equation*}
where $\ms I^{(0)}_{\ms V_j}(\mu)$ is given by formula \eqref{f12}
when the sum is performed over the directed edges $(x,y)$ whose head
and tail belong to $\ms V_j$. $ \ms I^{(0)}_{\ms C_k}(\mu)$ is defined
similarly, while $\ms I^{(0)}_{\ms R}(\mu)$ contains the remaining
edges.

By \cite[Theorem 5]{var}, 
\begin{equation*}
\ms I^{(0)}_{\ms V_j}(\mu) \;=\;
\big\< \sqrt{f_j} \,,\, (-\, \bb L^{(0)}_{j}) \sqrt{f}
\,\big\>_{\pi^{(1)}_j}\;, 
\end{equation*}
where $f_j (x) = \mu(x)/\pi^{(1)}_j(x)$, $x\in \ms V_j$.
An analogous result holds for $\ms I^{(0)}_{\ms C_k}(\mu)$.
These two terms correspond to the first two terms in \eqref{09b}.

We turn to $\ms I^{(0)}_{\ms R}(\mu)$, which can be written as
\begin{equation*}
\ms I_{\ms R}^{(0)}(\mu) 
\;=\; \sum_{(x,y)} \mu(x) \, \bb R_0(x,y) 
\;+\; \sup_{u>0} \, -\, \sum_{(x,y)} \frac{\mu(x)}{u (x)}\,
\bb R_0(x,y) \,  u (y) \;.
\end{equation*}
where the sums are performed over directed edges whose head and tail
belong to different equivalent classes. Since the second term is
negative,
\begin{equation*}
\ms I_{\ms R}^{(0)}(\mu) 
\;\le \; \sum_{(x,y)} \mu(x) \, \bb R_0(x,y) \;.
\end{equation*}
We have seen in the first part of the proof that this sum can be
written as the third and fourth terms in $\ms J^{(0)}(\mu)$,
completing the proof of the lemma.
\end{proof}

Note that for each $1\le k\le \mf m$, there exists at least on
$x\in\mc C_k$ such that $\bb R_0(x,y)>0$ for some $y\not\in \ms
C_k$. On the other hand, as the Markov chain associated to
$\bb L^{(0)}_{T,k}$ is ergodic,
$\big\< \sqrt{g} \,,\, (-\, \bb L^{(0)}_{T,k}) \sqrt{g}
\,\big\>_{\nu_k} = 0$ entails that $g$ is constant.  Therefore,
$\ms I^{(0)} (\mu) \,=\, 0$ if and only if there exists a probability
measure $\omega$ on $S_1$ such that
\begin{equation}
\label{79}
\mu \;=\; \sum_{j\in S_1} \omega_j\, \pi^{(1)}_j\;.
\end{equation}

Fix $1\le p\le \mf q$, and let
$\ms J^{(p)} \colon \ms P(V) \to [0,+\infty]$ be the functional
defined as follows. If
$\mu = \sum_{j\in S_p} \omega_j \, \pi^{(p)}_j$ for some probability
measure $\omega$ in $S_p$, $\omega\in \ms P(S_p)$,
\begin{equation}
\label{09c}
{\color{blue} \ms J^{(p)} (\mu)} \, :=\,
\sum_{m\in S_{p+1}} 
\big\< \sqrt{f_m} \,,\, (-\, \bb L^{(p)}_m)
\sqrt{f_m} \,\big\>_{M^{(p)}_m}
\;+\; \sum_{j \in \mf T_p} \sum_{k\in S_p}
\omega_j \, r^{(p)} (j,k)\;.
\end{equation}
In this formula, $\color{blue} \bb L^{(p)}_m$ stands for the generator
associated to Markov chain $\bb X^{(p)}_t$ restricted to
the closed irreducible set $\mf R^{(p)}_m$ and
$f_m(j) = \omega_j / M^{(p)}_m (j)$, $j\in \mf R^{(p)}_m$.
To complete the definition of $\ms J^{(p)}$, set
\begin{equation}
\label{83}
\ms J^{(p)} (\mu) \, :=\, +\infty \quad\text{if $\mu$ is not a convex
combination of $\pi^{(p)}_j$, $j\in S_p$}\;.
\end{equation}

Recall from \eqref{83b} the definition of the functional
$\ms I^{(p)} \colon \ms P(V) \to [0,+\infty]$. The proof of Lemma
\ref{l35} yields that

\begin{lemma}
\label{l36}
For all $1\le p\le \mf q$, $\mu\in \ms P(V)$, $\ms I^{(p)} (\mu) = \ms
J^{(p)} (\mu)$. 
\end{lemma}

Note that, by \eqref{09c} and \eqref{80}, $\ms I^{(p)} (\mu) =0$ if
and only if  there exists a probability measure $\widehat \omega$ in
$S_{p+1}$ such that
\begin{equation}
\label{81}
\mu \;=\; \sum_{m\in S_{p+1}} \widehat \omega_m
\sum_{j\in \mf R^{(p)}_m} M^{(p)}_m (j) \, \pi^{(p)}_j
\;=\;  \sum_{m\in S_{p+1}} \widehat \omega_m \, \pi^{(p+1)}_m
\;.
\end{equation}
On the other hand, if $\mu$ is not of this form, by \eqref{83},
$\ms I^{(p+1)} (\mu)$ is set to be equal to $+\infty$. Hence, the
functional $\ms I^{(p+1)}$ is finite only at the $0$-level set of $\ms
I^{(p)}$. Furthermore, since the right-hand side of \eqref{09c} is
always finite, 
\begin{equation}
\label{82}
\ms I^{(p+1)} (\mu) \;<\; \infty \quad \text{if and only if}\quad
\ms I^{(p)} (\mu) \;=\; 0\;.
\end{equation}
By \eqref{79}, this assertion holds also for $p=0$. 

\subsection*{The $\Gamma$-convergence}

We turn to the proof of Theorem \ref{mt3}. We proceed by induction.
We first show that $\ms I_n$ $\Gamma$-converges to the functional
$\ms I^{(0)}$.  Then, we observe that, according to \eqref{79}, the
$0$-level set of $\ms I^{(0)}$ corresponds to the convex combinations
of the measures $\pi^{(1)}_j$, $j\in S_1$. In the sequel, we prove
that $\theta^{(1)}_n\, \ms I_n$ $\Gamma$-converges to $\ms
I^{(1)}$. Clearly, by definition, $\ms I^{(1)}(\mu) = +\infty$ if
$\mu$ is not a convex combinations of the measures $\pi^{(1)}_j$,
$j\in S_1$, while $\ms I^{(1)}(\mu) < +\infty$ if it is. By
\eqref{81}, the $0$-level set of $\ms I^{(1)}$ consists of the convex
combinations of the measures $\pi^{(2)}_j$, $j\in S_2$.

At this point, we iterate the procedure by examining the behavior of
$\theta^{(2)}_n\, \ms I_n$, and so on until proving that
$\theta^{(\mf q)}_n\, \ms I_n$ $\Gamma$-converges to
$\ms I^{(\mf q)}$. The $0$-level set of this functional is the
singleton formed by the measure $\pi^{(\mf q+1)}$. As the level set is
a singleton, the iterative procedures ends.  Note that this approach
produced the state $\pi^{(\mf q+1)}$ which is is the limit of the
stationary measures $\pi_n$:
$\pi^{(\mf q+1)} (x) = \lim_{n\to\infty} \pi_n(x)$, $x\in V$.

We turn to the proof that $\ms I_n$ $\Gamma$-converges to $\ms I^{(0)}$.  

\begin{proposition}
\label{p01b}
The functional $\ms I_n$ $\Gamma$-converges to $\ms I^{(0)}$.
\end{proposition}

\begin{proof}
We start with the $\Gamma-\limsup$. Fix $\mu\in \ms P(V)$ and consider
the sequence $\mu_n$ constant equal to $\mu$. By \eqref{f6},
\begin{equation*}
\ms I_n (\mu) \;=\; \frac{1}{2} \, \sum_{(x,y)\in E} \pi_n(x) \,
R_n(x,y)\, \Big\{ \sqrt{\frac{\mu(y)}{\pi_n(y)}} \,-\,
\sqrt{\frac{\mu(x)}{\pi_n(x)}} \, \Big\}^2\;,
\end{equation*}

Fix an edge $(x,y)\in E$. We examine the asymptotic behavior of 
\begin{equation}
\label{75}
\pi_n(x) \,
R_n(x,y)\, \Big\{ \sqrt{\frac{\mu(y)}{\pi_n(y)}} \,-\,
\sqrt{\frac{\mu(x)}{\pi_n(x)}} \, \Big\}^2\;.
\end{equation}
By reversibility, this term is symmetric in $x$, $y$. 

There are three types of edges. Assume first that $R_n(x,y) \to 0$ and
$R_n(y,x)\to 0$. By \cite[Lemma 3.1]{lx16}, either $\pi_n(x)/\pi_n(y)$
converges to a nonnegative real number or so does
$\pi_n(y)/\pi_n(x)$. Assume, without loss of generality because
\eqref{75} is symmetric, that
$\pi_n(y)/\pi_n(x) \to a \in [0,\infty)$. In this case, \eqref{75} is
equal to
\begin{equation*}
R_n(y,x)\, \Big\{ \sqrt{\frac{\mu(x)\, \pi_n(y)}{\pi_n(x)}} \,-\,
\sqrt{\mu(y)} \, \Big\}^2\;,
\end{equation*}
which vanishes as $n\to\infty$.

Assume that $R_n(x,y) \not\to 0$ and $R_n(y,x)\to 0$. Hence,
$\bb R_0(x,y) >0$, $\bb R_0(y,x) = 0$. In particular, as the set
$\ms V_j$ are closed irreducible classes, $x\not \in \ms V$ (if
$x\in \ms V_j$ and $\bb R_0(x,y) >0$, then $y\in \ms V_j$ because it
is a closed irreducible class. Hence, by \eqref{77b},
$\bb R_0(y,x) >0$, which is a contradiction). Two possibilities
remain, either $x\in \ms C_k$ for some $k$ or $x\not \in \ms C$.

By reversibility, $\pi_n(x)/\pi_n(y) \to 0$. Hence, \eqref{75} is
equal to
\begin{equation*}
R_n(x,y)\, \Big\{ \sqrt{\frac{\mu(y)\, \pi_n(x)}{\pi_n(y)}} \,-\,
\sqrt{\mu(x)} \, \Big\}^2\;,
\end{equation*}
which converges to $\mu(x) \, \bb R_0(x,y)$.  If $x\in \ms C_k$, by
\eqref{77b}, $y\not\in \ms C_k$. These are the pairs which appear in
the third term on the right-hand side of \eqref{09b}. If
$x\not \in \ms C$ these pairs are responsible for the fourth term on
the right-hand side of \eqref{09b}.

Finally, assume that $R_n(x,y) \not\to 0$ and $R_n(y,x)\not\to 0$.
This means that $x$ and $y$ belong to the same equivalence class, say
$\ms V_j$ or $\ms C_k$. Assume that $x$ and $y\in \ms V_j$. The
argument is identical if we replace $\ms V_j$ by $\ms C_k$.  Replace
$\pi_n(x)$, $\pi_n(y)$ by $\pi_n(x)/\pi_n(\ms V_j)$,
$\pi_n(y)/\pi_n(\ms V_j)$, respectively. By Lemma \ref{l32},
$\pi_n(x)/\pi_n(\ms V_j)$ converges to
$\pi^\sharp_j(x) = \pi^{(1)}_j(x)>0$. Hence, \eqref{75} converges to
\begin{equation*}
\pi^{(1)}_j (x) \,
\bb R_0(x,y)\, \Big\{ \sqrt{\frac{\mu(y)}{\pi^{(1)}_j(y)}} \,-\,
\sqrt{\frac{\mu(x)}{\pi^{(1)}_j(x)}} \, \Big\}^2\;.
\end{equation*}

Putting together the previous estimates yields that $\ms I_n(\mu) \to
\ms J^{(0)}(\mu)$. To complete the proof of the $\Gamma-\limsup$, it
remains to recall the statement of Lemma \ref{l35}.

We turn to the $\Gamma-\liminf$. Fix  $\mu\in \ms P(V)$ and a sequence
of probability measures $\mu_n$ in $\ms P(\ms V)$ converging to $\mu$.
By definition of $\ms I_n$,
\begin{equation*}
\ms I_n(\mu_n) \; \ge \; -\, \int_V \frac{\ms L_n u}{u}\, d\mu_n
\;=\; -\, \sum_{x\in V} \frac{\mu_n(x)}{u(x)}\,
\sum_{y\in V} R_n(x,y) \, [\, u(y) - u(x)\,]
\end{equation*}
for all $u: V \to (0,\infty)$. As $\mu_n\to \mu$ and $R_n \to \bb
R_0$, this expression converges to
\begin{equation*}
-\, \sum_{(x,y) \in \bb E_0} \frac{\mu(x)}{u(x)}\,
\bb R_0(x,y) \, [\, u(y) - u(x)\,] \;.
\end{equation*}
Therefore,
\begin{equation*}
\liminf_{n\to\infty} \ms I_n(\mu_n) \; \ge \;
\sup_{u>0}  \, -\, \sum_{(x,y) \in \bb E_0} \frac{\mu(x)}{u(x)}\,
\bb R_0(x,y) \, [\, u(y) - u(x)\,] \;=\; \ms I^{(0)}(\mu) \;,
\end{equation*}
which completes the proof of the lemma.
\end{proof}

Recall from \eqref{09c} the definition of the functionals $\ms I^{(p)}$,
$1\le p\le \mf q$.

\begin{proposition}
\label{p02b}
Fix $1\le p\le \mf q$.  The functional $\theta^{(p)}_n\, \ms I_n$
$\Gamma$-converges to $\ms I^{(p)}$.
\end{proposition}

\begin{proof}
We start with the $\Gamma-\limsup$. Fix $\mu\in \ms P(V)$. If $\mu$ is
not a convex combinations of the measures $\pi^{(p)}_j$, $j\in S_p$,
there is nothing to prove. Assume, therefore, that
$\mu = \sum_{j\in S_p} \omega_j \, \pi^{(p)}_j$ for some weights
$\omega_j$.

Let $f_n : \ms V^{(p)} \to \bb R_+$ be the function given by
$f_n = \sum_{j\in S_p} \omega_j(n) \, \chi_{_{\ms V^{(p)}_j}}$, where
$\omega_j(n) = \omega_j / \pi_n(\ms V^{(p)}_j)$. To extend
this function to $V$, solve the Poisson equation \eqref{21} with
$\ms L = \ms L_n$, $V_0 = \ms V^{(p)}$, $g = \sqrt{f_n}$. Denote by $h_n$
the solution of the equation. Let $\mu_n = \alpha_n \, h^2_n \,
\pi_n$, where $\alpha_n$ is a normalizing constant which turns $\mu_n$
into a probability measure.

We claim that $\alpha_n\to 1$ and $\mu_n \to \mu$. By definition,
\begin{equation*}
\alpha_n^{-1} \;=\; \sum_{x\not\in \ms V^{(p)}} h_n(x)^2 \, \pi_n(x)
\;+\;
\sum_{j\in S_p} \sum_{x\in \ms V^{(p)}_j} f_n(x) \, \pi_n(x)\;.
\end{equation*}
By definition of $h_n$, for $x\not\in \ms
V^{(p)}$,
\begin{equation*}
\begin{aligned}
h_n(x)^2 \, \pi_n(x) \; & =\; \Big\{\,
\sum_{j\in S_p} \sqrt{\omega_j(n)} \, \mb P^n_{\! x}
\big[\, H_{\ms V^{(p)}_j} =  H_{\ms V^{(p)}} \,\big]  \, \Big\}^2
\, \pi_n(x) \\
& \le \; C_0\, 
\sum_{j\in S_p} \frac{\omega_j}{\pi_n(\ms V^{(p)}_j)}  \, \mb P^n_{\! x}
\big[\, H_{\ms V^{(p)}_j} =  H_{\ms V^{(p)}} \,\big]^2  \, \pi_n(x)\;,
\end{aligned}
\end{equation*}
where the constant $C_0$ bounds the cardinality of $V$.
By Lemma \ref{l20}, this expression vanishes as $n\to\infty$. By definition
of $f_n$, the second term of the penultimate displayed equation is equal to
\begin{equation*}
\sum_{j\in S_p} \sum_{x\in \ms V^{(p)}_j}
\frac{\omega_j}{\pi_n(\ms V^{(p)}_j)}
\, \pi_n(x) \;=\; 1 \;,
\end{equation*}
which proves that $\alpha_n\to 1$.

The previous argument shows that
$\mu_n(x) \,=\, \alpha_n\, h_n(x)^2 \, \pi_n(x) \to 0 = \mu(x)$ if
$x\not\in \ms V^{(p)}$. If $x\in \ms V^{(p)}_j$,
$\mu_n(x) \,=\, \alpha_n\, f_n(x) \, \pi_n(x) \,=\, \alpha_n\,
\omega_j \, \pi_n(x)/\pi_n(\ms V^{(p)}_j)$. Since $\alpha_n \to 1$, by
Corollary \ref{l34}, the previous expression converges to
$\omega_j \, \pi^{(p)}_j(x) \,=\, \mu(x)$.

To complete the proof of the $\Gamma-\limsup$, it remains to show that
$\limsup_n \theta^{(p)}_n\, \ms I_n(\mu_n) \le \ms I^{(p)} (\mu)$. By
\eqref{f6} and the definition of $\mu_n$,
\begin{equation*}
\ms I_n(\mu_n) \;=\; \alpha_n\,
\< \, h_n \, , \, (-\ms L_n)\, h_n \,\>_{\pi_n} \;.
\end{equation*}
By Corollary \ref{l01} and the definition of $h_n$, the right-hand
side is equal to
\begin{equation*}
\begin{aligned}
& \alpha_n\, \pi_n(\ms V^{(p)})\,
\< \sqrt{f_n} \, , (-\ms L^{(p)}_n) \sqrt{f_n}\>_{\pi^{(p)}_n} \\
&\quad =\;
-\, \alpha_n\,
\sum_{x, y\in \ms V^{(p)}} \pi_n(x) \, R^{(p)}_n (x,y)\,
\sqrt{f_n (x) } \,\big\{\,
\sqrt{f_n (y) } \,-\, \sqrt{f_n (x) }\, \big\} \;,
\end{aligned}
\end{equation*}
where $\color{blue} \ms L^{(p)}_n$ stands for the generator of the
trace process $Y^{n,p}_t$ introduced in \eqref{40}, and
$\color{blue} \pi^{(p)}_{n}$ for the measure $\pi_n$ conditioned to
$\ms V^{(p)}$. Since $f_n$ is constant and equal to $\omega_j(n)$ on
each set $\ms V^{(p)}_j$, the previous expression is equal to
\begin{equation*}
\begin{aligned}
& -\, \alpha_n\, \sum_{j\in S_p} \sqrt{\omega_j (n) }
\sum_{k\in S_p \setminus \{j\}} \,\big\{\,
\sqrt{\omega_k (n) } \,-\, \sqrt{\omega_j (n) }\, \big\} 
\sum_{x\in \ms V^{(p)}_j} \pi_n(x) \,
\sum_{y\in \ms V^{(p)}_k} R^{(p)}_n (x,y)  \\
&\quad =\; -\, \alpha_n\, \sum_{j\in S_p} \sqrt{\omega_j (n) }
\sum_{k\in S_p \setminus \{j\}} \,\big\{\,
\sqrt{\omega_k (n) } \,-\, \sqrt{\omega_j (n) }\, \big\}
\pi_n(\ms V^{(p)}_j)\, r^{(p)}_n (j,k) \;,
\end{aligned}
\end{equation*}
where $r^{(p)}_n (j,k)$ is defined in \eqref{41}. Up to this point, we
proved that
\begin{equation*}
\theta^{(p)}_n\, \ms I_n(\mu_n)  \;=\; 
\frac{\alpha_n}{2} \, \theta^{(p)}_n\, \sum_{j, k\in S_p}
\pi_n(\ms V^{(p)}_j)\, r^{(p)}_n (j,k) 
\,\Big\{\,\sqrt{\omega_k (n) } \,-\, \sqrt{\omega_j (n) }\, \Big\}^2 
\;,
\end{equation*}
where we used that
$\pi_n(\ms V^{(p)}_j)\, r^{(p)}_n (j,k) \,=\, \pi_n(\ms V^{(p)}_k)\,
r^{(p)}_n (k,j)$, an identity which follows from the reversibility
assumption. 

Recall that $\alpha_n\to 1$. In view of the definition of
$\omega_i (n)$, it remains to examine the asymptotic behavior of
\begin{equation}
\label{76}
\pi_n(\ms V^{(p)}_j)\, \theta^{(p)}_n\, r^{(p)}_n (j,k) 
\,\Big\{\,\sqrt{\frac{\omega_k}{\pi_n(\ms V^{(p)}_k)} }
\,-\, \sqrt{\frac{\omega_j}{\pi_n(\ms V^{(p)}_j)}}\, \Big\}^2 \;.
\end{equation}

As in the proof of Proposition \ref{p01b}, we divide the pairs $(j,k)$
in three types. Assume first that
$\theta^{(p)}_n \, r^{(p)}_n (j,k) \to 0$ and
$\theta^{(p)}_n \, r^{(p)}_n (k,j) \to 0$. By \cite[Lemma 3.1]{lx16},
and \eqref{58}, either $\pi_n(\ms V^{(p)}_j)/\pi_n(\ms V^{(p)}_k)$
converges to a nonnegative real number or so does
$\pi_n(\ms V^{(p)}_k)/\pi_n(\ms V^{(p)}_j)$. Assume that
$\pi_n(\ms V^{(p)}_k)/\pi_n(\ms V^{(p)}_j) \to a \in [0,\infty)$. In
this case, by reversibility, \eqref{76} is equal to
\begin{equation}
\label{77}
\theta^{(p)}_n \, r^{(p)}_n (k,j) 
\,\Big\{\,\sqrt{\omega_k}
\,-\, \sqrt{\omega_j\, \frac{\pi_n(\ms V^{(p)}_k)}
{\pi_n(\ms V^{(p)}_j)}}\, \Big\}^2
\;,
\end{equation}
which vanishes as $n\to\infty$.

Next, suppose that $\theta^{(p)}_n \, r^{(p)}_n (j,k) \to 0$ and
$\theta^{(p)}_n \, r^{(p)}_n (k,j) \to r^{(p)} (k,j) > 0$, where
$r^{(p)} (k,j)$ has been introduced in \eqref{34}. In particular, $k$
is a transient state of the chain $\bb X^{(p)}_t$.  By reversibility,
$\pi_n(\ms V^{(p)}_k)/\pi_n(\ms V^{(p)}_j) \,=\, r^{(p)}_n
(j,k)/r^{(p)}_n (k,j) \to 0$. Hence, \eqref{76}, which is equal to
\eqref{77}, converges to
\begin{equation*}
r^{(p)} (k,j) \,\omega_k\;.
\end{equation*}

Finally, suppose that
$\theta^{(p)}_n \, r^{(p)}_n (j,k) \to r^{(p)} (j,k) > 0$ and
$\theta^{(p)}_n \, r^{(p)}_n (k,j) \to r^{(p)} (k,j) > 0$. This means
that $j$ and $k$ belong to some closed irreducible class
$\mf R^{(p)}_m$ of the chain $\bb X^{(p)}_t$.  By Lemma \ref{l33}, the
expression \eqref{76} converges to
\begin{equation*}
M^{(p)}_m(j)\, r^{(p)} (k,j) \,\Big\{\,
\sqrt{\frac{\omega_k} {M^{(p)}_m(k)}}
\,-\, \sqrt{\frac{\omega_j} {M^{(p)}_m(j)}}\, \Big\}^2\;.
\end{equation*}
Combining the previous estimates yields that
$\theta^{(p)}_n \ms I_n (\mu_n)$ converges to $\ms J^{(p)} (\mu)$,
which completes the proof of the $\Gamma-\limsup$ in view of Lemma
\ref{l36}.

We turn to the $\Gamma-\liminf$ where we use an induction
argument. Fix $1\le p\le \mf q$ and assume that the
$\Gamma$-convergence of $\theta^{(p-1)}_n\, \ms I_n$ to
$\ms I^{(p-1)}$ has been proved. Fix a probability measure $\mu$ on
$V$ and a sequence $\mu_n$ converging to $\mu$.

Suppose that $\ms I^{(p-1)} (\mu)>0$. In this case, since
$\theta^{(p-1)}_n\, \ms I_n$ $\Gamma$-converges to $\ms I^{(p-1)}$ and
$\theta^{(p)}_n/\theta^{(p-1)}_n \to\infty$,
\begin{equation*}
\liminf_{n\to\infty} \theta^{(p)}_n\, \ms I_n(\mu_n) \;=\;
\liminf_{n\to\infty} \frac{\theta^{(p)}_n}{\theta^{(p-1)}_n}
\, \theta^{(p-1)}_n\, \ms I_n(\mu_n) \;\ge\;
\ms I^{(p-1)} (\mu)\,
\lim_{n\to\infty} \frac{\theta^{(p)}_n}{\theta^{(p-1)}_n}
\;=\; \infty\;.
\end{equation*}
On the other hand, by \eqref{82}, $\ms I^{(p-1)} (\mu) = \infty$. This
proves the $\Gamma-\liminf$ convergence for measures $\mu$ such that
$\ms I^{(p-1)} (\mu)>0$.

Assume that $\ms I^{(p-1)} (\mu)=0$. By \eqref{83}, there exists a
probability measure $\omega$ on $S_p$ such that $\mu = \sum_{j\in S_p}
\omega_j\, \pi^{(p)}_j$. By \eqref{f4},
\begin{equation*}
\ms I_n(\mu_n) \; \ge \; -\, \int_V \frac{\ms L_n u}{u}\, d\mu_n
\end{equation*}
for all $u: V \to (0,\infty)$.

Fix a function $h: \ms V^{(p)} \to (0,\infty)$ which is constant on
each $\ms V^{(p)}_j$, $j\in S_p$:
$h = \sum_{j\in S_p} \mb h(j) \, \chi_{\ms V^{(p)}_j}$. Let
$u_n\colon V\to \bb R$ be the solution of the Poisson equation
\eqref{21} with $\ms L = \ms L_n$, $V_0 = \ms V^{(p)}$ and $g = h$. By
the representation \eqref{71}, it is clear that
$u_n(x) \in (0,\infty)$ for all $x\in V$.

Since $u_n$ is harmonic on $V \setminus \ms V^{(p)}$ and $u_n = h$ on
$\ms V^{(p)}$, by Lemma \ref{l19}, the right-hand side of the previous
displayed equation with $u=u_n$ is equal to
\begin{equation*}
-\, \int_{\ms V^{(p)}} \frac{\ms L_n u_n}{u_n}\, d\mu_n
\;=\; -\, \int_{\ms V^{(p)}} \frac{\ms L_n u_n}{h}\, d\mu_n
\;=\; -\, \int_{\ms V^{(p)}} \frac{\ms L^{(p)}_n h}{h}\, d\mu_n \;.
\end{equation*}
Here, as in the first part of the proof, $\ms L^{(p)}_n$ stands for
the generator of the trace process $Y^{n,p}_t$ introduced in
\eqref{40}.

Since $h$ is constant on each set $\ms V^{(p)}_j$ (and equal to $\mb
h(j)$), the last integral is equal to
\begin{equation*}
-\, \sum_{j,k\in S_p} \frac{[\, \mb h(k) - \mb h(j)\,]}{\mb h(j)}\,
\sum_{x\in \ms V^{(p)}_j} \pi_n(x)\, \frac{\mu_n (x)}{\pi_n(x)}\,
R^{(p)}_n(x, \ms V^{(p)}_k)\;,
\end{equation*}
where
$R^{(p)}_n(x, \ms V^{(p)}_k) = \sum_{y\in \ms V^{(p)}_k} R^{(p)}_n(x,
y)$. By Proposition \ref{p3},
$\pi_n(x)/\pi_n(\ms V^{(p)}_j) \to \pi^{(p)}_j(x)$ for all
$x\in \ms V^{(p)}_j$. Thus, since
$\mu_n \to \mu = \sum_{j\in S_p} \omega_j\, \pi^{(p)}_j$,
\begin{equation*}
\lim_{n\to \infty} \pi_n(\ms V^{(p)}_j)\, \frac{\mu_n (x)}{\pi_n(x)}
\;=\; \omega_j \quad\text{for all $x\in \ms V^{(p)}_j$} \;.
\end{equation*}
Therefore, by \eqref{34}, as $n\to\infty$, the penultimate expression
multiplied by $\theta^{(p)}_n $ converges to 
\begin{equation*}
-\, \sum_{j\in S_p} \omega_j \, \frac{1}{\mb h(j)}\,
\sum_{k\in S_p} r^{(p)}(j,k) \, [\, \mb h(k) - \mb h(j)\,]
\;=\; -\, \sum_{j\in S_p} \omega_j \, \frac{\bb L^{(p)} \mb h}{\mb h} \;.
\end{equation*}

Summarising, we proved that
\begin{equation*}
\liminf_{n\to\infty} \theta^{(p)}_n\, \ms I_n(\mu_n) \; \ge \;
\sup_{\mb h} \,  -\, \sum_{j\in S_p} \omega_j \, \frac{\bb L^{(p)} \mb
h}{\mb h} \;,
\end{equation*}
where the supremum is carried over all functions
$\mb h: S_p \to (0,\infty)$. By \eqref{83b}, the right-hand side is
precisely $\ms I^{(p)}(\mu)$, which completes the proof of the
$\Gamma-\liminf$.
\end{proof}

\appendix

\section{Potential theory}
\label{sec8}

We present in this section some results on potential theory used in
the article. We do not assume reversibility. We keep the same notation
of the article, removing the index $n$. In particular, $X_t$ is a
$V$-valued, continuous-time irreducible Markov process whose jump
rates are represented by $R(x,y)$. Denote by
$\color{blue} (\mc F_t : t\ge 0)$ the canonical filtration induced by
the chain $X_t$. Hence, $\mc F_t$ is the $\sigma$-algebra generated by
the variables $X_s$, $0\le s\le t$.

We first recall for the reader's convenience the definition
of the trace of a process on a subset.

\subsection*{Trace process}

Fix a non-empty subset $W$ of $V$.  Denote by $T^{W}(t)$ the
total time the process $X_t$ spends in $W$ in the
time-interval $[0,t]$:
\begin{equation*}
T^{W}(t)\;=\;\int_{0}^{t}\,\chi_{W}(X_s)\,ds\;,
\end{equation*}
where, recall, $\chi_{W}$ represents the indicator function of the set
$W$. Denote by $S^{W}(t)$ the generalized inverse of $T^{W}(t)$:
\begin{equation*}
S^{W}(t)\;=\;\sup\{\,s\ge0\,:\,T^{W}(s)\le t\,\}\;.
\end{equation*}

The trace of $X_t$ on $W$, denoted by
$\color{blue}(X^{W}_t : t \ge 0)$, is defined by
\begin{equation}
\label{100}
X^{W}_t\;=\; X_{S^{W}(t)} \;;\;\;\;t\ge0\;.
\end{equation}
By Propositions 6.1 and 6.3 in \cite{bl2}, the trace process is an
irreducible, $W$-valued continuous-time Markov chain, obtained by
turning off the clock when the process $X_t$ visits the set $W^{c}$,
that is, by deleting all excursions to $W^{c}$. For this reason, it is
called the trace process of $X_t$ on $W$.

Denote by $\color{blue} \ms L_W$, $\color{blue} R_W$,
$\color{blue} \lambda_W$, $\color{blue} p_W$ and $\color{blue} \pi_W$
its generator, jump rates, holding times, transition matrix and
stationary state, respectively.  The measure $\pi_W$ is obtained by
conditioning $\pi$ to $W$: $\pi_W(x) = \pi(x)/\pi(W)$.

Let $\color{blue} \mb P^W_{\! x}$, $x\in W$, be the probability
measure on the path space $D(\bb R_+, W)$ induced by the Markov chain
$X^W_t$ starting from $x$. Expectation with respect to $\mb P^W_x$
is represented by $\color{blue} \mb E^W_x$.

\subsection*{Poisson equation}

Fix a non-empty proper subset $V_0$ of $V$ and a function
$g:V_0 \to \bb R$. Let $f$ be the solution of the Poisson equation
\begin{equation}
\label{21}
\left\{
\begin{aligned}
& \ms L f \,=\, 0\;, \quad V\setminus V_0\;, \\
& f \,=\, g\;, \quad V_0\;.
\end{aligned}
\right.
\end{equation}

Recall from \eqref{201} the definition of the hitting and return times
to a subset $\ms A$ of $V$. By the strong Markov property, the
solution of the Poisson equation can be represented as
\begin{equation}
\label{71}
f(x) \,=\, \; \mb E_x[\, g(X_{H_{V_0}})\,]\;, \quad x\in V \;.
\end{equation}

Fix $V_0 \subset W \subset V$ and denote by $f_W$ the solution of
the Poisson equation
\begin{equation}
\label{21b}
\left\{
\begin{aligned}
& \ms L_W f \,=\, 0\;, \quad W\setminus V_0\;, \\
& f \,=\, g\;, \quad V_0\;.
\end{aligned}
\right.
\end{equation}
Mind that $W$ may be equal to $V_0$.

Starting from $x\not\in V_0$, the processes $X_t$ and $X^W_t$ hit the
set $V_0$ at the same point: $X^W_{H_{V_0}(X^W)} = X_{H_{V_0}}$,
$\mb P_{\! x}$ a.s. In this formula and below,
$\color{blue} H_{V_0}(X^W)$, $\color{blue} H^+_{V_0}(X^W)$ stand for
hitting and return time to $V_0$ for the process $X^W$. By the
representation \eqref{71} and the previous observation, for $x\in W$
\begin{equation}
\label{73}
f_W(x) \;=\; \mb E^W_x[\, g(X_{H_{V_0}})\,]\; =\;
\mb E_x[\, g(X^W_{H_{V_0}(X^W)})\,] \; =\;
\mb E_x[\, g(X_{H_{V_0}})\,] \; =\; f(x)\;.
\end{equation}

\begin{lemma}
\label{l19}
Fix $V_0 \subset W \subset V$.  Denote by $f$, $f_W$ the solutions of
the Poisson equations \eqref{21}, \eqref{21b}, respectively.  Then,
\begin{equation*}
(\ms L f)(x) \;=\; (\ms L_W f_W)(x) \;, \quad x\in V_0\;.
\end{equation*}
\end{lemma}

\begin{proof}
Fix $x\in V_0$.  The left-hand side of the identity appearing in the
statement of the lemma can be written as
\begin{equation*}
\lambda(x) \,
\sum_{y\in V} p(x,y) \, [\, f(y) \,-\, f(x)\,]\;.
\end{equation*}
Without loss of generality, assume that $p(z,z)=0$ for all $z\in V$
(if this is not the case, one redefines the holding time $\lambda (z)$
for the identity to hold).  By \eqref{71},
$f(y) = \mb E_y[\, g(X_{H_{V_0}})\,]$ for all $y\in V$, and by the
strong Markov property
$\mb E_x[\, g(X_{H^+_{V_0}})\,] = \sum_{y\in V} p(x,y) \,
f(y)$. Hence, the previous sum can be written as
\begin{equation}
\label{72}
\lambda(x) \,
\, \big \{\,  \mb E_x[\, g(X_{H^+_{V_0}})\,] \,-\, f(x)\,\big\}\;.
\end{equation}
 
Recall that we denote by $X^W_t$ the trace of the process $X_t$ on
$W$. We consider two cases. If $H^+_{W} < H^+_x$ then the process
$X_t$ and $X^W_t$ return to $V_0$ at the same point $X_{H^+_{V_0}}$
(to prove this assertion, consider separately the two situations
$\{H^+_{V_0} < H^+_x\}$ and $\{H^+_{V_0} = H^+_x\}$). Thus, if
$H^+_{W} < H^+_x$ we may replace in \eqref{72} $g(X_{H^+_{V_0}})$ by
$g(X^W_{H^+_{V_0}(X^W)})$.

If $H^+_W = H^+_x$, the process $X_t$ returns to $V_0$ (and also to
$W$) at $x$. In contrast, in the time interval $[0, H^+_W]$ the trace
process on $W$ remains at $x$, and $X^W_t$ may return to $V_0$ at a
point $y \neq x$.  In particular, $X_t$ and $X^W_t$ may return to $V$
at different points. In this case, since $H^+_{V_0} = H^+_x$, we have
\begin{equation*}
\mb E_x\big[\, g(X_{H^+_{V_0}}) \, \chi(H^+_x = H^+_W)
\,\big]
\;=\; g(x)\, 
\mb P_x\big[\, H^+_x = H^+_W \,\big]\;.
\end{equation*}

Up to this point, we proved that for $x\in V_0$,
\begin{equation*}
\mb E_x\big[\, g(X_{H^+_{V_0}}) \,\big] \;=\; 
\mb E_x\big[\, g(X^W_{H^+_{V_0}(X^W)})  \, \chi_{_{A^c}} \,\big]
\;+\; g(x)\, 
\mb P_x[\, A \,]\;,
\end{equation*}
where $A$ is the event $\{H^+_x = H^+_W\}$. Write
$\chi_{_{A^c}}$ as $1 - \chi_{_{A}}$. On the event $A$,
$H^+_{V_0}(X^W) = H^+_{V_0} + H^+_{V_0}(X^W) \circ
\vartheta_{H^+_{V_0}}$. Hence, conditioning on $\mc F_{H^+_{V_0}}$,
since $A$ is $\mc F_{H^+_{V_0}}$-measurable and $X(H^+_{V_0}) = x$ on
the event $A$, by the strong Markov property,
\begin{equation*}
\mb E_x\big[\, g(X^W_{H^+_{V_0}(X^W)})  \, \chi_{_{A}} \,\big] \;=\;
\mb P_x[\, A \,] \, \mb E_x\big[\, g(X^W_{H^+_{V_0}(X^W)})   \,\big] \;.
\end{equation*}
Therefore, for $x\in V_0$,
\begin{equation}
\label{74}
\mb E_x\big[\, g(X_{H^+_{V_0}}) \,\big]  \; -\; g(x) \;=\;
\mb P_x[\, H^+_W \,<\, H^+_x \,] \; \Big\{\,
\mb E_x\big[\, g(X^W_{H^+_{V_0}(X^W)})   \,\big] \;-\; g(x)
\,\Big\}\;. 
\end{equation}
By equation (6.9) in \cite{bl2}, $\lambda (x) \, \mb P_x[\, H^+_W
\,<\, H^+_x \,]  = \lambda_W(x)$. Therefore, \eqref{72} is
equal to 
\begin{equation*}
\lambda_W(x) \,
\, \big \{\,  \mb E^W_x[\, g(X_{H^+_{V_0}})\,] \,-\, f(x)\,\big\}\;.
\end{equation*}
By the strong Markov property and \eqref{73}, this expression is equal
to
\begin{equation*}
\begin{aligned}
& \lambda_W(x) \, \sum_{y\in W} p_W(x,y)
\, \big \{\,  \mb E^W_y[\, g(X_{H_{V_0}})\,] \,-\, f(x)\,\big\} \\
&\quad \;=\;
\lambda_W(x) \, \sum_{y\in W} p_W(x,y)
\, \big \{\,  f_W(y) \,-\, f_W(x)\,\big\}
\;=\; (\ms L_W f_W)(x) \;,
\end{aligned}
\end{equation*}
as claimed.
\end{proof}

Denote by $D(f)$ the Dirichlet form of a function $f:V\to \bb R$:
\begin{equation*}
D(f) \;:=\; \< \, f \,,\, (-\, \ms L) f \,\>_{\pi}\;.
\end{equation*}

\begin{corollary}
\label{l01}
Fix $V_0 \subset W \subset V$.  Denote by $f$, $f_W$ the solutions of
the Poisson equations \eqref{21}, \eqref{21b}, respectively.  Then,
\begin{equation}
\label{18}
D(f)
\;=\; \pi(W)\, \< \, f_W \,,\, (-\, \ms L_W) f_W \,\>_{\pi_W} \;.
\end{equation}
\end{corollary}

\begin{proof}
By definition of the Dirichlet form,
\begin{equation*}
D(f)
\;=\;  \< \, f \,,\, (-\, \ms L) f \,\>_{\pi}
\;=\; -\, \sum_{x\in V} \pi(x)\, f(x) \, (\ms L f)(x)\;.
\end{equation*}
Since $f$ is harmonic on $V_0^c$, the sum can be restricted to
$V_0$. Hence, the previous expression is equal to
\begin{equation*}
-\, \sum_{x\in V_0} \pi(x)\, f(x) \, (\ms L f)(x)\;.
\end{equation*}
By \eqref{73} and Lemma \ref{l19}, this sum is equal to
\begin{equation*}
-\, \sum_{x\in V_0} \pi(x)\, f_W(x) \, (\ms L_W f_W)(x)\;.
\end{equation*}
Since $f_W$ is $\ms L_W$-harmonic on $W\setminus V_0$, we may extend
the sum to $W$. To complete the proof, it remains to recall that 
$\pi_W(\,\cdot\,) = \pi(\,\cdot\,)/\pi(W)$.
\end{proof}

The same proof yields the following result.

\begin{corollary}
\label{l13}
Fix $V_0 \subset V$, $g:V_0 \to \bb R$, and let $u$ be the solution of
\eqref{21}. Then, 
\begin{equation*}
\int_V \frac{\ms L\, u}{u}\, d\mu \;=\;
\int_{V_0} \frac{\ms L_{V_0}\, g}{g}\, d\mu
\end{equation*}
for all probability measures $\mu$ on $V$.
\end{corollary}

\begin{proof}
Since $u$ is harmonic on $V\setminus V_0$, we may restrict the
integral to $V_0$. By Lemma \ref{l19}, on $V_0$ we may
replace $\ms L u$ by $\ms L_{V_0} u_{V_0}$, where $u_{V_0}$ is the
solution of \eqref{21b} with $W=V_0$. However, as $W=V_0$, the
solution of \eqref{21b} is $u_{V_0}=g$. Hence, $\ms L_{V_0} u_{V_0} =
\ms L_{V_0} g$. As $u=g$ on $V_0$, the proof is complete.
\end{proof}

We turn to an estimate of hitting times. Denote by $\pi_A$, $A\subset
V$, the stationary measure $\pi$ conditioned to $A$
\begin{equation*}
\pi_A(x) \;=\; \frac{\pi(x)}{\pi(A)}\;, \quad x\in A\;.
\end{equation*}
Next result is \cite[Proposition 8.4]{lms1}. It holds for
non-reversible dynamics.  The assertion in the case where $A$ is a
singleton follows from the proofs of \cite[Corollary 4.2]{BL15} and
\cite[Proposition 8.4]{lms1}.

\begin{lemma}
\label{l06}
Let $A$, $B$ be two nonempty disjoint subsets of $E$. Then, for every
probability measure $\nu$ concentrated on the set $A$ and $\varrho >0$
\begin{equation*}
\mathbb{\mathbf{P}}_{\nu}
\big[\, H_{B} \;\le\; \varrho \, \big]^{2} \;\le\;
e^{2} \, {E}_{\pi_A}\Big[\,
\Big(\frac{\nu}{\pi_A}\Big)^{2}\, \Big]\,
\frac{\Cap (A ,\,B)}{\pi (A)} \, \varrho \;. 
\end{equation*}
If $A$ is a singleton, $A=\{x\}$, then for every $\varrho >0$
\begin{equation*}
\mathbb{\mathbf{P}}_{\! x}
\big[\, H_{B} \;\le\; \varrho \, \big] \;\le\;
e \, \frac{\Cap (\{x\} ,\,B)}{\pi (x)} \, \varrho \;. 
\end{equation*}
\end{lemma}

This result helps in showing that the left-hand side vanishes
asymptotically if
$[\, \Cap_n (\{x\} ,\,B) /\pi_n (x)\,] \, \varrho_n \to 0$.

\begin{remark}
For two sets $A$, $B$ satisfying the hypotheses of Lemma \ref{l06},
let $\nu_{A,B}$ be the equilibrium measure on $A$:
\begin{equation*}
\nu_{A,B}(x) \;=\; 
\frac{1}{\Cap (A\, ,\,B)}\, \pi(x)\, \lambda(x)\, 
\mathbb{\mathbf{P}}_{\! x} [\, H_B < H^+_A\,]
\;, \quad x\in A  \;.
\end{equation*}
By Chebychev inequality and \cite[Proposition A.2]{bl7},
\begin{equation*}
\mathbb{\mathbf{P}}_{\! \nu_{A,B}}
\big[\, H_{B} \;\ge \; \varrho \, \big] \;\le\;
\frac{1}{\varrho}\,
E_{\nu_{A,B}} \big[\, H_{B} \, \big]
\;=\; \frac{E_{\pi} [h^*_{A,B}]}{\varrho \, \Cap (A ,\,B)}
\;,
\end{equation*}
where $h^*_{A,B}$ stands for the equilibrium potential of the
time-reversed process (sometimes called the adjoint process):
$h^*_{A,B} (y) = \mb P^*_{\! y}[H_A < H_B]$, and $\mb P^*$ stands for
the distribution of the continuous-time Markov chain with jump rates
$R^*(x,y)$ given by $R^*(x,y) = \pi(y) \, R(y,x)/\pi(x)$. In many
cases, $E_{\nu_{A,B}} [\, H_{B} \, ] \,=\, [1+o(1)] \, \pi(A)$ so that
\begin{equation*}
\mathbb{\mathbf{P}}_{\! \nu_{A,B}}
\big[\, H_{B} \;\ge \; \varrho \, \big] \;\le\;
[1+o(1)] \, \frac{\pi(A) }{\varrho \, \Cap (A ,\,B)}
\;.
\end{equation*}
This inequality demonstrates that the bound in Lemma \ref{l06} is
sharp whenever $E_{\nu_{A,B}} [\, H_{B} \, ] \,=\, [1+o(1)] \, \pi(A)$.
\end{remark}

\subsection*{Acknowledgments}

C. L. has been partially supported by FAPERJ CNE E-26/201.207/2014, by
CNPq Bolsa de Produtividade em Pesquisa PQ 303538/2014-7.

\end{document}